\documentclass[
letterpaper,12pt]
{amsart}
\usepackage{amssymb,latexsym}
\usepackage{mathrsfs}
\usepackage{euscript}
\usepackage{hyperref}
\usepackage{amsmath}

\usepackage[cjk]{kotex}

\newtheorem{theorem}{Theorem}[section]
\newtheorem{definition}[theorem]{Definition}
\newtheorem{lemma}[theorem]{Lemma}
\newtheorem{corollary}[theorem]{Corollary}
\newtheorem{remark}[theorem]{Remark}
\newtheorem{proposition}[theorem]{Proposition}
\newtheorem{notation}[theorem]{Notation}

\def\cal{\mathcal}

\def\R{{\mathbb R}}

\def\<{\left<}
\def\>{\right>}

\def\I{{\cal I}}

\def\BMO{{\rm BMO}}

\def\CMO{{\rm CMO}}

\def\({( \hspace{-0.335em}(}
\def\){) \hspace{-0.335em})}

\def\fu2{\frac{n}{2}}

\def\l({\left(}
\def\r){\right)}
\def\be{\begin{enumerate}}
\def\ee{\end{enumerate}}

\def \diam {{\rm diam}}

\def \dist {{\rm dist}}

\def \rdist {{\rm \, rdist}}
\def \inrdist {{\rm \, inrdist}}
\def \ec {{\rm ec}}
\def \child{{\rm ch}}

\def \Int{{\rm Int}}

\newcommand\smland{\raise.2ex\hbox{$\scriptstyle\land$}\hspace{.01cm}}
\newcommand\smlor{\raise.1ex\hbox{$\scriptstyle\lor$}}

\title[A global $Tb$ Theorem for compactness and boundedness]{A global $Tb$ Theorem\\ for compactness and boundedness}
\author{Paco Villarroya}
\address{Department of Mathematics, University of Georgia, Athens, GA 30602, USA}
\email{paco.villarroya@uga.edu}

\thanks{The author has been partially supported by Spanish Ministry of Economy and Competitiveness (project grants 
MTM2011-23164 and MTM2014-53009-P)}

\subjclass[2010]{42B20, 42B25, 42C40, 47G10}
\keywords{Calder\'on-Zygmund operator, compact operator, accretive function, wavelet systems.} 

\begin{document}

\maketitle

\begin{abstract}
We prove a $Tb$ Theorem that characterizes all Calder\'on-Zygmund operators 
that extend compactly 
on 
$L^{p}(\mathbb R^{n})$, $1<p<\infty $.
The result, whose proof 
does not require the property of accretivity, can be used to prove compactness of the Double Layer Potential operator on a wide class of domains.

The study also provides conditions for boundedness of singular integral operators 
by means of non-accretive testing functions.

\end{abstract}

\section{Introduction}


The seminal $T1$ Theory \cite{DJ}
was soon extended to a $Tb$ Theory
in which boundedness of singular integral operators is tested through their action over functions $b$ more general than the function $1$.
A. McIntosh and I. Meyer \cite{MM} obtained a $Tb$ Theorem in the special case $Tb_{1}=T^{*}b_{2}=0$ and, in an independent work, G. David, J. L. Journ\'e and S. Semmes \cite{DJS} solved the general case. 
Whereas the $T1$ Theorem proved
boundedness of the Cauchy integral over graphs with small Lipschitz constant, the
$Tb$ result established this result in full generality, and also boundedness of the Double Layer Potential operator.
More on the early developments of the theory can be found in \cite{David}, \cite{ke}.

These results generated an
intense flow of research, 
still active nowadays, 
producing
a variety of $Tb$ Theorems 
that apply to different settings:  
from singular integrals on non-homogeneous spaces \cite{NTVTb},
to operators with vector-valued Calder\'on-Zygmund kernels \cite{Hy}
or singular integral operators between weighted spaces \cite{Hy3}. 

The recent paper \cite{V} introduced a $T1$ Theorem to characterize compactness of Calder\'on-Zygmund operators. 
Now, following the classical line of progress, 
we 
present in the current paper a compact $Tb$ result, that is, a
criterion of compactness relying on the action of the operator over testing functions 
$b$ as general as possible (Theorem \ref{Mainresult2}). 

In the classical theory,
the testing functions 
used to check 
boundedness 
satisfy a non-degeneracy property called
accretivity, which essentially implies the existence of 
lower bounds for the testing functions or for their averages (see \cite{David}). 
In the setting of compact operators, we show that the hypothesis of accretivity can be relaxed to a large extend. The reason, speaking quite broadly, is that compact singular integral operators exhibit an extra decay to zero (see \cite{V}).
Then one can use this decay to allow the averages of the testing functions to tend to zero as long as 
their inverses
grow slower than 
the operator extra decay tends to zero. 
As a result,
compactness can be checked over a larger class of testing functions. The class varies with the operator under study: the faster its bounds decay, the larger the class can be.
This allows the existence of global but well localized testing functions and so, it
justifies the development of a global $Tb$ Theorem before studying the corresponding local result.

The main result in the paper is \hyperref[Mainresult2]{Theorem \ref{Mainresult2}}, which proves compactness of the Double Layer Potential operator for 
a large the class of domains (see \cite{ke}). 
Classically, compactness is proved after verifying, by means of $Tb$ Theorem, that the operator is bounded. Since the latter result requires the testing function being accretive, this imposes non necessary hypotheses on the regularity of the boundary of the domain. The new results weaken these hypotheses by allowing the use of non-accretive testing functions. 

Furthermore, since the proof of compactness is based on deeper investigations on boundedness, 
in Corollary \ref{Mainresultrestrictedbddness1}
we extend the classical $Tb$ Theorem to a 
criterium of boundedness which
does not require accretive testing functions. 

In sections 2, 3 we introduce some notation and definitions, while in section 4 we state the main results.
Sections 5, 6 and 7 are devoted to study the auxiliary functions used to characterize compactness,  
develop estimates for the dual pair over functions with adjacent supports, and define $Tb_{1}$ and $T^{*}b_{2}$.
In the following four sections we prove sufficiency of the hypotheses of Theorem \ref{Mainresult2}, leaving their necessity for section \ref{necessity section}.

I express my appreciation to Christoph Thiele and Diogo Oliveira e Silva
for the organization of 
the Summer School '$T(1)$ and $T(b)$ Theorems and Applications' and to all its participants. The meeting was 
a very exciting 
event and 
a great source of inspiration for this project. 
I also thank the support from
\hspace{-.3cm}
\begin{CJK}{UTF8}{mj}
^^ea^^b9^^80^^ec^^9e^^90^^ec^^98^^81
\end{CJK}
\hspace{-.3cm} 
in Sunnyvale, USA, where most of this research was developed. 

\section{Notation and definitions}

%

\subsection{Notation}\label{ecandrdist}
We denote by ${\mathcal C}$, ${\mathcal D}$ the families of cubes $I=\prod_{i=1}^{n}[a_{i},b_{i})$ and dyadic cubes
$I=2^{j}\prod_{i=1}^{n}[k_{i},k_{i}+1)$ for $j,k_{i}\in \mathbb Z$, respectively. Given a measurable set $\Omega \subset \mathbb R^{n}$, we denote by $\mathcal D(\Omega )$ the family of all $I\in {\mathcal D}$ such that $I\subset \Omega$.

For $I\in \mathcal C$, we denote its centre by $c(I)$, its side length by $\ell(I)$ and its volume by $|I|$.
For $\lambda >0$, we denote by $\lambda I$, the unique cube such that $c(\lambda I)=c(I)$ and $\ell(\lambda I)=\lambda \ell(I)$. 
We write $\mathbb B=[-1/2,1/2)^{n}$ and $\mathbb B_{\lambda }=\lambda \mathbb B$.
We denote by $|\cdot |_{\infty }$ the $l^{\infty }$-norm in $\mathbb R^{n}$ and by $|\cdot |$ the modulus of a complex number. 

Given two cubes $I,J\in {\mathcal C}$,
if $\ell(J)\leq \ell(I)$ we denote $I\smland J=J$, $I\smlor J=I$;
while if $\ell(I)<\ell(J)$ we write
$I\smland J=I$, $I\smlor J=J$.
We define $\langle I,J\rangle$ as the unique cube containing $I\cup J$ with the smallest possible side length and 
such that
$\sum_{i=1}^{n}c(I)_{i}$ is minimum. 
We denote its side length by $\diam(I\cup J)$. 
Note 
the equivalence:
\begin{eqnarray*}
\diam(I\cup J) & \approx & \ell(I)/2+|c(I)-c(J)|_{\infty }+\ell(J)/2.
\end{eqnarray*}

We also define the eccentricity and relative distance of $I$ and $J$ as
$$
\ec(I,J)=\frac{\ell(I\smland J)}{\ell(I\smlor J)},
\hskip30pt
\rdist(I,J)=\frac{\ell(\langle I,J\rangle )}{\ell(I\smlor J)}.
$$
Note that
\begin{align*}
\rdist(I,J) & \approx  1+\frac{|c(I)-c(J)|_{\infty }}{\ell(I\smlor J)}
\approx  1+\frac{\dist(I,J)}{\ell(I\smlor J)},
\end{align*}
where $\dist(I,J)$ is the set distance between $I$ and $J$ in the norm $| \cdot |_{\infty }$. 

Given $I\in \mathcal D$, we denote by $\partial I$ the boundary of $I$ and by 
$\child(I)$ the family of dyadic cubes $I'\subset I$ such that $\ell(I')=\ell(I)/2$. 
Given $I\in \mathcal D$, we denote by $I_{p}$ the parent cube of $I$, that is, the only dyadic cube such that 
$I\in \child(I_{p})$. 
We define
the inner boundary of $I$ as $\mathfrak{D}_{I}=\displaystyle{\cup_{I'\in \child(I)}\partial I'}$.

When $J\subset 3I$, we define the inner relative distance of $J$ and $I$ by
$$
\inrdist(I,J)=1+\frac{\dist(J,{\mathfrak D}_{I})}{\ell(J)}.
$$


\subsection{Compact Calder\'on-Zygmund kernel}

\begin{definition} \label{Imdef}
For every $M\in \mathbb N$, let ${\cal C}_{M}$ be the family of cubes in $\mathbb R^{n}$ such that
$2^{-M}\leq \ell(I)\leq 2^{M}$ and
$\rdist(I,\mathbb B_{2^{M}})\leq M$.  We define ${\cal D}_{M}={\mathcal D}\cap {\cal C}_{M}$ and ${\cal D}_{M}(\Omega )={\cal D}(\Omega )\cap {\cal D}_{M}$.

\end{definition}

%

\begin{notation}\label{LSDF}
To study compactness of singular operators, we use three 
bounded functions 
$L,S,D: 
[0,\infty)\rightarrow [0,\infty )$
satisfying 
\begin{equation}\label{limits}
\lim_{x\rightarrow \infty }L(x)=\lim_{x\rightarrow 0}S(x)=\lim_{x\rightarrow \infty }D(x)=0.
\end{equation}
We denote
$F(a, b, c)=L(a)S(b)D(c)$ and $F(a)=F(a,a,a)$. 

By abuse of notation, we write for each cube $I$,
$L(I)=L(\ell(I))$, $S(I)=S(\ell(I))$ and $D(I)=D(\rdist(I,\mathbb B))$. 
Given three cubes $I_{1},I_{2},I_{3}$, 
we define
$
F(I_{1}, I_{2}, I_{3})=L(I_{1})S(I_{2})D(I_{3})
$
and $F(I)=F(I,I,I)$. 

For $\delta >0$, we denote
\begin{align*}
\tilde{L}(I)&=\sum_{k\geq 0}2^{-kn}L(2^{-k}\ell(I)),
\hskip15pt
\tilde{D}(I)=\sum_{k\geq 0}2^{-k\delta }D(\rdist(2^{k}I,\mathbb B)),
\end{align*}
and write 
$\tilde F(I_{1}, I_{2}, I_{3})=\tilde{L}(I_{1})S(I_{2})\tilde{D}(I_{3})$
%
and $\tilde F(I)=\tilde F(I,I,I)$.

\end{notation}
Since the dilation of a function satisfying one of the limits in (\ref{limits}) satisfies the same limit, namely 
${\mathcal D}_{\lambda}L(a)=L(\lambda^{-1}a)$ satisfies the first limit, we often omit universal constants appearing in the argument of these functions.
We note that, by Lebesgue's Dominated Convergence Theorem, 
$\tilde{L}$ and $\tilde{D}$ satisfy the corresponding limits in \eqref{limits}.

\begin{definition}
\label{prodCZoriginal}
A measurable function $K:(\mathbb R^{n}\times \mathbb R^{n}) \setminus 
\{ (t,x)\in \mathbb R^{n}\times \mathbb R^{n} : t=x\} \to \mathbb C$ is a
compact Calder\'on-Zygmund kernel if it is bounded on compact sets of its domain
and
there exist $0<\delta \leq 1$ and functions $L,S,D$ satisfying Definition \ref{LSDF}
such that
\begin{equation}\label{smoothcompactCZ}
|K(t,x)-K(t',x')|
\lesssim
\frac{(|t-t'|_{\infty }+|x-x'|_{\infty })^\delta}{|t-x|_{\infty }^{n+\delta}}
F_{K}(t,x),
\end{equation}
whenever $2(|t-t'|_{\infty }+|x-x'_{\infty }|)<|t-x|_{\infty }$ with
\begin{align*}
F_{K}(t,x)&
=L(|t-x|_{\infty })S(|t-x|_{\infty })D(|t+x|_{\infty }).
\end{align*}
\end{definition}

For technical reasons, we will mostly use the following alternative formulation of a compact Calder\'on-Zygmund kernel: 
\begin{equation}\label{LSD}
|K(t, x)-K(t',x')|
\lesssim
\frac{(|t-t'|_{\infty }+|x-x'|_{\infty })^\delta}{|t-x|_{\infty }^{n+\delta}}F_{K}(t,x,t',x'),
\end{equation}
whenever $2(|t-t'|_{\infty }+|x-x'|_{\infty })<|t-x|_{\infty}$ with $0<\delta <1$ and 
$$F_{K}(t,x,t',x')=L_{1}(|t-x|_{\infty })S_{1}(|t-t'|_{\infty }+|x-x'|_{\infty })D_{1}\Big(1+\frac{|t+x|_{\infty }}{1+|t-x|_{\infty }}\Big),$$ 
where
$L_{1}, S_{1}, D_{1}$ satisfy the limits in
(\ref{limits}). 
As it is explained in \cite{V},
condition \eqref{LSD} can be obtained from 
\eqref{smoothcompactCZ}.

In \cite{V}, we proved in the one-dimensional case that the smoothness condition 
\eqref{smoothcompactCZ}
essentially implies the pointwise decay condition 
\begin{equation}\label{kerneldecay}
|K(t, x)|
\lesssim
\frac{F_{K}(t, x)}{|t-x|_{\infty }^{n}}
\end{equation}
with $F_{K}(t,x)\! =\! L(|t-x|_{\infty })S(|t-x|_{\infty })D(1+\frac{|t+x|_{\infty }}{1+|t-x|_{\infty }})$.

\subsection{Operator with a compact Calder\'on-Zygmund kernel}

%


\begin{definition}\label{intrep}
Let $T$ be a linear operator bounded on $L^{2}(\mathbb R^{n})$. 
Let  $b_{1}, b_{2}$ be locally integrable functions. 


$T$ is associated with a compact Calder\'on-Zygmund kernel if there exists a function $K$ 
satisfying Definition \ref{prodCZoriginal} such that 
for all $f,g$ 
with disjoint compact supports, 
the following
integral representation holds:
\begin{equation}\label{kernelrep}
\langle T(b_{1}f), b_{2}g\rangle =\int_{\R^{n}}\int_{\R^{n}} f(t)g(x) K(t,x)b_{1}(t)b_{2}(x)\, dt \, dx.
\end{equation}
\end{definition}

Boundedness of the operator is assumed to provide the integral representation, but we will only use this hypothesis qualitatively. We will work to obtain bounds that only depend on the implicit constant of the compact Calder\'on-Zygmund kernel and the conditions of next section: the weak compactness condition and the $\BMO$ norm of $Tb_{1}$, $T^{*}b_{2}$.


\begin{notation}

Given an operator $T$ and $b_{1}, b_{2}$ measurable functions,  
we write 
$T_{b}=M_{b_{2}}^{*}\circ T \circ M_{b_{1}}$, 
with $M_{b_{i}}(f)=b_{i}f$ the pointwise multiplication operator.
\end{notation}

\section{The weak compactness and the cancellation conditions}
We introduce the hypotheses for compactness of singular integral operators:
the weak compactness condition and the membership of $Tb_{1}, T^{*}b_{2}$ 
to $\CMO_{b}(\mathbb R^{n})$.

\begin{definition}
Given  $b$ a 
locally integrable function from $\mathbb R^{n}$ to $\mathbb C$ and $1\leq q\leq \infty $, 
we denote for every 
cube $I\in \mathcal D$
$$
\langle b\rangle_{I}=\frac{1}{|I|}\int_{I}b(x)dx,
\hskip30pt
[b]_{I,q}
=\Big(\frac{1}{|I|}\int_{I}|b(x)|^{q}dx\Big)^{\frac{1}{q}}.
$$
Then the maximal function can be written as $\displaystyle M_{q}b(x)=\sup_{x\in I\in \mathcal C}[b]_{I,q}$.
\end{definition}

\subsection{The weak compactness condition}
\begin{definition}\label{weakcompactnessdef}
Let $1\leq q_{1},q_{2}\leq \infty $, and
$b_{1}, b_{2}$ be locally integrable functions. 
A linear operator $T$ 
satisfies the weak compactness condition
if there exists a bounded function $F_{W}$ satisfying Definition \ref{LSDF}
such that for every $\epsilon>0$ there
exists $M_{0}\in \mathbb N$ so that 
\begin{equation}\label{restrictcompact2}
|\langle T(b_{1}^{\alpha }\chi_I),b_{2}^{\beta}\chi_{I}\rangle |
\lesssim 
|I|[b_{1}^{\alpha }]_{I,q_{1}}[b_{2}^{\beta}]_{I,q_{2}}(F_{W}(I;M)
+\epsilon )
\end{equation}
for all $I\in {\mathcal D}$, $M>M_{0}$ and $\alpha, \beta \in \{0,1\}$. We note that $b_{i}^{0}\equiv 1$, $b_{i}^{1}\equiv b_{i}$.

We say
that $T$ satisfies the weak boundedness condition
if \eqref{restrictcompact2} holds with a function $F_{W}$ for which some of the limits in (\ref{limits}) may not hold. 
\end{definition}

Due to the presence of the exponents $\alpha $, $\beta$, 
this definition is more restrictive than the classical concept of weak boundedness. But 
for most operators the same calculations used to check the standard inequality of weak boundedness (or compactness) suffice to establish \eqref{restrictcompact2}. 
The need for this particular formulation originates in Lemma \ref{orthorep}. 

The factor $F_{W}(I;M)+\epsilon $ is justified by the result in Proposition \ref{neceweakcompactnessfinal}. In \cite{V}, some other alternative definitions of this property are discussed.

\subsection{Characterization of compactness}

\begin{definition}\label{lagom}
Let $E$ be a Banach space of functions with domain in $\mathbb R^{n}$. 
Let $(\psi_{I})_{I\in {\mathcal D}}$ be a wavelet system in $E$ and 
$(\tilde{\psi}_{I})_{I}$ be the dual system. 
Then for $M\in \mathbb N$ 
we define the lagom projection operator by
$$
P_{M}f=\sum_{I\in {\cal D}_{M}}\langle f,\tilde{\psi}_{I}\rangle \psi_{I},
$$
where $\langle f,\tilde{\psi}_{I}\rangle =\int_{\mathbb R^{d}}f(x)\tilde{\psi}_{I}(x)dx$.
Note that 
$\displaystyle
P_{M}^{*}f=\sum_{I\in {\cal D}_{M}}\langle f,\psi_{I}\rangle \tilde{\psi}_{I}
$.

\vspace{-.3cm}
We also define $P_{M}^{\perp }f=f-P_{M}f$.
\end{definition}


\begin{remark}\label{actionproduct}
The unusual definition of $\langle ,\rangle $
in Definition \ref{lagom} is a customary license in the literature on $Tb$ theorems.
\end{remark}

To prove compactness of an operator on $L^{2}(\mathbb R^{n})$ it is enough to show that 
$
\langle {P_{M}^{*}}^{\perp}TP_{M}^{\perp}f,g\rangle 
$
tends to zero uniformly for all $f,g$ 
in the unit ball of 
$L^{2}(\mathbb R^{n})$. 
The reason for this is the following decomposition:
$$
Tf
=P_{M}^{*}Tf+{P_{M}^{*}}^{\perp}TP_{M}f+{P_{M}^{*}}^{\perp}TP_{M}^{\perp}f.
$$
The first term is a finite rank operator and thus, compact on $L^{2}(\mathbb R^{n})$. 
The adjoint of the second term, that is $P_{M}^{*}T^{*}P_{M}^{\perp}$, is of finite rank and so, the second term is also compact on $L^{2}(\mathbb R^{n})$. Therefore, we only need to prove that the operator norm of the third term tends to zero. 

\subsection{The cancellation condition. The spaces $\CMO_{b}(\mathbb R^{n})$ and
$H^{1}_{b}(\mathbb R^{n})$} 
We now provide the definition of the space to which the functions $Tb_{1}, T^{*}b_{2}$ must belong when $T$ is compact.

\begin{notation}\label{CBI} 
A locally integrable function $b$ has non-zero dyadic averages if $\langle b\rangle_{I}\neq 0$ for all 
$I\in \mathcal D$.
Then, for $I\in \mathcal D$ and $1\leq q\leq \infty$, we write
$C_{I}^{b}=\frac{1}{|\langle b\rangle_{I}|}+\frac{1}{|\langle b\rangle_{I_{p}}|}$
and
$B_{I,q}^{b}=\frac{ [b]_{I,q}}{|\langle b\rangle _{I}|}+\frac{[b]_{I_{p},q}}{|\langle b\rangle_{I_{p}}|}$, where 
$I\in \child(I_{p})$.
\end{notation}

\begin{definition}\label{CMObdef}
Let  $b_{1},b_{2}$ be locally integrable functions with nonzero dyadic averages. Let $(\psi_{I}^{b_{2}})_{I}$ be the wavelet system of Definition \ref{defpsi}.
We define 
${\BMO}_{b}(\mathbb R^{n})$ as the space of locally integrable functions $f$ such that 
\begin{equation}\label{BMOb2}
\sup_{I\in \mathcal I}\Big(\frac{1}{|I|}
\sum_{J\in {\cal D}(I)}
|\alpha_{I,J}|^{2}|\langle f,\tilde\psi_{J}^{b_{2}}\rangle |^{2}\Big)^{1/2}<\infty 
\end{equation}
with $\alpha_{I,J}=
(B_{I,q_{2}}^{b_{2}})^{2}
\Big(1+\frac{1}{|\langle b_{1}\rangle_{J}|}\Big)\frac{[b_{2}]_{J,2}}{|\langle b_{2}\rangle_{J}|}$ 
and similarly changing the roles of $b_{1}$ and $b_{2}$. 

We define $\CMO_{b}(\mathbb R^{n})$ as the closure in $\BMO_{b}(\mathbb R^{n})$ of the space of continuous functions vanishing at infinity.
\end{definition}


Since $\alpha_{I,J}\gtrsim 1$, we have that  $\BMO_{b}(\mathbb R^{n})\subseteq \BMO(\mathbb R^{n})$. When $b_{i}$ are bounded and accretive, 
we have  $\BMO_{b}(\mathbb R^{n})=\BMO(\mathbb R^{n})$
and $\CMO_{b}(\mathbb R^{n})=\CMO(\mathbb R^{n})$.
Next lemma gives a characterization of $\CMO_{b}(\mathbb R^{n})$ in terms of 
wavelet decompositions:


\begin{lemma} \label{lem:cmochar} The following statements are equivalent:
\begin{enumerate}
\item[i)] $f\in \CMO_{b}(\mathbb R^{n})$,

\item[ii)] $f\!\in \! \BMO_{b}(\mathbb R^{n})$ with
${\displaystyle 
\lim_{M\rightarrow \infty} 
\sup_{I\in \mathcal I}\Big(\frac{1}{|I|}
\hspace{-.1cm} \sum_{J\in {\cal D}_{M}(I)^{c}}
\hspace{-.3cm}|\alpha_{I,J}|^{2}
|\langle f,\tilde\psi_{J}^{b_{2}}\rangle |^{2}\Big)^{1/2}\hspace{-.4cm}=0}$,
and similarly changing the roles of $b_{1}$ and $b_{2}$.


\end{enumerate}
\end{lemma}

\begin{definition}\label{defH1dual}
A locally integrable function $a$ is called a $p$-atom with respect $b$ and $I\in \mathcal D$ 
if it is compactly supported on $I_{p}$, 
it has mean zero with respect to $b$ 
and
$\| a_{I}b\|_{L^{p}(\mathbb R^{n})}\leq 
B_{I,p}^{b}
|I|^{-1/p'}$. 

Then $H^{1}_{b}(\mathbb R^{n})$ is the space of functions $f=\sum_{I\in \mathcal D}\lambda_{I}a_{I}$ where 
$\lambda_{I}\in \mathbb C$ with
$\sum_{I\in \mathcal D}B_{I,p}^{b}
|\lambda_{I}|<\infty $ and 
$a_{I}$ is a $p$-atom with respect to $b$ and $I$. 

We denote by $\| f\|_{H^{1}_{b}(\mathbb R^{n})}$ the infimum of $\sum_{I\in \mathcal D}B_{I,p}^{b}
|\lambda_{I}|$ among all possible atomic decompositions. 

\end{definition}

\subsection{Compatibility instead of accretivity} We define the class of non-accretive testing functions available to characterize compactness.

\begin{notation}
Let $T$ be a linear operator with compact Calder\'on-Zygmund kernel and corresponding function $F_{K}$.
Let $b_{1}, b_{2}$ be locally integrable functions with non-zero dyadic averages such that 
$T$ satisfies the weak compactness condition with function $F_{W}$. 

Let 
$B\!F:\mathcal D\times \mathcal D\rightarrow (0,\infty )$ be defined by $B\! F= B\cdot F$ with 
\begin{align*}
B(I,J)&=C_{I}^{b_{1}}C_{J}^{b_{2}}
\Big(
\langle M_{q_{1}}b_{1}\rangle_{I}
\langle M_{q_{2}}b_{2}\rangle_{J} 
\\
&\hskip55pt 
+\langle M_{q_{1}}(b_{1}\chi_{I})\rangle _{I\smland J}
\langle M_{q_{2}}(b_{2}\chi_{J})\rangle _{I\smland J}\cdot  \chi_{\rdist(I\smland J, I\smlor J)\leq 3}\Big)
\\
F(I, J)&=
\tilde F_{K}(I\smland J, I\smland J,  \langle I,J\rangle )
+F_{W}(I;M_{T,\epsilon }) \chi_{I=J}, 
\end{align*}
where $C_{I}^{b_{1}}$
are defined in Notation \ref{CBI}.
\end{notation}

\begin{definition}[Compatible testing functions]\label{accretive} With previous notation, 
we say that $b_{1},b_{2}$ are testing functions compatible with $T$ if
\begin{equation}\label{b-averageandbound0}
\sup_{\tiny \begin{array}{c}I,J\in {\mathcal D}
\end{array}}
B\!F(I,J)
\Big(\frac{[b_{1}]_{I,q_{1}}}{|\langle b_{1}\rangle_{I}|}\Big)^{2}\Big(\frac{[b_{2}]_{J,q_{2}}}{|\langle b_{2}\rangle_{J}|}\Big)^{2}
<\infty ,
\end{equation}
\begin{equation}\label{b-averageandbound0limit}
\lim_{M\rightarrow \infty }\hspace{-.1cm}\sup_{\tiny \begin{array}{c}I,J
\in {\mathcal D}_{M}^{c}\\ (I,J)\in \mathcal F_{M}
\end{array}}
B\!F(I,J)
\Big(\frac{[b_{1}]_{I,q_{1}}}{|\langle b_{1}\rangle_{I}|}\Big)^{2}\Big(\frac{[b_{2}]_{J,q_{2}}}{|\langle b_{2}\rangle_{J}|}\Big)^{2}
=0,
\end{equation}
with $q_{1}^{-1}+q_{2}^{-1}<1$, and the supremum in \eqref{b-averageandbound0limit} is calculated over the family $\mathcal F_{M}$ of ordered pairs of cubes $I,J\in {\mathcal D}_{M}^{c}$ such that either $\ell(I\smland J)>2^{M}$,
$\ell(I\smland J)<2^{-M}$, or $\rdist(\langle I,J\rangle ,\mathbb B)>M^{\theta}$ for some $\theta \in (0,1)$. 

\end{definition}
Lemmata \ref{smallF} and \ref{Planche} justify the feasibility of Definition \ref{accretive}, in particular, equality \eqref{b-averageandbound0limit}.

\begin{definition}[Compatible testing functions 2]\label{accretive2}
We say that $b_{1},b_{2}$ are testing functions compatible with $T_{b}$ if
\begin{equation}\label{b-averageandbound02}
\sup_{\tiny \begin{array}{c}I,J\in {\mathcal D}
\end{array}}
B\!F(I,J)
\Big(\Big(\frac{[b_{1}]_{I_{p},q_{1}}}{|\langle b_{1}\rangle_{I}||\langle b_{1}\rangle_{I_{p}}|}\Big)^{2}+\Big(\frac{[b_{2}]_{J_{p},q_{2}}}{|\langle b_{2}\rangle_{J}||\langle b_{2}\rangle_{J_{p}}|}\Big)^{2}
\Big)
<\infty ,
\end{equation}
\begin{equation}\label{b-averageandbound0limit2}
\lim_{M\rightarrow \infty }\hspace{-.3cm}
\sup_{\tiny \begin{array}{c}I,J
\in {\mathcal D}_{M}^{c}\\ (I,J)\in \mathcal F_{M}
\end{array}}
\hspace{-.5cm} 
B\!F(I,J)
\Big(\Big(\frac{[b_{1}]_{I_{p},q_{1}}}{|\langle b_{1}\rangle_{I}||\langle b_{1}\rangle_{I_{p}}|}\Big)^{2}\hspace{-.1cm} +\Big(\frac{[b_{2}]_{J_{p},q_{2}}}{|\langle b_{2}\rangle_{J}||\langle b_{2}\rangle_{J_{p}}|}\Big)^{2}
\Big)
\! =\! 0,
\end{equation}
\end{definition}

\section{Statement of the main results}\label{statemainres}

\subsection{Main result on compactness}

\begin{theorem}\label{Mainresult2}
Let $1<p<\infty$ and
$T$ be a linear operator associated with a standard Calder\'on-Zygmund kernel $K$. 

Then $T$ extends compactly on $L^p(\mathbb R^{n})$ if and only if $K$ 
is a compact Calder\'on-Zygmund kernel
and there exist functions $b_{1},b_{2}$ \hyperref[accretive]{compatible} with $T$ so that $T$
satisfies
the weak compactness condition
and 
$Tb_{1}, T^{*}b_{2}\in \CMO_{b}(\mathbb R^{n})$.

If we assume that $b_{1},b_{2}\in L^{\infty }(\mathbb R^{n})$ and that they are compatible with $T_{b}$,
then the same three conditions characterize compactness of $T_{b}$.

\end{theorem}
Given a domain 
$\Omega \subset \mathbb R^{d+1}$ and $S=\partial \Omega $, 
the Double Layer Potential operator $\mathcal K $ is defined as follows
$$
\mathcal K (f)(x)=\lim_{\epsilon \rightarrow 0}\int_{S\cap B(x,\epsilon )^{c}}f(y)
\Big\langle \nu(y), \frac{x-y}{|x-y|^{d}}\Big\rangle d\sigma (y)
$$
where $\nu $ is the exterior unit normal vector to the surface $S$ and $\sigma $ is the surface measure on $S$. 
Then Theorem \ref{Mainresult2} 
proves compactness of $\mathcal K $ 
for a large class of domains $\Omega $.

%
%
%


\subsection{The result on boundedness}
Theorem \ref{Mainresult2} 
can also
characterize bounded operators with a standard Calder\'on-Zygmund kernel, just 
omitting the considerations of limits in \eqref{limits} going to zero. 
For example, we can consider a kernel $K$ satisfying inequality \eqref{smoothcompactCZ} 
with auxiliary function $F_{K}$ defined only by the function $S$, without $L$ and $D$. Despite the associated operator cannot be compact, it might be bounded and, in that case, the testing functions used to check its boundedness do not need to satisfy the accretivity condition for small cubes. In this line, Corollary \ref{Mainresultrestrictedbddness1} describes when boundedness of singular integral operators can be checked by means of non-accretive testing functions. 

The following result holds:
\begin{corollary}\label{Mainresultrestrictedbddness1}
Let $1<p<\infty$ and
$T$ be a continuous linear operator associated with a standard Calder\'on-Zygmund kernel. 

Then $T$ extends boundedly  on $L^p(\mathbb R^{n})$ if and only 
if there exist functions $b_{1},b_{2}$ compatible with $T$ and 
such that $T$
satisfies
the weak boundedness condition and $Tb_{1},T^{*}b_{2}\in \BMO_{b}(\mathbb R^{n})$. 

If $b_{1},b_{2}\in L^{\infty }(\mathbb R^{n})$ and they are compatible with $T_{b}$,
then the same two conditions characterize boundedness of $T_{b}$.
\end{corollary}

Corollary \ref{Mainresultrestrictedbddness1} can be applied to prove boundedness of the double and single layer potential operators associated with boundary value problems for degenerate elliptic equations in divergence form, ${\rm div} (A\nabla u)-V\cdot u = 0$, with appropriate non-negative potentials. In \cite{BEO} it is shown that the Riesz potentials associated with these equations have kernels that decay for large cubes or cubes that are away from the origin, but not for small cubes. This would correspond to the case in which the functions $L$ and $D$ tend to zero, but not the function $S$.

\section{On compact Calder\'on-Zygmund kernels}\label{interlude}
In this section we describe some properties of the auxiliary functions $L$, $S$, $D$, and $F$ of Definition \ref{LSDF}. 

We first note that, without loss of generality, $L$ and $D$ can be assumed to be non-creasing
while $S$ can be assumed to be non-decreasing. These assumptions imply analog properties to $F$ 
and $\tilde F$.
Regarding the equivalent formulation \eqref{LSD} given after Definition \ref{prodCZoriginal}, 
we note that 
in the next lemma and forthcoming results we will often consider the particular case when $t'=t$ and $x'=c(J)$:
$$
F_{K}(t,x,c(J))=L(|t-c(J)|_{\infty })S(|x-c(J)|_{\infty })
D\Big(1+\frac{|t+c(J)|_{\infty }}{1+|t-c(J)|_{\infty }}\Big) .
$$

\begin{lemma}\label{boundtoFK}
For 
$I,J\in \mathcal C$, we write
$\Delta_{I,J}=\{ t\in \mathbb R^{n}: \ell(\langle I,J\rangle)/2<|t-c(J)|_{\infty }\leq \ell(\langle I,J\rangle)\}$.
Then, for $t\in I\cap \Delta_{I,J}$ and $x\in J$, we have
\begin{align*}
F_{K}(t,x,c(J))&\lesssim F_{K}(\langle I,J\rangle, J,\langle I,J\rangle ) 
\end{align*}
with 
$
F_{K}(\langle I,J\rangle, J,\langle I,J\rangle ) 
=L(\ell(\langle I,J\rangle))S(\ell(J))
D(\rdist(\langle I,J\rangle ,\mathbb B))
$.
\end{lemma}
\proof 
Since $L$ is non-creasing and $S$ is non-decreasing, 
$|t-c(J)|_{\infty }>\ell(\langle I,J\rangle )/2$, $|x-c(J)|_{\infty}\leq \ell(J)/2$, 
we only need to bound the factor $D$. 

For all $t\in I$, we have $|t-c(J)|_{\infty }\leq \diam(I\cup J)=\ell(\langle I,J\rangle )$. Then, 
using
$
|c(J)|_{\infty }\leq (|t-c(J)|_{\infty }+|t+c(J)|_{\infty })/2
$, 
we get
\begin{align}\label{boundingD}
1+&\frac{|c(J)|_{\infty }}{1+\ell(\langle I,J\rangle )}
\leq 1+\frac{|c(J)|_{\infty }}{1+|t-c(J)|_{\infty }}
\leq \frac{3}{2}\Big(1+\frac{|t+c(J)|_{\infty }}{1+|t-c(J)|_{\infty }}\Big).
\end{align}

Now,
since $|c(I)|_{\infty }-|c(J)|_{\infty }\leq |c(I)-c(J)|_{\infty }\leq \ell(\langle I,J\rangle )$, we bound below
the numerator in the left hand side of \eqref{boundingD} as follows:
\begin{align*}
1+\ell(\langle I,J\rangle )+|c(J)|_{\infty } 
&\geq 1+\frac{\ell(\langle I,J\rangle ) }{2}+\frac{|c(I)|_{\infty }-|c(J)|_{\infty }}{2}+|c(J)|_{\infty }
\\
&\geq \frac{1}{2}\big(1+\ell(\langle I,J\rangle )+\frac{1}{2}|c(I)+c(J)|_{\infty }\big).
\end{align*}
Therefore, 
\begin{align*}
1+\frac{|c(J)|_{\infty }}{1+\ell(\langle I,J\rangle )}
\geq \frac{1}{3}\Big(\frac{3}{2}+\frac{|c(I)+c(J)|_{\infty }/2}{1+\ell(\langle I,J\rangle )}\Big).
\end{align*}
Now, since $(c(I)+c(J))/2\in \langle I,J\rangle$, we have
$|(c(I)+c(J))/2-c(\langle I,J\rangle)|_{\infty }\leq \ell(\langle I,J\rangle)/2$
and so, we can bound below previous expression by
$$
\frac{1}{3}\Big(\frac{3}{2}+\frac{|c(\langle I,J\rangle )|_{\infty }}{1+\ell(\langle I,J\rangle )}-\frac{1}{2}\Big)
\geq \frac{1}{3}\Big(1+\frac{|c(\langle I,J\rangle )|_{\infty }}{2\max(\ell(\langle I,J\rangle),1)}\Big)
\gtrsim \rdist(\langle I,J\rangle ,\mathbb B ).
$$

Then, omitting constants and using that $D$ is non-creasing, we get
$$
F_{K}(t,x,c(J))\lesssim L(\ell(\langle I,J\rangle ))S(\ell(J))
D(\rdist(\langle I,J\rangle ,\mathbb B )).
$$

\section{Estimates near the diagonal}
In Lemma \ref{adjacentsquares}, we prove a Hardy's inequality for compact operators.
\begin{lemma}\label{adjacentsquarescoro} 
Let $1\leq q_{1},q_{2}\leq \infty $ such that $\frac{1}{q_{1}}+\frac{1}{q_{2}}<1$,
and let $K$ 
be a compact Calder\'on-Zygmund kernel. 
For every $I\in {\mathcal D}$ and every bounded functions $f,g $, 
\begin{align}\label{3IminusI}
|\langle T_{b}(f \chi_{(3I)\backslash I}),g\chi_I\rangle |
\lesssim 
|I|[fb_{1}]_{3I,q_{1}}[gb_{2}]_{I,q_{2}}
\tilde F_{K}(I)
\end{align}
where
$
\tilde F_{K}
$  is given in Definition \ref{LSDF}.
\end{lemma}
Previous lemma follows after proving the following result:
\begin{lemma}\label{adjacentsquares}
With the same hypotheses, let $I,I'\in {\mathcal D}$ be such that $\ell(I)=\ell(I')$ and $\dist(I,I')=0$. 
Let $f,g$ be integrable and compactly supported on $I$ and $I'$ respectively. Then
\begin{align}\label{adja}
|\langle T_{b}f,g\rangle |
\lesssim |I|  [fb_{1}]_{I,q_{1}}[gb_{2}]_{I',q_{2}}
\tilde F_{K}(I).
\end{align}
\end{lemma}

\proof[Proof (of Lemma \ref{adjacentsquarescoro})] 
We first assume $\frac{n}{q_{1}}+\frac{1}{q_{2}}<1$. 
We 
define the kernels $K_{\epsilon }(t,x)=K(t,x)$ if $|t-x|_{\infty }>\epsilon $ and zero otherwise. Then 
\begin{align*}
|\langle T_{b}f,g\rangle |&\leq 
\sup_{m>0}\int_{I'}\int_{I}
|K_{\frac{1}{m}}(t,x)||f(t) b_{1}(t)||g(x) b_{2}(x)|dtdx
=\sup_{m>0}A_{m}.
\end{align*}



Let $\theta \in (0,1)$ such that $S(\theta )\leq \tilde F(I)$. We denote $I_{\theta ,0}=\theta I$, $I_{\theta ,1}=I\backslash I_{\theta }$ and similar for $I'$. 
Then
\begin{align*}
A_{m}
&\leq \sum_{i,j\in \{0,1\}}\int_{I'_{\theta ,j}} \int _{I_{\theta ,i}}
|K_{\frac{1}{m}}(t,x)||f(t)b_{1}(t)||g(x)b_{2}(x)|dtdx
\\
&\leq |I|^{\frac{1}{q_{1}}+\frac{1}{q_{2}}}[f b_{1}]_{I,q_{1}}[g b_{2}]_{I',q_{2}}
\, \,
\sum_{i,j\in \{0,1\}}\| K_{\frac{1}{m}}\chi_{I_{\theta,i}\times I'_{\theta ,j}}\|_{L^{q_{2}'}L^{q_{1}'}(\mathbb R^{2n})}
\end{align*}
From the kernel decay condition \eqref{kerneldecay}, we have 
\begin{equation}\label{IntBR}
\| K_{\frac{1}{m}}\chi_{I_{\theta,i}\times I'_{\theta ,j}}\|_{L^{q_{2}'}L^{q_{1}'}(\mathbb R^{2n})} \lesssim \Big(\int _{I'_{\theta ,j}}\Big(\int_{I_{\theta,i}}\frac{F_{K}(t,x)^{q'_{1}}}{|t-x|_{\infty }^{q'_{1}n}}dt\Big)^{\frac{q'_{2}}{q'_{1}}}dx\Big)^{\frac{1}{q'_{2}}}
\end{equation}
with $F_{K}(t,x)=L(|t-x|_{\infty })S(|t-x|_{\infty })D(1+\frac{|t+x|_{\infty }}{1+|t-x|_{\infty }})$.

When $i=j=1$, 
 we have $
0 <|t-x|_{\infty }\leq \theta$ in the domain of integration and so, 
by the reasoning in the proof of Lemma \ref{boundtoFK}, 
$
F_{K}(t,x)\lesssim
S(\theta )$.
Then the left hand side of \eqref{IntBR} can be bounded by
\begin{equation}\label{last2int}
S(\theta )\Big(\int _{I'}\Big(\int_{I}\frac{1}{|t-x|_{\infty }^{q'_{1}n}}dt\Big)^{\frac{q'_{2}}{q'_{1}}}dx\Big)^{\frac{1}{q'_{2}}}.
\end{equation}

In all remains cases, we have  
$
\theta <|t-x|_{\infty }\lesssim \ell(I)$, 
$|x-c(I)|_{\infty}\lesssim \ell(I)$ and 
$|(t+x)/2-c(I)|_{\infty}\lesssim \ell(I)$  in the domain of integration.
Then, by the proof of Lemma \ref{boundtoFK}, we have
$
F_{K}(t,x)\leq L(|t-x|_{\infty })
S(\ell(I))D(\rdist (I ,\mathbb B))
$.
With this and using H\"olders's inequaltiy for $\tilde q_{i}'=(1+\epsilon )q_{i}'$ with $\epsilon >0$ sufficiently small, the left hand side of \eqref{IntBR} can be bounded by
\begin{equation}\label{last2int2}
\bar{L}(\ell(I))S(\ell(I))D(\rdist (I ,\mathbb B))\Big(\int _{I'}\Big(\int_{I}\frac{1}{|t-x|_{\infty }^{\tilde q'_{1}n}}dt\Big)^{\frac{\tilde q'_{2}}{\tilde q'_{1}}}dx\Big)^{\frac{1}{\tilde q'_{2}}},
\end{equation}
where $$\bar{L}(\ell(I))= \Big(\int _{I'}\Big(\int_{I}L(|t-x|_{\infty })^{q'_{1}n(1+\epsilon^{-1})}dt\Big)^{\frac{q'_{2}}{q'_{1}}}dx\Big)^{\frac{1}{q'_{2}(1+\epsilon^{-1})}}$$ still satisfies the limit properties of \eqref{limits}. 
Now we work to bound the double integral in \eqref{last2int}, being the integral in \eqref{last2int2} very similar.

For
$\dist_{\infty }(x,I)\leq \rho \leq \dist_{\infty }(x,I)+\ell(I)$, 
we denote $J(x,\rho)=\{ t\in I_{\theta }: |t-x|_{\infty }=\rho\}$, which satisfies 
$|J(x,\rho)|\lesssim 2^{n}\rho^{n-1}$. Then, the double integral in \eqref{last2int} can be bounded by
\begin{align}\label{intr1}
\nonumber
\Big(\int _{I'}
&\Big(\int_{\dist_{\infty }(x,I)}^{\dist_{\infty }(x,I)+\ell(I)}
\int_{J(x,\rho)}\frac{1}{\rho^{q'_{1}n}}dud\rho\Big)^{\frac{q'_{2}}{q'_{1}}}dx\Big)^{\frac{1}{q'_{2}}}
\\
&\lesssim \Big(\int _{I'}
\Big(\int_{\dist_{\infty }(x,I)}^{\dist_{\infty }(x,I)+\ell(I)}
\frac{1}{\rho^{(q'_{1}-1)n+1}}d\rho\Big)^{\frac{q'_{2}}{q'_{1}}}dx\Big)^{\frac{1}{q'_{2}}}
\end{align}

We consider first $1< q_{1}<\infty $, for which \eqref{intr1} can be bounded by
$$
\Big(\int _{I'}\dist_{\infty }(x,I)^{-n\frac{q_{2}'}{q_{1}}}
dx\Big)^{\frac{1}{q_{2}'}}.
$$
For $0\leq \rho \leq \ell(I)$, we denote
$J_{\rho}'=
\{ x\in I': \dist_{\infty } (x,I)=\rho \}$, which satisfies 
$|J_{\rho}'|\lesssim \ell(I)^{n-1}$. 
Then, since 
$n\frac{q_{2}'}{q_{1}}<1$, 
we bound the last expression by a constant times
\begin{align*} 
\Big(\ell(I)^{n-1}\int _{0}^{\ell(I)}\rho^{-n\frac{q_{2}'}{q_{1}}}
d\rho\Big)^{\frac{1}{q_{2}'}}
&\lesssim
\Big(
\ell(I)^{n-1}\ell(I)^{1-n\frac{q_{2}'}{q_{1}}}
\Big)^{\frac{1}{q_{2}'}}
=
|I|^{\frac{1}{q_{2}'}+\frac{1}{q_{1}'}-1}. 
\end{align*} 

Finally, when $q_{1}=\infty $ and $q_{2}>1$, we use the same notation to bound the expression prior to 
\eqref{intr1} by 
\begin{align*} 
\Big(\int _{I'}&\Big|\log \Big(1+\frac{\ell(I)}{\dist_{\infty }(x,I)}\Big)\Big|^{q_{2}'}
dx\Big)^{\frac{1}{q_{2}'}}
\\
&\leq \Big(\ell(I)^{n-1}\int _{0}^{\ell(I)}\Big|\log\Big(1+\frac{\ell(I)}{\rho}\Big) \Big|^{q'_{2}}
d\rho\Big)^{\frac{1}{q'_{2}}}
\\
&\leq |I|^{\frac{1}{q'_{2}}}\Big(\int _{1}^{\infty }
|\log(1+\rho)|^{q_{2}'}\frac{1}{\rho^{2}}d\rho\Big)^{\frac{1}{q_{2}'}}
\lesssim 
|I|^{\frac{1}{q_{2}'}+\frac{1}{q_{1}'}-1}
\end{align*}

With all this and the choice of $\theta $, we get
\begin{align*}
A_{m}&\leq |I|^{\frac{1}{q_{1}}+\frac{1}{q_{2}}}[fb_{1}]_{I,q_{1}}[gb_{2}]_{I',q_{2}}
|I|^{\frac{1}{q_{2}'}+\frac{1}{q_{1}'}-1}(S(\theta )+\tilde F_{K}(I))
\\
&\lesssim |I| [fb_{1}]_{I,q_{1}}[gb_{2}]_{I',q_{2}}\tilde F_{K}(I)
\end{align*}

By symmetry, we have the same result under the assumption that $\frac{1}{q_{1}}+\frac{n}{q_{2}}<1$. Now, 
we can interpolate between the cases $q_{1,\alpha_{0}}>1$, $q_{2,\alpha_{0}}=\infty $ and $q_{1,\alpha_{1}}=\infty $, $q_{2,\alpha_{1}}>1$. This way we obtain the result for $\frac{1}{\bar{q}_{1}}+\frac{1}{\bar{q}_{2}}=\frac{1-\alpha }{q_{1,\alpha_{0}}}+\frac{\alpha }{q_{2,\alpha_{1}}}<1$.

\section{The definition of $Tb_{i}$}

Lemma \ref{definecmo} below defines $Tb_{1}$, $T^{*}b_{2}$ for locally integrable testing functions $b_{1},b_{2}$
as functionals in the dual of the subspace of $\mathcal C_{0}(\mathbb R^{n})$ of functions with mean zero with respect $b_{2}$ or $b_{1}$ respectively.

Then
the hypothesis that $Tb_{1}\in \BMO_{b}(\mathbb R^{n})$ means that 
$
|\langle T_{b}1, f\rangle |=|\langle Tb_{1}, b_{2}f\rangle |\leq C
$
holds for a dense subset of the unit ball of $H_{b}^{1}(\mathbb R^{n})$. 
Furthermore, the hypothesis
$Tb_{1}\in \CMO_{b}(\mathbb R^{n})$ means that
$$
 \lim_{M\rightarrow \infty} |\langle P_{M}^{\perp}T_{b}1,f\rangle |=0
$$
holds uniformly in a dense subset of the unit ball of $H_{b}^{1}(\mathbb R^{n})$. 
In particular, we verify this estimate for $f\in \mathcal C_{0}(\mathbb R^{n})$ with mean zero with respect $b_{2}$.
The necessity of 
$Tb_{1} \in \CMO_{b}(\mathbb R^{n})$ when $T$ is a compact operator appears in Proposition \ref{necessity2}.

\begin{lemma}\label{definecmo}
 
Let $T$ be a linear operator associated with a compact Calder\'on-Zygmund kernel $K$ with parameter
$0<\delta <1$. 
Let $b_{1},b_{2}$ be test functions compatible with $T$. 

Let $J\in \mathcal C$, $f$ locally integrable with support on $J$ and mean zero
with respect to $b_{2}$. Then 
the limit
${\mathcal L}_{b}(f)=\lim_{k\to \infty} \langle T_{b}\chi_{2^{k}J},f\rangle $ exists.

Moreover, for all $k\geq 2$
\begin{align}\label{errorterm}
|{\mathcal L}_{b}(f)- \langle T_{b}\chi_{2^{k}J} ,f\rangle |
&\lesssim 2^{-k\delta } 
|J|\langle M_{q_{1}}b_{1}\rangle_{2^{k}J}
[ b_{2}f]_{J,q_{2}}
F_{K}(I,J,I).
\end{align}

\end{lemma}

\proof


For $k\geq 2$,
we denote $\Delta_{k}=(2^{k+1}J)\backslash (2^{k}J)=\{ t\in\mathbb R^{n} : 2^{k-1}\ell(J)\leq |t-c(J)|_{\infty }\leq 2^{k}\ell(J)\}$ and 
$\Psi_k=\chi_{\Delta_{k}}$.
We aim to estimate
$|\langle T_{b}\Psi_{k}, f\rangle |$.
For $t\in \sup \Psi_k =\Delta_{k}$ 
and $x\in \sup f\subset J$, we have 
\begin{equation}\label{forsmooth}
|x-c(J)|_{\infty }
\leq \ell(J)/2
<\ell(J)2^{k-2}\leq 2^{-1}|t-c(J)|_{\infty }.
\end{equation}
Then $t$ and $x$ cannot be equal, which implies that the supports of $\Psi_k$ and $f$
are disjoint. Therefore, we can use the kernel representation
and the zero mean of $f$ with respect to $b_{2}$ to write
\begin{align*}
\langle T_{b}\Psi_k, f\rangle
&=\int_{J} \int_{\Delta_{k}} \Psi_{k}(t)f(x)(K(t,x)-K(t,c(J)))b_{1}(t)b_{2}(x)\, dtdx .
\end{align*}
Whence, $|\langle T_{b}\Psi_k, f\rangle|$ can be bounded by
\begin{align*}
\| b_{1}\chi_{I}\|_{L^{q_{1}}(2^{k+1}J)}
\| f b_{2}\|_{L^{q_{2}}(J)}
\Big(\int_{J}
\Big(\int_{\Delta_{k}}
|K(t,x)-K(t,c(J))|^{q_{1}'}\, dt
\Big)^{\frac{q_{2}'}{q_{1}'}}dx\Big)^{\frac{1}{q_{2}'}}.
\end{align*}
We denote the last factor by $\Int $. By \eqref{forsmooth} and the smoothness condition of a compact Calder\'on-Zygmund kernel,
we have 
$$
\Int \lesssim \Big(\int_{J}\Big(\int_{\Delta_{k}}
\frac{|x-c(J)|_{\infty }^{q_{1}'\delta }}{|t-c(J)|_{\infty }^{(n+\delta )q_{1}'}}F_{K}(t,x,c(J))^{q_{1}'}\, dt
\Big)^{\frac{q_{2}'}{q_{1}'}}dx\Big)^{\frac{1}{q_{2}'}}
$$
with $F_{K}(t,x,c(J))=L(|t-c(J)|_{\infty })S(|x-c(J)|_{\infty })D(1+\frac{|t+c(J)|_{\infty }}{1+|t-c(J)|_{\infty }})$.
By Lemma \ref{boundtoFK},
$
F_{K}(t,x,c(J))\lesssim F_{K}(2^{k}J, J,2^{k}J)
$
and so, 
\begin{align*}
\Int \lesssim 
\frac{\ell(J)^{\delta }}{(2^{k}\ell(J))^{n+\delta }}
F_{K}(2^{k}J, J,2^{k}J) |\Delta_{k}|^{\frac{1}{q_{1}'}}|J|^{\frac{1}{q_{2}'}}.
\end{align*}
With this and $|\Delta_{k}|\lesssim 2^{(k+1)n}|J|$, we have
\vspace{-.2cm}
\begin{align*}
|\langle T_{b}\Psi_k, f\rangle|&\lesssim 
[b_{1}]_{2^{k+1}J,q_{1}}
|2^{k+1}J|^{\frac{1}{q_{1}}}
[f b_{2}]_{J,q_{2}}
|J|^{\frac{1}{q_{2}}}
\frac{|\Delta_{k}|^{\frac{1}{q_{1}'}}|J|^{\frac{1}{q_{2}'}}}{2^{k(n+\delta)}|J|}F_{K}(2^{k}J, J,2^{k}J)\\
&\lesssim 2^{-k\delta}
|J|
[b_{1}]_{2^{k+1}J,q_{1}}[f b_{2}]_{J,q_{2}}F_{K}(2^{k}J, J,2^{k}J)
\lesssim 2^{-k\delta}
|J|
\end{align*}
by hypothesis \eqref{b-averageandbound0}.
The right hand side of previous inequality tends to zero when $k$ tends to infinity, proving that the sequence $(\langle T_{b}\chi_{2^{k}J} ,f\rangle )_{k\geq 2}$ is Cauchy and thus, the existence of the limit, which we write as ${\mathcal L}_{b}(f)$.

Now, the stated rate of convergence follows by summing a geometric series. For
every $k'\geq 2$, we have
\begin{align*}
\hspace{-.1cm}|{\mathcal L}_{b}(f)&-\langle T_{b}\chi_{2^{k}J} ,f\rangle |
\leq \lim_{m\rightarrow \infty }|{\mathcal L}_{b}(f)-\langle T_{b}\chi_{2^{m}J},f\rangle |
+\sum_{k'=k}^{\infty }|\langle T_{b}\Psi_{k'},f\rangle |
\\
&\lesssim 
|J|
[f  b_{2}]_{J,q_{2}}\sum_{k'=k}^{\infty }
2^{-k'\delta }[b_{1}]_{2^{k'+1}J,q_{1}}F_{K}(2^{k'}J, J,2^{k'}J)
\\
&\lesssim 
|J|
[f b_{2}]_{J,q_{2}}L(2^{k}J)S(J)
\sum_{k'=k}^{\infty }
2^{-k'\delta }[b_{1}]_{2^{k'+1}J,q_{1}}D(2^{k'}J),
\end{align*}
using $\ell(2^{k}J)\leq \ell(2^{k'}J)$ and that $L$ is non-decreasing. 
For $k'\geq k$ we have
$
[b_{1}]_{2^{k'+1}J,q_{1}}
\leq \inf_{x\in 2^{k}J}(M_{q_{1}}b_{1})(x)
\leq \langle M_{q_{1}}b_{1}\rangle_{2^{k}J}
$
and so, we  obtain
\begin{align*}
|{\mathcal L}_{b}(f)-\langle T_{b}\chi_{2^{k}J} ,f\rangle |
\lesssim 2^{-k\delta}&|J|
\langle M_{q_{1}}b_{1}\rangle_{2^{k}J}[f b_{2}]_{J,q_{2}}
\\
&
L(2^{k}J)S(J)
\sum_{k'=0}^{\infty }
2^{-k'\delta }D(2^{k'}(2^{k}J)).
\end{align*}

\section{The operator acting on bump functions}\label{bump}

In this section, 
we develop estimates of the dual pair $\langle T_{b}h_{I},h_{J}\rangle $ in terms of the space and
frequency location of the bump functions $h_{I},h_{J}$.
Proposition \ref{twobumplemma1} is an improvement of the analog result in \cite{V}. 
Although the new proof is influenced by the works \cite{David}, \cite{NTVTb}, \cite{HyNa}, we follow a different approach: we modify the proof in \cite{V} by implementing all the necessary changes to deal with non-continuous bump functions. 

\begin{definition}\label{Haar} Let $b$ be locally integrable with 
$\langle b\rangle_{I}\neq 0$ for all $I\in \mathcal D$. 
We write
$h_{I}^{b}=|I|^{\frac{1}{2}}\big(\frac{1}{|I|\langle b\rangle_{I}}\chi_{I}-\frac{1}{|I_{p}|\langle b\rangle_{I_{p}}}\chi_{I_{p}}\big)$, 
where $I\in \child(I_{p})$. 

We note that 
$
\| h_{I}^{b}\|_{L^{q}(\mathbb R^{n})}\lesssim C_{I}^{b}|I|^{\frac{1}{q}-\frac{1}{2}}
$ and
$
\| h_{I}^{b}b\|_{L^{q}(\mathbb R^{n})}
\lesssim B_{I,q}^{b}|I|^{\frac{1}{q}-\frac{1}{2}}
$, with constants defined in \ref{CBI}.

\end{definition}

\begin{proposition}\label{twobumplemma1}
Let $T$
be a linear operator with a compact C-Z kernel $K$
and parameter $0<\delta <1$. 
Let $1<q_{i}\leq \infty $ with $q_{1}^{-1}+q_{2}^{-1}<1$ 
and $b_{1}, b_{2}$ be functions compatible with $T$. We assume that
$T$ satisfies the weak compactness condition
and  $Tb_{1}=T^{*}b_{2}=0$.

Let $I, J\in \mathcal D$ and  
$h_{I}=h_{I}^{b_{1}}$, $h_{J}=h_{J}^{b_{2}}$ as in Definition \ref{Haar}. 

\vskip5pt
1) When $\rdist (I_{p},J_{p})> 3$,
\begin{flalign*}
\hskip20pt \bullet \hskip10pt |\langle T_{b}(h_{I}),h_{J}\rangle |
\lesssim 
\frac{\ec(I,J)^{\frac{n}{2}+\delta}}{\rdist(I,J)^{n+\delta}}
B_{1}(I,J)F_{1}(I,J),&&
\end{flalign*}
with 
$
B_{1}(I,J)=B^{b_{1}}_{I,q_{1}}B^{b_{2}}_{J,q_{2}}
$ and
$
F_{1}(I, J)=F_{K}(\langle I,J\rangle ,I\smland J, \langle I,J\rangle ).
$ 
\vskip10pt
2) When $\rdist (I_{p},J_{p})\leq 3$  
and $\inrdist(I_{p},J_{p})>1$, 
\begin{flalign*}
\hskip20pt \bullet \hskip10pt  |\langle T_{b}(h_{I}),h_{J}\rangle |
&\lesssim
\frac{\ec(I,J)^{\frac{n}{2}}}{\inrdist(I,J)^{\delta}}
B_{2}(I,J)
F_{2}(I,J),&&
\end{flalign*}
where now, 
\begin{align*}
B_{2}(I,J)&=
\alpha \sum_{R\in \{I,I_{p}\}}
\frac{\langle M_{q_{1}}b_{1}\rangle_{R}}{|\langle b_{1}\rangle_{R}|}
\sum_{R\in \{J,J_{p}\}}\frac{\langle M_{q_{2}}b_{2}\rangle_{R}}{|\langle b_{2}\rangle_{R}|}
\\
&\hskip20pt 
+C_{I}^{b_{1}}C_{J}^{b_{2}}\langle M_{q_{1}}(b_{1}\chi_{I})\rangle_{I\smland J}
\langle M_{q_{2}}(b_{2}\chi_{J})\rangle_{I\smland J} ,
\\
F_{2}(I, J)&=\tilde F_{K}( \langle I,J\rangle ,I\smland J, \langle I,J\rangle )+
F_{K}(I\smland J, I\smland J,  \langle I,J\rangle )
\end{align*} 
with 
$\alpha =1$ if $I\smland J\subsetneq I\smlor J$, $\alpha =0$ otherwise, and 
$\tilde F_{K}$ as in def. \ref{LSDF}.

\vskip10pt
3) When $\rdist (I_{p},J_{p})\leq 3$ and $\inrdist(I_{p},J_{p})=1$, 
\begin{flalign*}
\hskip20pt\bullet \hskip10pt |\langle T_{b}(h_{I}),h_{J}\rangle |
&\lesssim
\ec(I,J)^{\frac{n}{2}}(B_{2}(I,J)F_{2}(I,J)+B_{3}(I,J)F_{3}(I,J)),&&
\end{flalign*}
where 
$B_{2}$, $F_{2}$ are as before and, when $I\neq J$,
\begin{align*}
B_{3}(I,J)&=C^{b_{1}}_{I}[b_{1}]_{3(I_{p}\smland J_{p}),q_{1}}C^{b_{2}}_{J}[b_{2}]_{3(I_{p}\smland J_{p}),q_{1}},
\hskip20pt F_{3}(I, J)=\tilde F_{K}(I\smland J) 
\end{align*}
while when $I=J$,
\begin{align*}
B_{3}(I,J)&\! =\hspace{-.5cm}\sum_{I',I''\in \child(I_{p})}\hspace{-.3cm}C_{I'}^{b_{1}}[b_{1}]_{I',q_{1}}C_{I''}^{b_{2}} [b_{2}]_{I'',q_{2}},
\hskip10pt F_{3}(I, J)\! =\! \tilde F_{K}(I)\! +\! F_{W}(I;M_{T,\epsilon })\! +\! \epsilon  
\end{align*}
for every $\epsilon >0$ with the value $M_{T,\epsilon }$ given in Definition \ref{weakcompactnessdef}.


\end{proposition}


\begin{remark}\label{defBF} We note that $B\! F$ 
in Definition \ref{accretive} of compatible testing function dominates all terms 
$B_{i}\cdot F_{i}$ in the statement of Proposition \ref{twobumplemma1}. 

In fact,  
$F_{i}(I,J)\lesssim \tilde F_{K}(I\smland J, I\smland J,  \langle I,J\rangle )+F_{W}(I;M_{T,\epsilon })$. Moreover,
$
[b_{1}]_{I,q_{1}}
\lesssim \inf_{I} M_{q_{1}}b_{1}
\lesssim \langle M_{q_{1}}b_{1}\rangle_{I}
$. 

We note that $2\rdist(I_{p},J_{p})-1\leq \rdist(I,J)\leq 2(\rdist(I_{p},J_{p})+1)$ and 
$2\inrdist(I_{p},J_{p})\leq \inrdist(I,J)\leq 2\inrdist(I_{p},J_{p})+1$.
\end{remark}

\begin{proof} By symmetry, we assume $\ell(J)\leq \ell(I)$. 
Let $\psi (t,x)=h_{I}(t)h_{J}(x)$, which is supported on $I_{p}\times J_{p}$ 
and has mean zero in the variable $x$ with respect to $b_2$. 



\vskip10pt
{\bf a)} When $3\ell(I_{p})<\diam (I_{p}\cup J_{p})$, we have that $(5I_{p})\cap J_{p}=\emptyset $ and so,  
we can use the kernel representation and
the zero mean of $\psi$ 
to write
$$
\langle T_{b}h_{I},h_{J}\rangle
=\int_{J_{p}}\int_{I_{p}} \psi(t,x) (K(t,x)-K(t,c(J_{p})))b_{1}(t)b_{2}(x)\, dtdx .
$$

Now, $(5I_{p})\cap J_{p}=\emptyset $ and $\ell(J_{p})\leq \ell(I_{p})$ 
imply 
$\diam(I_{p}\cup J_{p})
\leq \ell(I_{p})+|c(I_{p})-c(J_{p})|_{\infty }$.
With this 
and $|t-c(I_{p})|_{\infty }\leq \ell(I_{p})/2$, we prove:
\begin{align*}
|t-c(J_{p})|_{\infty }&\geq |c(I_{p})-c(J_{p})|_{\infty }-|t-c(I_{p})|_{\infty }
\\
&\geq \diam (I_{p}\cup J_{p})-3\ell(I_{p})/2
>\diam(I_{p}\cup J_{p})/2,
\\
|t-c(J_{p})|_{\infty }
&\leq |c(I_{p})-c(J_{p})|_{\infty }+\ell(I_{p})/2\leq \diam(I_{p}\cup J_{p}) .
\end{align*}

Then, 
\begin{align}\label{faraway}
\nonumber
|\langle &T_{b}h_{I},h_{J}\rangle |
\leq 
\| h_{I}b_{1}\|_{L^{q_{1}}(I)}
\| h_{J}b_{2}\|_{L^{q_{2}}(J)}
\\
&\Big(\int_{J_{p}}
\Big(\int_{I\cap \Delta_{I_{p},J_{p}}}|K(t,x)-K(t,c(J_{p}))|^{q_{1}'}\, dt
\Big)^{\frac{q_{2}'}{q_{1}'}}dx\Big)^{\frac{1}{q_{2}'}},
\end{align}
where $\Delta_{I_{p},J_{p}}=\{ t\in \mathbb R^{n}: 
\ell(\langle I_{p},J_{p}\rangle)/2<|t-c(J_{p})|_{\infty }\leq \ell(\langle I_{p},J_{p}\rangle)\}$.

We denote by $\Int $ the integral in \eqref{faraway}.
From 
$3\ell(I_{p})<\diam(I_{p}\cup J_{p})\leq \ell(I_{p})+|c(I_{p})-c(J_{p})|_{\infty }$, we get 
$2\ell(I_{p})<|c(I_{p})-c(J_{p})|_{\infty }$. This inequality
and 
$|x-c(J_{p})|_{\infty }\leq \ell(J_{p})/2$ imply 
\begin{align*}
|t-c(J_{p})|_{\infty }&
\geq 2\ell(I_{p})-\ell(I_{p})/2
\geq 3\ell(J_{p})/2\geq 3|x-c(J_{p})|_{\infty }.
\end{align*}
Then, by the smoothness condition 
of a compact C-Z kernel,
\begin{align*}
\Int\lesssim 
\Big(\int_{J_{p}}\Big(
\int_{I\cap \Delta_{I_{p},J_{p}}}
\frac{|x-c(J_{p})|_{\infty }^{\delta q_{1}'}}{|t-c(J_{p})|_{\infty }^{(n+\delta )q_{1}'}}F_{K}(t,x,c(J_{p}))^{q_{1}'}
dt\Big)^{\frac{q_{2}'}{q_{1}'}}dx\Big)^{\frac{1}{q_{2}'}},
\end{align*}
with 
$
F_{K}(t,x,c(J_{p}))=L(|t-c(J)|_{\infty })S(|x-c(J_{p})|_{\infty })D\Big(1+\frac{|t+c(J_{p})|_{\infty }}{1+|t-c(J_{p})|_{\infty }}\Big)
$. 
By Lemma \ref{boundtoFK}, 
$
F_{K}(t,x,c(J_{p}))\lesssim F_{K}(\langle I,J\rangle, J,\langle I,J\rangle ) 
$
and so, 
\begin{align*}
\Int \lesssim |I|^{\frac{1}{q_{1}'}}|J|^{\frac{1}{q_{2}'}}
\frac{\ell(J)^{\delta }}{\ell(\langle I,J\rangle )^{n+\delta }}
F_{K}(\langle I,J\rangle, J,\langle I,J\rangle ) .
\end{align*}
We then continue the bound in \eqref{faraway} as
\begin{align*}
|\langle T_{b}h_{I},h_{J}\rangle |
&\lesssim 
B_{I,q_{1}}|I|^{-\frac{1}{2}+\frac{1}{q_{1}}}B_{J,q_{2}}|J|^{-\frac{1}{2}+\frac{1}{q_{2}}}
\\
&\hskip20pt |I|^{\frac{1}{q_{1}'}}|J|^{\frac{1}{q_{2}'}}\frac{\ell(J)^{\delta }}{\ell(\langle I,J\rangle )^{n+\delta }}
F_{K}(\langle I,J\rangle , J, \langle I,J\rangle )
\\
&=\Big(\frac{\ell(J)}{\ell(I)}\Big)^{\frac{n}{2}+\delta }
\Big(\frac{\ell(I)}{\ell(\langle I,J\rangle )}\Big)^{n+\delta }
B_{1}(I,J)F_{1}(I,J).
\end{align*}
This is the result corresponding to the case 1) in the statement. 

{\bf b)} When $\diam (I_{p}\cup J_{p})\leq 3\ell(I_{p})$, we have
$J_{p}\subset 5I_{p}$.
%

We denote by $\tilde{I}_{p}=\frac{\ell(I)}{\ell(J)}J_{p}$, the cube with 
$c(\tilde{I}_{p})=c(J_{p})$, 
$\ell(\tilde{I})=\ell(I)$. Let $e\in \mathbb N$ such that $2^{e}=\frac{\ell(I)}{\ell(J)}\geq 1$. 
We write 
$\varphi_{R}=\frac{|R|^{1/2}}{|R|\langle b_{1}\rangle_{R}}\chi_{R}$ with $R\in \{I,I_{p}\}$ and 
define
$\tilde h_{I}(t)=\varphi_{I}(c(J_{p}))\chi_{I\cap \tilde{I}_{p}}(t)-\varphi_{I_{p}}(c(J_{p}))\chi_{I_{p}\cap \tilde{I}_{p}}(t)$.
Then we 
perform the decomposition
\begin{align*}
\psi&=\psi_{0}+\psi_{1},
\\
\nonumber 
\psi_{1}(t,x)&=\tilde h_{I}(t)h_{J}(x)
\end{align*}

We work first with the term $\psi_{1}$. We denote by
$\psi_{R}(x)
=\varphi_{R}(c(J_{p}))h_{J}(x)$, which satisfies $\psi_{I}\equiv 0$,
$\psi_{I_{p}}\equiv 0$ when $J_{p}\subset (5I_{p})\backslash I_{p}$ and 
$$
\| \psi_{R} b_{2}\|_{L^{q_{2}}(\mathbb R^{n})}
\lesssim |\langle b_{1}\rangle_{R}|^{-1}|R|^{-\frac{1}{2}}
B_{J,q_{2}}^{b_{2}}|J|^{\frac{1}{q_{2}}-\frac{1}{2}}.
$$
%
%

$\bullet $ When $\inrdist(I_{p},J_{p})>1$, we have
either $J_{p}\subset (5I_{p})\backslash I_{p}$ with $\ell(J)\leq \ell(I)$, or $J_{p}\subsetneq I_{p}$ with $\ell(J)\leq \ell(I)/8$. In the former case
$\tilde h_{I}\equiv 0$ and so, we have that $e\geq 3$.
Then, by the special cancellation condition $Tb_{1}=0$, 
equalities $\ell(\tilde{I}_{p})= 2^{e}\ell(J_{p})$ and $|I|\leq |R|$,
$h_{J}$ being supported on $J_{p}$ with mean zero with respect to $b_2$ 
and the error estimate \eqref{errorterm} of Lemma \ref{definecmo}
with the selected $e\geq 3$,
we can bound the contribution of $\psi_{1}$ by
\begin{align}\label{extradelta}
\nonumber
&\sum_{R\in \{I,I_{p}\}}|\langle T_{b}\chi_{I\cap \tilde{I}_{p}},\psi_{R}\rangle | 
= \sum_{R\in \{I,I_{p}\}}|\langle T_{b}\chi_{R\cap \tilde{I}_{p}},\psi_{R}\rangle 
-\langle Tb_{1},b_{2}\psi_{R}\rangle |
\\
\nonumber
&\lesssim 2^{-e\delta }|J|
\sum_{R\in \{I,I_{p}\}}\inf_{x\in R\cap 2^{e+1}J} M_{q_{1}}b_{1}(x)[\psi_{R}b_{2}]_{J,q_{2}}
\tilde F_{K}(2^{e}J,J,2^{e}J)
\\
\nonumber
&\lesssim \left(\frac{\ell(J)}{\ell(I)}\right)^{\delta }|J|
B_{J,q_{2}}^{b_{2}}|J|^{-\frac{1}{2}}|I|^{-\frac{1}{2}}
\hspace{-.3cm}\sum_{R\in \{I,I_{p}\}}\frac{\inf_{x\in R} M_{q_{1}}b_{1}(x)}{|\langle b_{1}\rangle_{R}|}
\tilde F_{K}(\tilde{I},J,\tilde{I})
\\
& \lesssim \left(\frac{\ell(J)}{\ell(I)}\right)^{\frac{n}{2}+\delta}
\sum_{R\in \{I,I_{p}\}}\frac{\langle M_{q_{1}}b_{1}\rangle_{K}}{|\langle b_{1}\rangle_{K}|}
\, B_{J,q_{2}}^{b_{2}}
\tilde F_{K}(I,J,I).
\end{align}
The last two inequalities are due to the facts that $I\subset I_{p}\subset 2\tilde{I}_{p}$ and that $\ell(\tilde{I})=\ell(I)$, $|c(I_{p})-c(\tilde{I}_{p})|_{\infty }\lesssim \ell(I)$ imply 
$\rdist (\tilde{I},\mathbb B)\approx \rdist (I,\mathbb B)$. 

Since $J_{p}\subset 5I_{p}$ implies  $\ell(J_{p})+\dist_{\infty }(J_{p},\mathfrak{D}_{I_{p}})\leq 2\ell(I_{p})$, we have that   
$\ec(I,J)=\frac{\ell(J_{p})}{\ell(I_{p})}\lesssim \frac{\ell(J_{p})}{\ell(J_{p})+\dist_{\infty }(J_{p},\mathfrak{D}_{I_{p}})}\lesssim \inrdist(I_{p},J_{p})^{-1}$ and so,  \eqref{extradelta} 
is smaller than the first term of case 2)  in the statement. 

$\bullet $ When $\inrdist(I_{p},J_{p})=1$, if $e\geq 2$ we can proceed exactly in the same way. The cases $e=0$ or $e=1$ will be treated at the end. 

{\bf c)} Now, we work with the term 
$\psi_{0}(t,x)
=(h_{I}(t)-\tilde h_{I}(t))h_{J}(x)$,
which we further decompose as follows:
\begin{align*}
\psi_{0}&=\psi_{out}+\psi_{in},
\\
\psi_{in}(t,x)&=\psi_{0}(t,x)\chi_{3J_{p}}(t).
\end{align*}

%

{\bf c.1)} We work first with 
\begin{align}\label{psiout}
\psi_{out}(t,x)= (h_{I}(t)-\tilde h_{I}(t))(1-\chi_{3J_{p}}(t))h_{J}(x)
\end{align}
and divide the study in two parts:
\begin{itemize}
\item When $J_{p}\subset (5I_{p})\backslash I_{p}$, we have $\varphi_{K}(c(J_{p}))=0$ for $K\in \{ I,I_{p}\}$. Then 
$\tilde h_{I}(t)=0$ and 
$\psi_{out}(t,x)= h_{I}(t)(1-\chi_{3J_{p}}(t))h_{J}(x)$.
Consequently,  $\psi_{out}(t,x)\neq 0$ implies $t\in I_{p}\cap (3J_{p})^{c}$ and  
$$
\hskip40pt |t-c(J_{p})|_{\infty }\geq \frac{\ell(J_{p})}{2}+\dist_{\infty}(I_{p},J_{p})=\frac{\ell(J_{p})}{2}+\dist_{\infty }(J_{p},{\mathfrak D}_{I_{p}}).
$$
\item When $J_{p}\subseteq I_{p}$, we further divide in two more cases:
\begin{itemize}
\item When $J_{p}=I_{p}$, we have  $\tilde{I}_{p}=I_{p}$. Then $h_{I}=\tilde h_{I}$ and so, 
$\psi_{out}(t,x)\equiv 0$. 
\item When $J_{p}\subsetneq I_{p}$, we have that $J_{p}\subseteq I'$ for some $I'\in \child(I_{p})$. Then 
$h_{I}(t)=\varphi_{I}(c(J_{p}))-\varphi_{I_{p}}(c(J_{p}))$
for all $t\in I'$ and, from \eqref{psiout}, we get $\psi_{out}(t,x)=0$.
That is, $\psi_{out}(t,x)\neq 0$ implies 
$t\in (I_{p}\backslash I')\cap (3J_{p})^{c}$, getting again 
$$
\hskip50pt |t-c(J_{p})|_{\infty }\geq \frac{\ell(J_{p})}{2}+\dist_{\infty}(I\backslash I', J_{p})
= \frac{\ell(J_{p})}{2}+\dist_{\infty }(J_{p},{\mathfrak D}_{I_{p}}) .
$$
\end{itemize}
\end{itemize}

In both cases then, $|t-c(J_{p})|_{\infty }\geq \frac{1}{2}\inrdist(I_{p},J_{p})\ell(J_{p})$. 
Also in both cases,
$|t-c(J_{p})|_{\infty }\leq \diam(I_{p}\cup J_{p})\leq 3\ell(I_{p})$ and
$|t-c(J_{p})|_{\infty }\geq 3\ell(J_{p})/2> \ell(J_{p})$. The latter inequality and $|x-c(J_{p})|_{\infty }\leq \ell(J_{p})/2$
imply $2|x-c(J_{p})|_{\infty }<|t-c(J_{p})|_{\infty }$. Then
we can use the kernel representation 
and the zero mean of $\psi_{out}$ with respect to $b_{2}$
to write
\begin{align}\label{newout}
\nonumber
\langle T_{b}((h_{I}&-\tilde h_{I})(1-\chi_{3J_{p}})),h_{J}\rangle 
\\
&=\int_{J_{p}} \int_{I_{p}\cap \bar{J}} \psi_{out}(t,x) (K(t,x)-K(t,c(J_{p}))) b_1(t) b_2(x) \, dtdx ,
\end{align}
where $\bar{J}=\{ t\in \mathbb R^{n} : \ell(J_{p})/2+\dist_{\infty }(J_{p},{\mathfrak D}_{I_{p}})< |t-c(J_{p})|_{\infty }\leq 3\ell(I_{p})\}$.

Now, we decompose 
$
\bar{J}\subset \bigcup_{m=m_{0}}^{m_{1}}J^{m},
$
where $$J^{m}=\{ t\in I_{p}: 2^{m}\ell(J_{p})< |t-c(J_{p})|_{\infty} \leq 2^{m+1}\ell(J_{p})\},$$ 
with
$m_{0}=\log (\inrdist(I_{p},J_{p})/4)\gtrsim \log (\inrdist(I,J))$ and $m_{1}=\log (3\frac{\ell(I)}{\ell(J)})$. 
Since $J^{m}$ is the difference of two concentric cubes with 
diameters $2^{m+1}\ell(J_{p})$ and $2^{m+2}\ell(J_{p})$, with abuse of notation we write
$\ell(J^{m})=2^{m+2}\ell(J_{p})$ and $c(J^{m})=c(J_{p})$. 
This way, the modulus of \eqref{newout} can be bounded by
\begin{align}\label{newout2}
\sum_{m=m_{0}}^{m_{1}}
&\| \psi_{out}\|_{L^{\infty }(J^{m})\times L^{\infty }(J_{p})}
\| b_{1}\chi_{I}\|_{L^{q_{1}}(J^{m})}\| b_{2}\|_{L^{q_{2}}(J_{p})}
\\
\nonumber
&\Big(\int_{J_{p}} \Big(\int_{I_{p}\cap J^{m}} 
|K(t,x)-K(t,c(J_{p}))|^{q_{1}'}
\, dt\Big)^{\frac{q_{2}'}{q_{1}'}}dx\Big)^{\frac{1}{q_{2}'}}.
\end{align}
We note that 
$
\| \psi_{out}\|_{L^{\infty }(J^{m})\times L^{\infty }(J_{p})}\lesssim 
C_{I}^{b_{1}}|I|^{-\frac{1}{2}}
C_{J}^{b_{2}}|J|^{-\frac{1}{2}}.
$ 
By the smoothness property \eqref{smoothcompactCZ},
we estimate the double integral, denoted again $\Int$: 
\begin{equation}\label{intin}
\Int \leq \Big(\int_{J_{p}}\Big(\int_{J_{m}}\frac{|x-c(J_{p})|^{q_{1}'\delta }}{|t-c(J_{p})|_{\infty }^{q_{1}'(n+\delta)}}
F_{K}(t,x,c(J_{p}))^{q_{1}'}dt\Big)^{\frac{q_{2}'}{q_{1}'}}
dx\Big)^{\frac{1}{q_{2}'}}
\end{equation}
with 
$F_{K}(t,x,c(J_{p}))=L(|t-c(J_{p})|_{\infty })
S(|x-c(J_{p})|_{\infty })D\Big(1+\frac{|t+c(J_{p})|_{\infty }}{1+|t-c(J_{p})|_{\infty }}\Big)$.
Since $2^{m+2}\ell(J)\geq 2^{m+1}\ell(J_{p})\geq |t-c(J_{p})|_{\infty }> 2^{m}\ell(J_{p})\geq \ell(J)$
and $|x-c(J_{p})|_{\infty }\leq \ell(J)$, 
by the proof of Lemma \ref{boundtoFK}, we have 
$$
F_{K}(t,x,c(J_{p}))\leq L(\ell(J))
S(\ell(J))D\Big(1+\frac{|c(J)|_{\infty }}{1+2^{m+2}\ell(J)}\Big) .
$$ 
Moreover, since $J^{m}\subset 10I_{p}$, we get
\begin{align*}
1+\frac{|c(J_{p})|_{\infty }}{1+2^{m+2}\ell(J)}
&\geq 1+\frac{|c(J^{m})|_{\infty }}{1+2^{m+3}\ell(J)}
\gtrsim \rdist (J^{m},\mathbb B)
\\
&\geq \rdist (10I_{p},\mathbb B)
\gtrsim \rdist (I,\mathbb B),
\end{align*}
with clear meaning of $\rdist (J^{m},\mathbb B)$ despite $J^{m}$ is not a cube. Then
\begin{align*}
F_{K}(t,x,c(J_{p}))
&\leq L(\ell(J))
S(\ell(J))D(\rdist (I,\mathbb B))
=F_{K}(J,J,I).
\end{align*}
With this and $|J^{m}|\approx 2^{mn}|J|$, we continue the bound in \eqref{intin} as
\begin{align*}
\Int
&\lesssim F_{K}(J,J,I)\frac{\ell(J)^{\delta }}{(2^{m}\ell(J))^{n+\delta}}
|J^{m}|^{\frac{1}{q_{1}'}}|J|^{\frac{1}{q_{2}'}}
\\
&\lesssim F_{K}(J,J,I) 2^{-m\delta}
|J^{m}|^{-\frac{1}{q_{1}}}|J|^{\frac{1}{q_{2}'}}.
\end{align*}

Therefore, we can estimate \eqref{newout2} by
\begin{align*}
&\sum_{m=m_{0}}^{m_{1}}C_{I}^{b_{1}}
|I|^{-\frac{1}{2}}
C_{J}^{b_{2}}|J|^{-\frac{1}{2}}
[b_{1}\chi_{I}]_{J^{m},q_{1}}|J^{m}|^{\frac{1}{q_{1}}}[b_{2}]_{J,q_{2}}|J|^{\frac{1}{q_{2}}}
\\
&
\hskip50pt  |J^{m}|^{-\frac{1}{q_{1}}}|J|^{\frac{1}{q_{2}'}} 
F_{K}(J,J,I)2^{-m\delta} 
\\
&\lesssim 
\left( \frac{|J|}{|I|} \right)^{\frac{1}{2}}[b_{2}]_{J,q_{2}}F_{K}(J,J,I) 
\sum_{m=m_{0}}^{m_{1}}
2^{-m\delta}[b_{1}\chi_{I}]_{J^{m},q_{1}}.
\end{align*}
Now, since $J^{m}\subset 2^{m+2}J_{p}\subset 14I$, we have
$$
[b_{1}\chi_{I}]_{J^{m},q_{1}}^{q_{1}}\lesssim \frac{1}{2^{mn}|J|}\int_{2^{m+2}J_{p}\cap I}|b_{1}(x)|^{q_{1}}dx
\lesssim [b_{1}\chi_{I}]_{2^{m+2}J_{p},q_{1}}^{q_{1}}
$$
Moreover $\displaystyle [b_{1}\chi_{I}]_{2^{m+2}J_{p},q_{1}}\leq \inf_{x\in J}M_{q_{1}}(b_{1}\chi_{I})(x)\leq \langle M_{q_{1}}(b_{1}\chi_{I})\rangle_{J}$ and so,
\begin{align*}
|\langle T_{b}&((h_{I}-\tilde h_{I})(1-\chi_{3J_{p}})),h_{J}\rangle |
\\
&\lesssim 
\ec(I,J)^{\frac{n}{2}}\langle M_{q_{1}}(b_{1}\chi_{I})\rangle_{J} [b_{2}]_{J,q_{2}}F_{K}(J,J,I) 
\hspace{-.2cm}\sum_{m\geq \log (\inrdist(I,J))}\hspace{-.4cm}
2^{-m\delta}
\\
&\lesssim 
\frac{\ec(I,J)^{\frac{n}{2}}}{\inrdist(I,J)^{\delta}}\langle M_{q_{1}}(b_{1}\chi_{I})\rangle_{J} [b_{2}]_{J,q_{2}}F_{K}(J,J,I), 
\end{align*}
smaller than the second term of case 2) and the first term of case 3).

{\bf c.2)} We now work with 
\begin{align}\label{finalpsiin}
\psi_{in}(t,x)= (h_{I}(t)-\tilde h_{I}(t))\chi_{3J_{p}}(t)h_{J}(x).
\end{align}

{\bf c.2.1)} We first consider the case $\ell(J)<\ell(I)$. 

We start by showing that when 
$\inrdist(J_{p},I_{p})>1$, we have $\psi_{in}\equiv 0$ and so, this term does not appear in case 2) in the statement.  
As said before, the cubes for which $\inrdist(J_{p},I_{p})>1$
satisfy either $J_{p}\subset (5I_{p})\backslash I_{p}$ with $I_{p}\cap 3J_{p}=\emptyset$ or $3J_{p}\subsetneq I_{p}$ with $\ell(J)\leq \ell(I)/8$ . 
In the former case, we have $h_{I}(t)\chi_{3J_{p}}(t)=\tilde h_{I}(t)=0$ and so, $\psi_{in}\equiv 0$. In the latter case, 
we get 
$3J_{p}\subseteq I'$ for some $I'\in \child(I_{p})$ and $3J_{p}\subset \tilde{I}_{p}$. Therefore,   
$h_{I}(t)\chi_{3J_{p}}(t)=\tilde h_{I}(t)\chi_{3J_{p}}(t)$ and  
$\psi_{in}\equiv 0$ again.  

We now consider those cubes $J$ such that 
$\inrdist(J_{p},I_{p})=1$. The cardinality of this family of cubes is at most $c^{n}(\ell(I)/\ell(J))^{n-1}$ for some constant 
$c>1$.
As before, we divide in two cases:
\begin{itemize}
\item When $J_{p}\subseteq (5I_{p})\backslash I_{p}$, we have $\varphi_{K}(c(J_{p}))=0$ for $K\in \{ I, I_{p}\}$ and 
\eqref{finalpsiin} reduces to 
$$
\psi_{in}(t,x)=h_{I}(t)\chi_{3J_{p}}(t)h_{J}(x)=h_{I}(t)\chi_{(3J_{p})\backslash J_{p}}(t)h_{J}(x),
$$ 
since $I_{p}\cap J_{p}=\emptyset $. 
We also note that in this case, $h_{I}$
and $h_{J}$ have disjoint compact support. 

\item When $J_{p}\subsetneq I_{p}$, we have $J_{p}\subseteq I'$ for some $I'\in \child(I_{p})$ and so, $c(J_{p})\in I'$. 
Then we decompose as 
$
3J_{p}=((3J_{p})\cap I')\cup ((3J_{p})\backslash I')
$. 

For all $t\in (3J_{p})\cap I'$, we have $t,c(J_{p})\in I'$, which implies $|t-c(J_{p})|_{\infty }\leq \ell(I_{p})/2$ and so, $t\in \tilde{I}_{p}$.  
Then $h_{I}(t)=\tilde h_{I}(t)$ and, from 
\eqref{finalpsiin}, we get 
$\psi_{in}(t,x)=0$.

On the other hand, for all $t\in (3J_{p})\backslash I'$, we have 
\begin{align*}
\psi_{in}(t,x)= (h_{I}(t)-\tilde h_{I}(t))\chi_{(3J_{p})\backslash I'}(t)h_{J}(x),
\end{align*}
where $(h_{I}-\tilde h_{I})\chi_{(3J_{p})\backslash I'}$ is  disjoint with $h_{J}$ since 
$J_{p}\subset I'$. Moreover, $\chi_{(3J_{p})\backslash I'}\leq \chi_{(3J_{p})\backslash J_{p}}$. 
\end{itemize}

Then we can write in both cases
$
\langle T_{b}h_{I},h_{J}\rangle 
=\langle T_{b}(h_{I}\chi_{(3J_{p})\backslash J_{p}}),h_{J}\chi_{J_{p}}\rangle 
$
%
and, by Lemma \ref{adjacentsquarescoro}, we have
\begin{align*}
|\langle &T_{b}h_{I},h_{J}\rangle |
\lesssim 
|J_{p}| [h_{I}b_{1}]_{3J_{p},q_{1}}[h_{J}b_{2}]_{J_{p},q_{2}}\tilde F_{K}(J)
\\
&\lesssim 
C^{b_{1}}_{I}[b_{1}]_{3J_{p},q_{1}}|I|^{-\frac{1}{2}}
B^{b_{2}}_{J_{p},q_{2}}|J|^{-\frac{1}{2}}|J_{p}| 
F_{3}(J)
\leq
\Big(\frac{\ell(J)}{\ell(I)}\Big)^{\frac{n}{2}}
B_{3}(I,J)F_{3}(J).
\end{align*}
This is the second term of case 3) in the statement when $\ell(J)<\ell(I)$. 

{\bf c.2.2)} Finally, we consider $\ell(J)=\ell(I)$. 
For this case,  which implies $\inrdist(I_{p},J_{p})=1$, we recover the original notation $h_{I}^{b_{1}}, h_{J}^{b_{2}}$ indicating the dependence of the bump functions. 
We note that $\tilde{I}_{p}=J_{p}$ and so, 
\begin{equation}\label{psiin}
\psi_{in}(t,x)= (h_{I}^{b_{1}}(t)-\tilde h_{I}^{b_{1}}(t))h_{J}^{b_{2}}(x).
\end{equation}

If $J_{p}\subset (5I_{p})\backslash I_{p}$, we have $\tilde h_{I}^{b_{1}}\equiv 0$. 
We apply Lemma \ref{adjacentsquarescoro} as in subcase c.2.1) to obtain the second term in 3) for
$I\neq J$ and $\ell(I)=\ell(J)$. 

We are left with the case $J_{p}=I_{p}$, for which $J\in \child(I_{p})$.
For the first term of $\psi_{in }$ in \eqref{psiin}, 
we have
\begin{align*}
h_{I}^{b_{1}}(t)h_{J}^{b_{2}}(x)
=\sum_{I'\in \child(I_{p})}\alpha_{I'}\chi_{I'}(t)
\sum_{I''\in \child(I_{p})}\beta_{I''}\chi_{I''}(x)
\end{align*}
with $\alpha_{I}=|I|^{\frac{1}{2}}(\frac{1}{|I|[b_{1}]_{I}}-\frac{1}{|I_{p}|[b_{1}]_{I_{p}}})$,  
$\alpha_{I'}=-|I|^{\frac{1}{2}}\frac{1}{|I_{p}|[b_{1}]_{I_{p}}}$ for $I'\neq I$ and the same for $\beta_{I''}$ just changing $b_{1}$ and $I$ by $b_{2}$ and $J$. This implies
\begin{align*}
\langle T_{b}h_{I}^{b_{1}}, h_{J}^{b_{2}}\rangle 
&
=\sum_{I'\in \child(I_{p})}\sum_{I''\in \child(I_{p})}\alpha_{I'}\beta_{I''}
\langle T_{b}\chi_{I'},\chi_{I''}\rangle .
\end{align*}
The same reasoning applied to the second term of $\psi_{in}$ in \eqref{psiin} gives
\begin{align*}
\langle T_{b}\tilde h_{I}^{b_{1}}, h_{I'}^{b_{2}}\rangle 
&=\alpha_{I'}\sum_{I''\in \child(I_{p})}\alpha_{I'}\beta_{I''}
\langle T_{b}\chi_{I'},\chi_{I''}\rangle ,
\end{align*}
with $I'$ such that $c(J_{p})\in I'$. Thus, we can study both cases together.

For $I'\neq I''$, since $\dist(I',I'')=0$, 
we can proceed as in c.2.1): from
\eqref{3IminusI}
in Lemma \ref{adjacentsquarescoro}  (or even  
 \eqref{adja} 
in Lemma \ref{adjacentsquares} )  
we get
\begin{align*}
|\langle T_{b}\chi_{I'},\chi_{I''}\rangle |
\lesssim |I| [b_{1}]_{I',q_{1}}[b_{2}]_{I'',q_{2}}
\tilde F_{K}(I) 
\end{align*}

For $I'=I''$, the weak compactness condition of Definition \ref{weakcompactnessdef} gives
\begin{align*}
|\langle T_{b}\chi_{I'},\chi_{I'}\rangle |
&\lesssim |I|[b_{1}]_{I',q_{1}}[b_{2}]_{I',q_{2}}
(F_{W}(I';M)+\epsilon ).
\end{align*}

From $|\alpha_{I}|\lesssim C_{I}^{b_{1}}|I|^{-\frac{1}{2}}$,
$|\beta_{I'}|\lesssim C_{I'}^{b_{2}}|I'|^{-\frac{1}{2}}$
we have $|I||\alpha_{I}||\beta_{I'}|\lesssim C_{I}^{b_{1}}C_{I'}^{b_{2}}$. With this, 
\begin{align*}
|\langle T_{b}h_{I}^{b_{1}}, h_{I}^{b_{2}}\rangle |
&\lesssim 
\hspace{-.2cm}
\sum_{I',I''\in \child(I_{p})}\hspace{-.3cm}C_{I'}^{b_{1}}[b_{1}]_{I',q_{1}}C_{I''}^{b_{2}} [b_{2}]_{I'',q_{2}}
 (\tilde F_{K}(I) +F_{W}(I;M_{T,\epsilon })+\epsilon )
\end{align*}
%
%
%
This is the second term of case 3) in the statement when $I=J$. 
There is still one case to end the proof: the term left undone at the end of case b), that is,
 the bound for $|\langle T_{b}\chi_{\tilde{I}},\tilde{\psi}\rangle |$ when $\inrdist(I_{p},J_{p})=1$ and $\ec(I,J) \in \{0,1\}$. But it now is clear that this expression can be bounded in the same way we did in case 
 c.2.1) with the use of Lemma \ref{adjacentsquarescoro} and case c.2.2) using the weak compactness condition.
 This provides the first term of case 3) in the statement.
\end{proof}

\section{The adapted wavelet system}

\begin{definition} Let $b$ be locally integrable with
$\langle b\rangle_{I}\neq 0$ for $I\in \mathcal D$. 
Following 
\cite{NTVTb}, we define the expectation associated with $b$  
$
E_{Q}^{b}f=\frac{\langle f\rangle_{Q}}{\langle b\rangle_{Q}}b\chi_{Q}
$
for every locally integrable function $f$. And for every $k\in \mathbb Z$,
$$
E_{k}^{b}f=\sum_{\tiny \begin{array}{c}Q\in \mathcal D\\ \ell(Q)=2^{-k}\end{array}}E_{Q}^{b}f .
$$
We also define their corresponding difference operators
$$
\Delta_{k}^{b}f=E_{k}^{b}f-E_{k-1}^{b}f
=\sum_{\tiny \begin{array}{c}Q\in \mathcal D\\ \ell(Q)=2^{-(k-1)}\end{array}}\Delta_{Q}^{b}f ,
$$
where
\begin{equation}\label{Deltaincoord}
\Delta_{Q}^{b}f=\Big(\sum_{I\in \child(Q)}E_{I}^{b}f\Big)-E_{Q}^{b}f
=\sum_{I\in \child(Q)}\Big(\frac{\langle f\rangle_{I}}{\langle b\rangle_{I}}-\frac{\langle f\rangle_{Q}}{\langle b\rangle_{Q}}\Big)b\chi_{I}.
\end{equation}
\end{definition}

\begin{definition}[Adapted Haar wavelets]\label{defpsi}
Let $b$ be a locally integrable function with non-zero dyadic averages. 
For $I\in \mathcal D$, we remind the functions given in Definition \ref{Haar},
$h_{I}^{b}
=|I|^{\frac{1}{2}}\big(\frac{1}{|I|\langle b\rangle_{I}}\chi_{I}-\frac{1}{|I|\langle b\rangle_{I_{p}}}\chi_{I_{p}}\big)$. 

Then for $I\in \mathcal D$ we define the Haar wavelets adapted to $b$ 
and their corresponding dual wavelets as
\begin{equation*}
\psi_{I}^{b}=h_{I}^{b}b ,
\hskip40pt
\tilde{\psi}_{I}^{b}=h_{I}^{b}\langle b\rangle_{I}.
\end{equation*}

\end{definition}

We have the following result:
\begin{lemma}\label{adaptedwavelets} For every locally integrable function $f$, 
$$
\Delta_{Q}^{b}f
=\sum_{I\in \child(Q)}\langle f, \tilde{\psi}_{I}^{b}\rangle 
\psi_{I}^{b},
$$
where $\langle f,g\rangle =\int_{\mathbb R^{n}}f(x)g(x)dx$. 
\end{lemma}
\proof
A direct computation starting at \eqref{Deltaincoord} shows that 
$$
\Delta_{Q}^{b}f
=\sum_{I\in \child(Q)}\langle f\rangle_{I}\Big(\frac{1}{\langle b\rangle_{I}}\chi_{I}-\frac{|I|}{|Q|}\frac{1}{\langle b\rangle_{Q}}\chi_{Q}\Big)b
=\sum_{I\in \child(Q)}\langle f\rangle_{I}|I|^{\frac{1}{2}}\psi_{I}^{b}.
$$ 
Also from \eqref{Deltaincoord}, we have
\begin{equation}\label{Deltaaverage}
\langle \Delta_{Q}^{b}f\rangle_{I}=\langle f\rangle_{I}-\frac{\langle b\rangle_{I}}{\langle b\rangle_{Q}}\langle f\rangle_{Q}
\end{equation}
and 
since 
\begin{equation}\label{noLI}
\sum_{I\in \child(Q)}\frac{\langle b\rangle_{I}}{\langle b\rangle_{Q}}\langle f\rangle_{Q}
|I|^{\frac{1}{2}}\psi_{I}^{b}=0
\end{equation}
we get 
$$
\Delta_{Q}^{b}f
=\sum_{I\in \child(Q)}\langle \Delta_{Q}^{b}f\rangle_{I}
|I|^{\frac{1}{2}}\psi_{I}^{b}.
$$
Now, we use \eqref{Deltaaverage} to compute the coefficients and get:
$$
|I|^{\frac{1}{2}}\langle \Delta_{Q}^{b}f\rangle_{I}
=|I|^{\frac{1}{2}}\int f(x) \Big(\frac{\chi_{I}(x)}{|I|}-\frac{\langle b\rangle_{I}}{\langle b\rangle_{Q}}\frac{\chi_{Q}(x)}{|Q|}\Big)dx
=\langle f, \tilde{\psi}_{I}^{b}\rangle .
$$

%
\begin{remark}
From \eqref{noLI} or the dual equality
\begin{equation}\label{2^d-1}
\sum_{I\in \child(Q)}
\tilde\psi_{I}^{b}=0,
\end{equation}
we see that this wavelet system is not linearly independent.
\end{remark}

\begin{corollary} For ${\displaystyle (\Delta_{Q}^{b})^{*}f
=\sum_{I\in \child(Q)}\Big(\frac{\langle fb\rangle_{I}}{\langle b\rangle_{I}}-\frac{\langle fb\rangle_{Q}}{\langle b\rangle_{Q}}\Big)\chi_{I}}
$, we have
$$
(\Delta_{Q}^{b})^{*}f
=\sum_{I\in \child(Q)}\langle f, \psi_{I}^{b}\rangle
\tilde\psi_{I}^{b}.
$$
\end{corollary}

Next lemma states the orthogonality properties of the adapted Haar wavelets. The proof follows from direct calculations.

\begin{lemma}\label{psiortho} For $I,J\in \mathcal D$, 
$
\int \psi_{I}^{b}(x)dx=\int \tilde{\psi}_{I}^{b}(x)b(x)dx =0.
$
Moreover, 
\begin{equation}\label{psiortho1}
\langle \psi_{I}^{b},\tilde{\psi}_{J}^{b}\rangle =0
\end{equation}
when $I_{p}\neq J_{p}$, while for $I_{p}=J_{p}$, we have
\begin{equation}\label{psiortho2}
\langle \psi_{I}^{b},\tilde{\psi}_{J}^{b}\rangle =\delta (I-J)-\frac{|J|\langle b\rangle_{J}}{|I_{p}|\langle b\rangle_{I_{p}}}.
\end{equation}
where 
$\delta (I-J)=1$ if $I=J$ and zero otherwise. 

Finally,  
$
\| \tilde{\psi}_{I}^{b}\|_{L^{q}(\mathbb R^{n})}\lesssim C_{I}^{b}|\langle b\rangle_{I}||I|^{\frac{1}{q}-\frac{1}{2}}
$ and
$
\| \psi_{I}^{b}\|_{L^{q}(\mathbb R^{n})}
\lesssim B_{I,q}^{b}|I|^{\frac{1}{q}-\frac{1}{2}}
$.
\end{lemma}

\vskip10pt
The next result, which generalizes the classical Carleson's Embedding Theorem, is used in Lemma \ref{Planche} and Section \ref{paraproducts}. The proof follows from a direct adaptation of the demonstration included in \cite{AHMTT}.

\begin{lemma}[Carleson Embedding Theorem]\label{Carleson}
Let $(a_{I})_{I\in \mathcal D}$ a collection of non-negative numbers such that for all $I\in \mathcal D$,
\begin{equation}\label{Carlesonsequence}
\sum_{J\in \mathcal D(I)}a_{J}\lesssim [b]_{I,2}^{2}|I|.
\end{equation}
Then for every $f\in L^{2}(\mathbb R^{n})$,
${\displaystyle
\sum_{I\in \mathcal D}[b]_{I,2}^{-2}\, a_{I}
|\langle f\rangle_{I}|^{2}
\lesssim \| f\|_{L^{2}(\mathbb R^{n})}^{2}}
$.
\end{lemma}
\begin{remark}
For $(a_{I})_{I\in \mathcal D}$, $(b_{I})_{I\in \mathcal D}$ with $a_{I}$, $b_{I}$ non-negative,
\begin{align}\label{Carleson2} 
\sum_{I\in \mathcal D}a_{I}b_{I}
|\langle f\rangle_{I}|^{2}
\lesssim \sup_{I\subset \mathbb R^{n}}\Big(\frac{b_{I}}{|I|}\sum_{J\in \mathcal D(I)}a_{J}\Big)\, \| f\|_{L^{2}(\mathbb R^{n})}^{2}.
\end{align}

\end{remark}

\begin{lemma}\label{Planche} Let $b$ be a locally integrable function compatible with an operator and let $B\!F$ as stated in  
Definition \ref{accretive}. 
Then, 
$$
\sum_{I\in \mathcal D}
B\!F(I,J)
|\langle f,\tilde{\psi}_{I}^{b}\rangle |^{2}
\lesssim \| f\|_{L^{2}(\mathbb R^{n})}^{2}
$$
for every locally integrable function $f$ and every $J\in \mathcal D$.

Moreover, for $\epsilon>0$, there is $M_{0}\in \mathbb N$, such that for all $M>M_{0}$, 
$$
\sum_{I\in \mathcal D_{M}^{c}}
\sup_{\tiny \begin{array}{c}J\in  \mathcal D_{M}^{c}\\ 
(I,J)\in \mathcal F_{M}\end{array}}\hspace{-.5cm}
B\!F(I,J)
|\langle f,\tilde{\psi}_{I}^{b}\rangle |^{2}
\lesssim \epsilon \| f\|_{L^{2}(\mathbb R^{n})}^{2}
$$
for every locally integrable function $f$, where $\mathcal F_{M}$ is given after condition 
\eqref{b-averageandbound0limit} of Definition \ref{accretive}.
\end{lemma}
\begin{remark}\label{realPlanche} The proof shows that the following inequality also holds  
$$
\sum_{I\in \mathcal D}
\Big(\frac{[b]_{I_{p},2}}{|\langle b\rangle_{I_{p}}|}\Big)^{-2}
|\langle f,\tilde{\psi}_{I}^{b}\rangle |^{2}
\lesssim \| f\|_{L^{2}(\mathbb R^{n})}^{2}.
$$
\end{remark}
\proof
On the one hand, for $I\in \child(I_{p})$, by the definition of $\tilde{\psi}_{I}^{b}$ we have
\begin{align*}
|\langle f,\tilde{\psi}_{I}^{b}\rangle | &=|I|^{\frac{1}{2}}
|\langle b\rangle_{I}|\Big|\frac{\langle f\rangle_{I}}{\langle b\rangle_{I}}-\frac{\langle f\rangle_{I_{p}}}{\langle b\rangle_{I_{p}}}\Big|
\\
&
=|I|^{\frac{1}{2}}\Big|\langle f\rangle_{I}-\langle f\rangle_{I_{p}}-\frac{1}{\langle b\rangle_{I_{p}}}\langle f\rangle_{I_{p}}\big(\langle b\rangle_{I}-\langle b\rangle_{I_{p}}\big)\Big|
\\
&\leq \frac{[b]_{I_{p},2}}{|\langle b\rangle_{I_{p}}|}\Big(|I|^{\frac{1}{2}}|\langle f\rangle_{I}-\langle f\rangle_{I_{p}}|+|I|^{\frac{1}{2}}[b]_{I_{p},2}^{-1}\big|\langle b\rangle_{I}-\langle b\rangle_{I_{p}}\big| |\langle f\rangle_{I_{p}}|\Big),
\end{align*}
since $|\langle b\rangle_{I_{p}}|\leq [b]_{I_{p},2}$.
Now, by conditions \eqref{b-averageandbound0}, \eqref{b-averageandbound0limit} of Definition 
\ref{accretive} of a compatible testing function, we have that 
$
B\!F(I,J)
[b]_{I,2}^{2}/|\langle b\rangle_{I}|^{2}\lesssim C
$
for all $I,J\in \mathcal D$ and that given $\epsilon >0$, there is $M_{0}\in \mathbb N$ satisfying 
$
B\!F(I,J)
[b]_{I,2}^{2}/|\langle b\rangle_{I}|^{2}\lesssim \epsilon
 $
for all $M>M_{0}$ and $I,J\in \mathcal D_{M}^{c}$ such that $(I,J) \in \mathcal F_{M}$.
With this, we obtain
\begin{align*}
\sum_{I\in \mathcal D_{M}^{c}}&
\sup_{\tiny \begin{array}{c}J\in  \mathcal D_{M}^{c}\\ 
(I,J)\in \mathcal F_{M}\end{array}}\hspace{-.5cm}
B\!F(I,J)
|\langle f,\tilde{\psi}_{I}^{b}\rangle |^{2}
\\
&
\lesssim \epsilon \Big(
\sum_{I\in \mathcal D}
\big| \langle f\rangle_{I}-\langle f\rangle_{I_{p}}\big|^{2}|I|
+
\sum_{I\in \mathcal D}[b]_{I_{p},2}^{-2}\big| \langle b\rangle_{I}-\langle b\rangle_{I_{p}}\big|^{2}|I|\, |\langle f\rangle_{I_{p}}|^{2}
\Big)
\end{align*}
and the last expression is bounded by a constant times $\epsilon \| f\|_{L^{q}(\mathbb R^{n})}$ as we briefly indicate. 
The first term follows by the standard square function estimate. Moreover, the same square function estimate  
shows that 
$$
\sum_{J\in \mathcal D(I)}
\big| \langle b\rangle_{J}-\langle b\rangle_{J_{p}}\big|^{2}|J|
\leq 
\| b\|_{L^{2}(I_{p})}=[b]_{I_{p},2}|I_{p}|,
$$
which proves that $\big(
\sum_{I\in \child(Q)}|\langle b\rangle_{I}-\langle b\rangle_{Q}|^{2}|I|\big)_{Q \in \mathcal D}
$ satisfies hypothesis \eqref{Carlesonsequence} of Lemma 
 \ref{Carleson}. Then
$$
\sum_{Q\in \mathcal D}\sum_{I\in \child(Q)}[b]_{Q,2}^{-2}
\big| \langle b\rangle_{I}-\langle b\rangle_{Q}\big|^{2}|I||\langle f\rangle_{Q}|^{2}
\lesssim \| f\|_{2}.
$$
\begin{corollary}\label{Planchedual}The following dual statement also holds:
$$
\sum_{I\in \mathcal D}
\Big(\frac{[b]_{I_{p},2}}{|\langle b\rangle_{I}||\langle b\rangle_{I_{p}}|}\Big)^{-2}
|\langle f,\psi_{I}^{b}\rangle |^{2}
\lesssim \| fb\|_{L^{2}(\mathbb R^{n})}.
$$
\end{corollary}
\proof Since 
$
|\langle f,\psi_{I}^{b}\rangle | 
=|\langle b\rangle_{I}|^{-1}|\langle fb,\tilde{\psi}_{I}^{b}\rangle | 
$,
we can apply Remark \ref{realPlanche}.

\begin{lemma}\label{densityinL2}
Let $b$ be a locally integrable function. 
Then the equality 
\begin{equation}\label{representationoff}
f=\sum_{I\in \mathcal D}
\langle f,\tilde{\psi}_{I}^{b}\rangle \psi_{I}^{b}
\end{equation}
holds pointwise a.e. 
almost everywhere for 
$f$ integrable, compactly supported and with mean zero. 

\end{lemma}
%
\proof
By Lemma \ref{adaptedwavelets} we have $\Delta_{Q}^{b}f=\sum_{I\in \child(Q)}\langle f,\tilde{\psi}_{I}^{b}\rangle \psi_{I}^{b}$. Then
the right hand side of \eqref{representationoff} is understood as
$$
\lim_{M\rightarrow \infty }
\sum_{\tiny \begin{array}{c}I\in \mathcal D\\ 2^{-M}< \ell(I)\leq 2^{M}\end{array}}
\langle f,\tilde{\psi}_{I}^{b}\rangle \psi_{I}^{b}
=\lim_{M\rightarrow \infty }\sum_{-M< k\leq M}\Delta_{k}^{b}f .
$$
We choose $R\in {\mathcal D}$ with $\sup f\subset R$, and $M\in \mathbb N$ with
$2^{-M}< \ell(R)< 2^{M}$. For every $x\in R$, we select $I,J\in \mathcal D$ such that $x\in J\subset I$, 
$\ell(J)=2^{-M}$ and $\ell(I)=2^{M}$. Since $R\subseteq I$ and $f$ has zero mean, then $\langle f\rangle_{I}=0$. 
With this, by summing a telescopic series, we have
\begin{align}\label{pointconv}
\sum_{-M<k\leq M}\hspace{-.5cm}\Delta_{k}^{b}f(x)
&=E_{M}^{b}f(x)-E_{-M}^{b}f(x)
=\frac{\langle f\rangle_{J}}{\langle b\rangle_{J}}\chi_{J}(x)b(x).
\end{align}
Now, 
since $f$ and $b$ are both locally integrable, by Lebesgue's Differentiation Theorem, the right hand side of \eqref{pointconv} tends to $f(x)$
pointwise almost everywhere when $M$ tends to infinity.

By a similar reasoning, we can prove the following dual result:
\begin{lemma}\label{densityinL2dual}
Let $b$ be a locally integrable function. 
Then the equality 
\begin{equation}\label{representationoff2}
f=\sum_{I\in \mathcal D}\langle f,\psi_{I}^{b}\rangle \tilde\psi_{I}^{b}
\end{equation}
holds pointwise almost everywhere for 
$f$ integrable, compactly supported and with mean zero with respect to $b$. 

\end{lemma}

\begin{lemma}\label{orthorep} Let $T$ be a bounded operator on $L^{2}(\mathbb R^{n})$ with compact Calder\'on-Zygmund kernel $K$.  Let $b_{i}$ be two locally integrable functions compatible with $T$ and 
$(\psi_{I}^{b_{i}})_{I\in {\mathcal D}}$ 
be the wavelet systems
of Definition \ref{defpsi}. 
Then for $f,g$ locally integrable,
\begin{align}\label{generaldec}
\langle Tf,g\rangle =\sum_{I,J\in \mathcal D}
\langle f,\tilde{\psi}_{I}^{b_{1}}\rangle 
\langle g,\tilde{\psi}_{J}^{b_{2}}\rangle  \langle T\psi_{I}^{b_{1}}, \psi_{J}^{b_{2}}\rangle ,
\end{align}
\begin{align}\label{generaldec2}
\langle T_{b}f,g\rangle =\sum_{I,J\in \mathcal D}
\langle f,\psi_{I}^{b_{1}}\rangle 
\langle g,\psi_{J}^{b_{2}}\rangle  \langle T_{b}\tilde{\psi}_{I}^{b_{1}}, \tilde{\psi}_{J}^{b_{2}}\rangle .
\end{align}

\end{lemma}
We note that the lack of accretivity is the reason for the unusual definition of weak compactness (Definition \ref{weakcompactnessdef}) and the extra work required to prove Lemma \ref{orthorep}. 
When the testing functions are accretive,
the lemma follows directly from convergence of a wavelet frame on $L^{2}(\mathbb R^{n})$ and the continuity of $T$. However, without accretivity 
the chosen wavelet system does not converge on $L^{2}(\mathbb R^{n})$.  
Moreover, one can not use the classical $Tb$ Theorem to deduce that $T$ is already known to be bounded because in general the testing functions are not accretive. 

\proof[Proof (of Lemma \ref{orthorep})] 
We only show \eqref{generaldec}. 
Let $(h_{I}^{1})_{I}$ be the Haar-wavelet system.
By Lemma  \ref{densityinL2} applied to the accretive functions $b_{i}=1$ and the continuity of $T$ on $L^{2}(\mathbb R^{n})$,
\begin{align*}
\langle Tf,g\rangle =\sum_{I,J\in \mathcal D}
\langle f,h_{I}^{1}\rangle 
\langle g,h_{J}^{1}\rangle  \langle Th_{I}^{1}, h_{J}^{1}\rangle 
=\lim_{N\rightarrow \infty }\langle Tf_{N},g_{N}\rangle
\end{align*}
with $f_{N}\!=\!\sum_{I\in {\mathcal D}_{N}}\langle f, h_{I}^{1}\rangle h_{I}^{1}$ and similar for $g_{N}$.
Then, we can assume 
$f,g$ to be in the unit ball of $L^{2}(\mathbb R^{n})$, supported on $Q\in \mathcal D$ with $\ell(Q)>1$, constant on dyadic cubes of side length 
$\ell \leq 1$ and with mean zero. 

As shown in the proof of Lemma \ref{densityinL2}, we have
$$
\sum_{\tiny \begin{array}{c}I\in \mathcal D\\ 2^{-M}\leq \ell(I)< 2^{M}\end{array}}
\langle f,\tilde{\psi}_{I}^{b}\rangle \psi_{I}^{b}
=\sum_{-M<k\leq M}\Delta_{k}^{b}f =E_{M}^{b}f,
$$
and similar expression for $g$. We then need to prove that 
$
|\langle Tf,g\rangle-\langle T(E_{M}^{b_{1}}f),E_{M}^{b_{2}}g\rangle |
$
tends to zero when $M$ tends to infinity uniformly for functions $f, g$ in the unit ball of $L^{2}(\mathbb R^{n})$.  We bound this difference by
$$
|\langle T(f-E_{M}^{b_{1}}f),g\rangle |+|\langle T(E_{M}^{b_{1}}f),g-E_{M}^{b_{2}}g\rangle| 
=S_{1}+S_{2}
$$
and we only work to estimate the second term. 
%
As explained before, $g-E_{M}^{b_{2}}g$ does not necessarily tend to zero on $L^{2}(\mathbb R^{n})$. 

By condition \eqref{b-averageandbound0limit} of Definition \ref{accretive}, given $\epsilon >0$ there exists 
$M_{0}\in \mathbb N$ such that $2^{-M_{0}}<\ell$ and for all $M>M_{0}$,
\begin{equation}\label{appliedb-averageandbound0limit}
\hspace{-.6cm}\sup_{\tiny \begin{array}{c}I,J
\in {\mathcal D}\\\ell(I\smland J)<2^{-(M-1)}\end{array}}
\hspace{-.6cm} B\!F(I,J)
\leq \hspace{-.6cm} \sup_{\tiny \begin{array}{c}(I,J)\in \mathcal F_{M-1}
\end{array}}
\hspace{-.6cm} B\!F(I,J)
\Big(\frac{[b_{1}]_{I,q_{1}}^{2}}{|\langle b_{1}\rangle_{I}|^{2}}+\frac{[b_{2}]_{J,q_{2}}^{2}}{|\langle b_{2}\rangle_{J}|^{2}}\Big)
<\epsilon ,
\end{equation}
where $\mathcal F_{M}$ is the family of ordered pairs of cubes $I,J\in {\mathcal D}_{M}^{c}$ with either $\ell(I\smland J)>2^{M}$, $\ell(I\smland J)<2^{-M}$ or $\rdist(\langle I,J\rangle ,\mathbb B)>M^{\theta}$ for $\theta \in (0,1)$.


We fix now $M>M_{0}$. Let $(I_{i})_{i\in \mathbb Z_{M}^{n}}\subset \mathcal D$ a partition of $Q$ with 
$\mathbb Z_{M}^{n}=\{ i\in \mathbb N^{n}: \|i\|_{\infty }\leq 2^{M}\ell(Q)\}$ such that $\ell(I_{i})=2^{-M}$ and the cubes $I_{i}$ are enumerated so that $\dist(I_{i},I_{j})=\max(\|i-j\|_{\infty }-1,0)2^{-M}$. 
Then, 
$$
S_{2}
\leq \sum_{i,j\in \mathbb Z_{M}^{n}}|\langle T(E_{M}^{b_{1}}f\chi_{I_{i}}),(g-E_{M}^{b_{2}}g)\chi_{I_{j}}\rangle|
=\sum_{i,j\in \mathbb Z_{M}^{n}}T_{i,j}
$$
We consider the cases $\dist(I_{i},I_{j})>0$ or $\dist(I_{i},I_{j})=0$. In the first one, 
by the integral representation of $T$ and the mean zero of $(g-E_{M}^{b}g)\chi_{I_{j}}$,
$$
T_{i,j}=\Big|\int_{I_{j}}\int_{I_{i}}\frac{\langle f\rangle_{I_{i}}}{\langle b_{1}\rangle_{I_{i}}}b_{1}(t)
(g(x)-\frac{\langle g\rangle_{I_{j}}}{\langle b_{2}\rangle_{I_{j}}}b_{2}(x))(K(t,x)-K(t,c(I_{j})))dtdx\Big|
$$
Since $\dist(I_{i},I_{j})>0$ and $\ell(I_{i})=\ell(I_{j})$, we have that $\dist(I_{i},I_{j})\geq \ell(I_{j})$. Then, 
for all $t\in I_{i}$ and $x\in I_{j}$ we have $2|x-c(I_{j})|_{\infty }\leq \ell(I_{j})\leq \dist(I_{i},I_{j})< |t-c(I_{j})|_{\infty }$. With this, 
$$
|K(t,x)-K(t,c(I_{j})))|\leq \frac{|x-c(I_{j})|^{\delta }}{|t-c(I_{j})|_{\infty }^{n+\delta}}
F_{K}(t,x,c(I_{j}))
$$
with 
$F_{K}(t,x,c(I_{j}))=L(|t-c(I_{j})|_{\infty })
S(|x-c(I_{j})|_{\infty })D\Big(1+\frac{|t+c(I_{j})|_{\infty }}{1+|t-c(I_{j})|_{\infty }}\Big)$.
Now, the properties of $I_{i},I_{j}$ also imply  
$\dist(I_{i},I_{j})\geq 
\ell(\langle I_{i},I_{j}\rangle )/3$. 
With this, we get the inequalities: $|t-c(I_{j})|_{\infty }> 
\ell(\langle I_{i},I_{j}\rangle )/3$, $|x-c(I_{j})|_{\infty }\leq \ell(I_{j})/2$ and $|c(I_{i})-c(\langle I_{i},I_{j}\rangle )|_{\infty }\leq \ell(\langle I_{i},I_{j}\rangle )/2$. Then,  
by the proof of Lemma \ref{boundtoFK}, we have in the domain of integration 
\begin{align*}
F_{K}(t,x,c(J))&\lesssim L(\ell(\langle I_{i},I_{j}\rangle ))
S(\ell(I_{j}))D\Big(1+\frac{|c(I_{j})|_{\infty }}{1+\ell(\langle I_{i},I_{j}\rangle )}\Big) 
\\
&\lesssim F_{K}(\langle I_{i},I_{j}\rangle,I_{j},\langle I_{i},I_{j}\rangle) =: F_{K}(i,j).
\end{align*}
Since $g$ is constant on $I_{j}$, we have $g\chi_{I_{j}}=\langle g\rangle_{I_{j}}\chi_{I_{j}}$ and so, for all $x\in I_{j}$, 
$$
g(x)-\frac{\langle g\rangle_{I_{j}}}{\langle b_{2}\rangle_{I_{j}}}b_{2}(x)
= \frac{\langle g\rangle_{I_{j}}}{\langle b_{2}\rangle_{I_{j}}}
(\langle b_{2}\rangle_{I_{j}}-b_{2}(x))
$$

Moreover, since $\frac{\ell(\langle I_{i},I_{j}\rangle)}{\ell(I_{j})}=\rdist(I_{i},I_{j})=\frac{\ell(I_{i})+\dist(I_{i},I_{j})}{\ell(I_{i})}=
\|i-j\|_{\infty }$, 
\begin{align*}
T_{i,j}&
\leq \frac{|\langle f\rangle_{I_{i}}|}{|\langle b_{1}\rangle_{I_{i}}|}
\|b_{1}\|_{L^{1}(I_{i})}
\frac{|\langle g\rangle_{I_{j}}|}{|\langle b_{2}\rangle_{I_{j}}|} (|\langle b_{2}\rangle_{I_{j}}||I_{j}|+\| b_{2}\|_{L^{1}(I_{j})})
\frac{\ell(I_{j})^{\delta }}{\ell(\langle I_{i},I_{j}\rangle)^{n+\delta}}F_{K}(i,j)
\\
&\lesssim   |\langle f\rangle_{I_{i}}||\langle g\rangle_{I_{j}}||I_{i}|
\frac{\ell(I_{j})^{n+\delta }}{\ell(\langle I_{i},I_{j}\rangle)^{n+\delta}}\frac{[b_{1}]_{I_{i},1}}{|\langle b_{1}\rangle_{I_{i}}|}
\Big(1+\frac{[b_{2}]_{I_{j},1}}{|\langle b_{2}\rangle_{I_{j}}|}\Big)
F_{K}(i,j)
\\
&\lesssim \epsilon |\langle f\rangle_{I_{i}}||\langle g\rangle_{I_{j}}||I_{i}|
\|i-j\|_{\infty }^{-(n+\delta)}
\end{align*}
The last inequality is due to 
 $\ell(I_{j})=2^{-M}$,
Definition \ref{accretive} and \eqref{appliedb-averageandbound0limit}: 
\begin{align*}
\frac{[b_{1}]_{I_{i},1}}{|\langle b_{1}\rangle_{I_{i}}|}
&\Big(1+\frac{[b_{2}]_{I_{j},1}}{|\langle b_{2}\rangle_{I_{j}}|}\Big)
F_{K}(i,j)
\leq B\!F(I_{i},I_{j})
\leq \hspace{-.6cm} \sup_{\tiny \begin{array}{c}I,J \in {\mathcal D}\\ \ell(I\smland J)<2^{-(M-1)}\end{array}}\hspace{-.8cm}
B\!F(I,I)
\lesssim \epsilon
\end{align*}
Therefore, the corresponding sum can be bounded as follows:

\begin{align*}
\sum_{
i,j\, :\, \|i-j\|_{\infty }\geq 1
}\hspace{-.1cm}T_{i,j}
&\lesssim \epsilon
\hspace{-.1cm}\sum_{
i,j\,:\,\|i-j\|_{\infty }\geq 1
}\hspace{-.1cm} |\langle f\rangle_{I_{i}}||\langle g\rangle_{I_{j}}||I_{i}|
\|i-j\|_{\infty }^{-(n+\delta)}
\\
&\leq \epsilon \Big(\sum_{i\in \mathbb Z_{M}^{n}}|\langle f\rangle_{I_{i}}|^{2}|I_{i}|
\sum_{
j\,:\,\|i-j\|_{\infty }\geq 1
}\|i-j\|_{\infty }^{-(n+\delta)}\Big)^{\frac{1}{2}}
\\
&\hskip20pt \Big(\sum_{j\in \mathbb Z_{M}^{n}}|\langle g\rangle_{I_{j}}|^{2}|I_{j}|
\sum_{
i\,:\,\|i-j\|_{\infty }\geq 1
}\|i-j\|_{\infty }^{-(n+\delta)}\Big)^{\frac{1}{2}}
\\
&\lesssim \epsilon \|f\|_{L^{2}(\mathbb R^{n})}\|g\|_{L^{2}(\mathbb R^{n})}\leq \epsilon 
\end{align*}

On the other hand, when $\dist(I_{i},I_{j})=0$, $I_{i}\neq I_{j}$, 
we write 
\begin{align*}
T_{i,j}&=\Big|\Big\langle T\Big(\frac{\langle f\rangle_{I_{i}}}{\langle b_{1}\rangle_{I_{i}}}b_{1}\chi_{I_{i}}\Big), 
 \frac{|\langle g\rangle_{I_{j}}|}{|\langle b_{2}\rangle_{I_{j}}|}(\langle b_{2}\rangle_{I_{j}}-b_{2})\chi_{I_{j}}\Big\rangle \Big|
\\
&\leq  |\langle f\rangle_{I_{i}}||\langle g\rangle_{I_{j}}|
\Big(\frac{1}{|\langle b_{1}\rangle_{I_{i}}}|\langle T(b_{1}\chi_{I_{i}}),\chi_{I_{j}}\rangle | 
+\frac{1}{|\langle b_{1}\rangle_{I_{i}}||\langle b_{2}\rangle_{I_{j}}|} |\langle T_{b}\chi_{I_{i}},\chi_{I_{j}}\rangle |\Big)
\end{align*}
By Lemma \ref{adjacentsquares}, we have that the terms inside the parentheses can be bounded by a constant times
\begin{align*}
|I_{i}|  \frac{[b_{1}]_{I_{i},q_{1}}}{|\langle b_{1}\rangle_{I_{i}}|} \Big(1+ \frac{[b_{2}]_{I_{j},q_{2}}}{|\langle b_{2}\rangle_{I_{j}}|}\Big) F^{L}_{K}(I_{i})
\lesssim |I_{i}| B\!F(I_{i},I_{j})\lesssim \epsilon |I_{i}|, 
\end{align*}
where the last inequality is due to \eqref{appliedb-averageandbound0limit}  and the fact that $\ell(I_{i})=2^{-M}$.

%
Since for each fixed index $i$ there are only $3^{n}-1$ indexes $j$ such that $\dist(I_{i},I_{j})=0$, $I_{i}\neq I_{j}$, 
the corresponding sum can be bounded by
\begin{align*}
\sum_{
i,j\, :\,\|i-j\|_{\infty }=1
}T_{i,j}
&\lesssim \epsilon \sum_{
i,j\, :\,\|i-j\|_{\infty }=1
}|\langle f\rangle_{I_{i}}||\langle g\rangle_{I_{j}}| |I_{j}|
\\
&\lesssim \epsilon\Big(\sum_{i\in \mathbb Z_{M}^{n}} |\langle f\rangle_{I_{i}}|^{2}|I_{i}|\Big)^{\frac{1}{2}}
\Big(\sum_{j\in \mathbb Z_{M}^{n}} |\langle g\rangle_{I_{j}}|^{2}|I_{j}|\Big)^{\frac{1}{2}}\leq \epsilon
\end{align*}

Finally, when $I_{i}=I_{j}$, we have similarly as before:
\begin{align*}
T_{i,i}&
= \frac{|\langle f\rangle_{I_{i}}||\langle g\rangle_{I_{i}}|}{|\langle b_{1}\rangle_{I_{i}}||\langle b_{2}\rangle_{I_{i}}|}
|\langle T(b_{1}\chi_{I_{i}}),(\langle b_{2}\rangle_{I_{i}}-b_{2})\chi_{I_{i}}\rangle |
\end{align*}

By weak compactness, 
$|\langle T(b_{1}\chi_{I_{i}}),b_{2}\chi_{I_{i}}\rangle|\lesssim |I_{i}|[b_{1}]_{I_{i},q_{1}}[b_{2}]_{I_{i},q_{2}}F_{W}(i,i)$
and $|\langle T(b_{1}\chi_{I_{i}}),\chi_{I_{i}}\rangle|\lesssim |I_{i}|[b_{1}]_{I_{i},q_{1}}F_{W}(i,i)$,
where now we write $F_{W}(i,i)=F_{W}(I_{i};M)+\epsilon $. Then, by \eqref{appliedb-averageandbound0limit} and $\ell(I_{i})=2^{-M}$, we get
\begin{align*}
\sum_{i\in \mathbb Z_{M}^{n}}T_{i,i}
&\lesssim
\sum_{i\in \mathbb Z_{M}^{n}}|\langle f\rangle_{I_{i}}||\langle g\rangle_{I_{i}}||I_{i}|\frac{[b_{1}]_{I_{i},q_{1}}}{|\langle b_{1}\rangle_{I_{i}}|}\frac{[b_{2}]_{I_{i},q_{2}}}{|\langle b_{2}\rangle_{I_{i}}|}F_{W}(i,i)
\\
&\lesssim \sum_{i\in \mathbb Z_{M}^{n}} |\langle f\rangle_{I_{i}}||\langle g\rangle_{I_{i}}||I_{i}|B\!F(I_{i},I_{i})\lesssim \epsilon
\end{align*}




\begin{corollary}\label{repofproj} With the same hypotheses of Lemma \ref{orthorep}, let 
$P_{1,M}$, $P_{2,M}$ be the projections related to each system. 
Then 
$$
\langle (P_{2,M}^{*})^{\perp}TP_{1,M}^{\perp}f,g\rangle =\sum_{I,J\in \mathcal D_{M}^{c}}
\langle f,\tilde{\psi}_{I}^{b_{1}}\rangle 
\langle g,\tilde{\psi}_{J}^{b_{2}}\rangle  \langle T\psi_{I}^{b_{1}}, \psi_{J}^{b_{2}}\rangle .
$$
The dual representation for $T_{b}$ also holds. 

\end{corollary}
\proof We have that 
\begin{align*}
\langle &(P_{2,M}^{*})^{\perp}TP_{1,M}^{\perp}f,g\rangle 
=\langle TP_{1,M}^{\perp}f,P_{2,M}^{\perp}g\rangle
\\
&=\langle Tf,g\rangle 
-\langle Tf,P_{2,M}g\rangle 
-\langle TP_{1,M}f,g\rangle 
+\langle TP_{1,M}f,P_{2,M}g\rangle 
\end{align*}
and by \eqref{generaldec} the last expression coincides with the statement.

\section{$L^{p}$ compactness}\label{L2}
We start this section with the following technical result:
\begin{lemma}\label{smallF} Let $\tilde{L}, S, \tilde{D}$ be the functions of Definition \ref{LSDF} and let 
$
F(I, J) \! = \!
\tilde F_{K}(I\smland J, I\smland J,  \langle I,J\rangle )
+F_{W}(I;M_{T,\epsilon }) \chi_{I=J}$.
Given $\epsilon >0$, let $M>0$ be large enough so that
$\tilde{L}(2^{M})+S(2^{-M})+\tilde{D}(M^{\frac{1}{8}})<\epsilon $. 

Then
for all 
$I\in {\mathcal D}_{2M}^{c}$ and $J\in {\mathcal D}_{M}^{c}$ 
we have that: either $F(I,J)<\epsilon $, $|\log(\ec(I,J))|\gtrsim \log M$, or $\rdist(I, J)\gtrsim M^{\frac{1}{8}}$. 
\end{lemma}
\begin{remark} As showed in the proof, 
$F(I,J)<\epsilon $ holds when either
$\ell(I\smland J)>2^{M}$, or
$\ell(I\smland J)<2^{-M}$, or $\rdist(\langle I,J\rangle ,\mathbb B)>M^{1/8}$. For this reason, in Definition \ref{accretive}
we denote by $\mathcal F_{M}$ the family of
ordered pairs $(I,J)$ with $I,J\in \mathcal D_{M}^{c}$ satisfying some of these three inequalities.
\end{remark}
\begin{proof} We start with $F_{W}(I;M_{T,\epsilon }) \chi_{I=J}$, for which the proof is simpler.  Since  
$I\smlor J=I\smland J=\langle I,J\rangle =I=J\in \mathcal D_{M}^{c}$, we study three cases:
\vskip5pt
\noindent
a) When $\ell(I)<2^{-M}$, we have 
$F_{W}(I;M_{T,\epsilon })\lesssim S(\ell(I))\leq S(2^{-M})<\epsilon$.
\vskip5pt
\noindent
b) When $\ell(I)>2^{M}$, we get
$F_{W}(I;M_{T,\epsilon })\lesssim \tilde{L}(\ell(I))\leq \tilde{L}(2^{M})<\epsilon$.
\vskip5pt
\noindent
c) When $2^{-M}\leq \ell(I)\leq 2^{M}$ with $\rdist(I,\mathbb B_{2^{M}})>2M$, we finally get
$F_{W}(I;M_{T,\epsilon })\lesssim \tilde{D}(\rdist(I,\mathbb B_{2^{M}}))\leq \tilde{D}(2^{M})<\epsilon$.

\vskip8pt
We continue with $\tilde F_{K}$.
Since $I\in {\mathcal D}_{2M}^{c}$, we consider three cases: 
\vskip5pt
\noindent
a) When $\ell(I)<2^{-2M}$, we have $\ell(I\smland J)<2^{-2M}$ and so, we get\\
$\tilde F_{K}(I\smland J, I\smland J,  \langle I,J\rangle )\lesssim S(\ell(I\smland J))\leq S(2^{-2M})<\epsilon$.
\vskip5pt
\noindent
b) When $\ell(I)>2^{2M}$, since $J\in {\mathcal D}_{M}^{c}$ we distinguish two cases:
\begin{itemize}
\item[b.1)] When $\ell(J)>2^{M}$, we get $\ell(I\smland J)>2^{M}$ and so, we obtain 
$\tilde F_{K}(I\smland J, I\smland J,  \langle I,J\rangle )\lesssim \tilde L(\ell(I\smland J))<\tilde L(2^{M})<\epsilon$.
\item[b.2)] When 
$\ell(J)\leq 2^{M}$,
we have that 
$$
\ec(I,J)=\frac{\ell(I\smland J)}{\ell(I\smlor J)}=\frac{\ell(J)}{\ell(I)}<\frac{2^{M}}{2^{2M}}=2^{-M}.
$$
\end{itemize}
\vskip5pt
\noindent
c) When $2^{-2M}\leq \ell(I)\leq 2^{2M}$ with $\rdist(I,\mathbb B_{2^{2M}})>2M$, we 
have $|c(I)|_{\infty }>(2M-1)2^{2M}$. We fix $\alpha =\frac{1}{8}$, $\beta =\gamma =\frac{1}{4}$. Then,
\begin{itemize}
\item[c.1)] When $\ell(J)>(2M)^{\alpha}2^{2M}$, since $\alpha >0$ we have 
$$\ec(I,J)=\frac{\ell(I)}{\ell(J)}<\frac{2^{2M}}{(2M)^{\alpha }2^{2M}}\lesssim 
M^{-\frac{1}{8}}.$$ 

\item[c.2)] When $\ell(J)\leq (2M)^{\alpha }2^{2M}$, we have $\ell(I\smlor J)<(2M)^{\alpha }2^{2M}$. Now:
\vskip5pt
\noindent c.2.1) When $\rdist(\langle I,J\rangle,\mathbb B 
)>(2M)^{\beta }$, we obtain 
\begin{align*}
\hskip30pt \tilde F_{K}&(I\smland J, I\smland J,  \langle I,J\rangle )\lesssim \tilde D(\rdist(\langle I,J\rangle,\mathbb B))
<
\tilde D(M^{\frac{1}{8}})<\epsilon.
\end{align*}

\noindent c.2.2) When $\rdist(\langle I,J\rangle,\mathbb B_{2^{(2M)^{\beta }}})\leq (2M)^{\beta }$, we get 
$|c(\langle I,J\rangle)|_{\infty }\leq (2M)^{\beta }(
1+\ell(\langle I,J\rangle))$. Then, we examine the last two cases:
\begin{itemize}
\item When $\ell(\langle I,J\rangle)>(2M)^{\gamma}2^{2M}$,  
we get
$$
\hspace{.5cm}\rdist(I,J)=\frac{\ell(\langle I,J\rangle)}{\ell(I\smlor J)}>\frac{(2M)^{\gamma}2^{2M}}{(2M)^{\alpha }2^{2M}}
\gtrsim M^{\gamma -\alpha }=M^{\frac{1}{8}}.
$$

\item When $\ell(\langle I,J\rangle)\leq (2M)^{\gamma }2^{2M}$, we have instead
\begin{align*}
\hspace{1.5cm}
&|c(I)-c(J)|_{\infty }>|c(I)|_{\infty }-|c(\langle I,J\rangle)-c(J)|_{\infty }-|c(\langle I,J\rangle)|_{\infty }
\\
&\geq |c(I)|_{\infty }-2^{-1}\ell(\langle I,J\rangle)-(2M)^{\beta }(
1+\ell(\langle I,J\rangle))
\\
&\geq (2M-1)2^{2M}-(2M)^{\gamma }2^{2M}-(2M)^{\beta }(
1+(2M)^{\gamma }2^{2M})
\\
&\gtrsim (M\!-\!M^{\gamma}\!-\!M^{\beta}\!-\!M^{\beta +\gamma })2^{2M}
\gtrsim (M\!-\!3M^{\frac{1}{2}})2^{2M}\geq 2^{-1}M2^{2M}
\end{align*}
\hspace{-.5cm} for $M\geq 36$. Whence, 
\begin{align*}
\hspace{1.5cm}
\rdist(I,J)&\geq \frac{|c(I)-c(J)|_{\infty }}{\ell(I\smlor J)}
\gtrsim  \frac{M2^{2M}}{(2M)^{\alpha }2^{2M}}\gtrsim M^{1-\alpha }
=M^{\frac{7}{8}}.
\end{align*}
\end{itemize}

\end{itemize}
\end{proof}

We now demonstrate our main result on compactness of singular integral operators
when the special cancellation conditions hold.

\begin{theorem}
\label{L2bounds}
Let
$T$ 
be a linear operator bounded on $L^{2}(\mathbb R^{n})$ 
with a compact Calder\'on-Zygmund kernel.
Let $b_{1}$, $b_{2}$ be locally integrable functions compatible with $T$. We assume that $T$ satisfies 
the weak compactness condition
and 
$Tb_{1}=T^{*}b_{2}=0$.

Then $T$ can be extended to a compact operator on $L^2(\mathbb R^{n})$.
\end{theorem}
\proof
Let $(\psi_{I}^{b_{i}})_{I\in {\mathcal D}}$ 
be the Haar wavelet systems of Definition \ref{defpsi}  and $P_{i,M}$ be the projections associated with each system. 
By the comments after Remark \ref{actionproduct}, to prove compactness of $T$ on $L^{2}(\mathbb R^{n})$
it is enough to show that 
$
\langle (P_{2,M}^{*})^{\perp}TP_{1,M}^{\perp}f,g\rangle 
$
tends to zero uniformly for $f,g$ 
in the unit ball of 
$L^{2}(\mathbb R^{n})$. 
From now we write both projections as $P_{M}^{\perp}, (P_{M}^{*})^{\perp}$.

By Lemma \ref{smallF}, given $\epsilon >0$, there exists $M_{0}\in \mathbb N$ so that 
$F(I,J)<\epsilon $, $|\log(\ec(I,J))|\gtrsim \log M$, or $\rdist(I, J)\gtrsim M^{\frac{1}{8}}$. 

Let now $B\!F:\mathcal D\times \mathcal D\rightarrow [0,\infty ]$ be as given in Definition \ref{accretive}.
By Lemma \ref{Planche}, there exists $M_{1}\in \mathbb N$, $M_{1}>M_{0}$ so that $M_{1}>3^{8}$ and 
\begin{equation}\label{useofPlanche0}
\Big(\sum_{I\in {\mathcal D}}
B\!F(I,J)
|\langle f,\tilde{\psi}_{I}^{b_{1}}\rangle |^{2}
\Big)^{1/2}
\lesssim \| f\|_{L^{2}(\mathbb R^{n})}
\end{equation}
for all $J\in {\mathcal D}$ and 
\begin{equation}\label{useofPlanche}
\Big(\sum_{I\in {\mathcal D}_{M}^{c}}
\hspace{-.2cm}\sup_{\tiny \begin{array}{c}J\in  \mathcal D_{M}^{c}\\
(I,J)\in \mathcal F_{M}
\end{array}}\hspace{-.5cm}B\!F(I,J)
|\langle f,\tilde{\psi}_{I}^{b_{1}}\rangle |^{2}
\Big)^{1/2}
\lesssim \epsilon \| f\|_{L^{2}(\mathbb R^{n})}
\end{equation}
for all $M>M_{1}$. Similarly 
for $b_{2}$ and $g$. 

Now, for fixed $\epsilon >0$ and the chosen $M_{1}\in \mathbb N$, we prove that for  
$M>M_{1}$ such that $2M^{-\frac{\delta }{8}}+M^{-\frac{\delta }{1+2\delta}}<\epsilon $, 
we have 
$$
|\langle (P_{M}^{*})^{\perp}TP_{2M}^{\perp }f,g\rangle |\lesssim \epsilon .
$$

By Corollary \ref{repofproj},  
\begin{align}\label{orthopro}
\langle (P_{M}^{*})^{\perp}TP_{2M}^{\perp }f,g\rangle =
\sum_{I\in {\cal D}_{2M}^{c}}\sum_{J\in {\cal D}_{M}^{c}}
\langle f,\tilde{\psi}_{I}^{b_{1}}\rangle 
\langle g, \tilde{\psi}_{J}^{b_{2}}\rangle
\langle T\psi_I^{b_{1}},\psi_J^{b_{2}}\rangle .
\end{align}

According to Proposition \ref{twobumplemma1},  
we parametrize the sums by eccentricity, relative distance and inner relative distance of the cubes $I,J$ as follows.
For fixed $e\in \mathbb Z$, $m\in \mathbb N$ and $J\in \mathcal D$,
we define the family
$$
J_{e,m}=\{ I\in {\mathcal D}:\ell(I)=2^{e}\ell(J), m\leq  \rdist(I,J)< m+1 \} .
$$
When $m\leq 3$, we define for every $1\leq k\leq 2^{-\min(e,0)}$, 
$$
J_{e,m, k}=J_{e,m}\cap \{ I\in {\mathcal D}:k\leq \inrdist(I,J)< k+1 \} .
$$

The cardinality of 
$J_{e,m}$ is comparable to $2^{-\min(e,0)n}n(2m)^{n-1}$. On the other hand, when $m\leq 3$, the cardinality of 
$J_{e,m,k}$ is comparable to $n(2^{-\min(e,0)}-k)^{n-1}$.
Moreover, by symmetry, the family $\{ (I,J): I\in J_{e,m}\}$ can be  parametrized as 
$\{ (I,J): J\in I_{-e,m}\}$ and, similarly,  $I\in J_{e,m,k}$ if and only if $J\in I_{-e,m,k}$.
With all this, \eqref{orthopro} equals
\begin{align*}
\sum_{e\in \mathbb Z}\sum_{m, k\in \mathbb N}
\sum_{\tiny \begin{array}{c}J{\in \cal D}_{M}^{c}\end{array}}
\sum_{I\in J_{e,m,k}\cap {\cal D}_{2M}^{c}}
\langle f,\tilde{\psi}_{I}^{b_{1}}\rangle 
\langle g, \tilde{\psi}_{J}^{b_{2}}\rangle
\langle T\psi_I^{b_{1}},\psi_J^{b_{2}}\rangle , 
\end{align*}
Notice the abuse of notation since the written sums actually mean:
\begin{align*}
\sum_{e\in \mathbb Z}
\Big( \sum_{m\geq 4}
\sum_{\tiny \begin{array}{c}J{\in \cal D}_{M}^{c}\end{array}}\hspace{-.2cm}
\sum_{I\in J_{e,m}\cap {\cal D}_{2M}^{c}}
\hspace{-.1cm}+\hspace{-.1cm}\sum_{1\leq m\leq 3}\sum_{1\leq k\leq 2^{-\min(e,0)}}
\hspace{-.1cm}\sum_{\tiny \begin{array}{c}J{\in \cal D}_{M}^{c}\end{array}}
\hspace{-.1cm}
\sum_{I\in J_{e,m,k}\cap {\cal D}_{2M}^{c}}\Big) .
\end{align*}

Since 
$\psi_{I}^{b_{1}}=h_{I}^{b_{1}}b_{1}$, $\psi_{J}^{b_{2}}=h_{J}^{b_{2}}b_{2}$, we have
$
\langle T\psi_I^{b_{1}},\psi_J^{b_{2}}\rangle =
\langle T_{b} h_I^{b_{1}},h_J^{b_{2}}\rangle 
$.
Then by Proposition \ref{twobumplemma1} 
we get the following inequalities: for $m>3$,
$
|\langle T\psi_I^{b_{1}},\psi_J^{b_{2}}\rangle |
\lesssim 2^{-|e|(\frac{n}{2}+\delta )}m^{-(n+\delta )}
B\!F(I,J),
$
while for $m\leq 3$ and $k\geq1$,
$
|\langle T\psi_I^{b_{1}},\psi_J^{b_{2}}\rangle |
\lesssim 2^{-|e|\frac{n}{2}}k^{-\delta }
B\!F(I,J).
$
Writing both bounds in a unified manner as
$G_{e,m,k}
B\!F(I,J),$
%
%
we can estimate $|\langle (P_{2}^{*})^{\perp}TP_{2M}^{\perp }f,g\rangle |$ by
\begin{align}\label{moduloin2}
&\sum_{e\in \mathbb Z}\sum_{m,k\in \mathbb N}G_{e,m,k}
\sum_{\tiny \begin{array}{c}J\in {\cal D}_{M}^{c}\end{array}}
\sum_{I\in J_{e,m,k}\cap {\cal D}_{2M}^{c}}
|\langle f,\tilde{\psi}_{I}^{b_{1}}\rangle | 
|\langle g,\tilde{\psi}_{J}^{b_{2}}\rangle | B\!F(I,J) .
\end{align}

Now, in order to estimate this last quantity, we divide the study into two cases:
1) $(I,J)\in \mathcal F_{M}$ 
and
2) $(I,J)\notin \mathcal F_{M}$

1) $(I,J)\in \mathcal F_{M}$ implies $F(I,J)<\epsilon $ and so, \eqref{useofPlanche} holds. Now, 
we divide this case in two sub-cases depending whether $m>3$ or $m\leq 3$. 

1.1) 
We bound the terms in (\ref{moduloin2}) 
corresponding to the case $m>3$ by
\begin{align*}
\sum_{e\in \mathbb Z}\sum_{m\geq 4}\frac{2^{-|e|(\frac{n}{2}+\delta )}}{m^{n+\delta }}
&\Big( \sum_{I\in {\mathcal D_{M}^{c}}}\hspace{-.1cm} \sum_{\tiny \begin{array}{c}J\! \in \! I_{-e,m}\cap {\mathcal D}_{M}^{c}\\
(I,J)\in \mathcal F_{M}
\end{array}}\hspace{-.1cm}
B\!F(I,J)
|\langle f,\tilde{\psi}_{I}^{b_{1}}\rangle |^2\Big)^{\frac{1}{2}}
\\
&\Big( \sum_{J\in {\mathcal D}_{M}^{c}}\hspace{-.1cm} \sum_{\tiny \begin{array}{c}I\! \in \! J_{e,m}\cap {\mathcal D}_{M}^{c}\\
(I,J)\in \mathcal F_{M}
\end{array}}\hspace{-.1cm}
B\!F(I,J)
|\langle g,\tilde{\psi}_{J}^{b_{2},j}\rangle |^2\Big)^{\frac{1}{2}}.
\end{align*}
The cardinality of
$J_{e,m}$ is comparable to $2^{-\min(e,0)n}m^{n-1}$ and so, 
the cardinality of $I_{-e,m}$ is comparable to 
$2^{\max(e,0)n}m^{n-1}$. Then
we bound the previous expression by a constant comparable to $2^{n}$ multiplied by 

\begin{align*}
&\sum_{e\in \mathbb Z}\sum_{m\geq 4}
\frac{2^{-|e|(\frac{n}{2}+\delta )}}{m^{n+\delta }}
\Big( 2^{\max(e,0)n}m^{n-1}\hspace{-.2cm}\sum_{I\in {\mathcal D}_{M}^{c}}
\sup_{\tiny \begin{array}{c}J\in  \mathcal D_{M}^{c}\\
 (I,J)\in \mathcal F_{M}
 \end{array}}\hspace{-.5cm}B\!F(I,J)
|\langle f,\tilde{\psi}_{I}^{b_{1}}\rangle |^2\Big)^{\frac{1}{2}}
\\
&\hskip90pt \Big( 2^{-\min(e,0)n}m^{n-1}\hspace{-.2cm}\sum_{J\in {\mathcal D}_{M}^{c}}
\hspace{-.3cm}
\sup_{\tiny \begin{array}{c}I\in  \mathcal D_{M}^{c}\\
(I,J)\in \mathcal F_{M}
\end{array}}\hspace{-.5cm}B\!F(I,J)
|\langle g,\tilde{\psi}_{J}^{b_{2}}\rangle |^2\Big)^{\frac{1}{2}}
\\
&\lesssim \epsilon \Big( \sum_{e\in \mathbb Z}2^{-|e|\delta }\sum_{m\geq 4}m^{-(1+\delta )}\Big)\| f\|_{L^{2}(\mathbb R^{n})}\| g\|_{L^{2}(\mathbb R^{n})}
\lesssim \epsilon 
\end{align*}
where in the first inequality we used \eqref{useofPlanche} and $2^{\max(e,0)}2^{-\min(e,0)}=2^{|e|}$.


1.2) When $m\leq 3$, we assume $\ell(J)\leq \ell(I)$ to simplify notation. 
This choice implies that $e\geq 0$. 
Then, 
as before, we bound the terms in (\ref{moduloin2}) 
corresponding to this case by a constant times the following quantity
\begin{align}\label{firstpart}
\nonumber
\hspace{-.2cm}\sum_{e\geq 0}2^{-e\frac{n}{2}}
&\Big( \sum_{I\in {\mathcal D}_{M}^{c}}
\sup_{\tiny \begin{array}{c}J\in  \mathcal D_{M}^{c}\\
(I,J)\in \mathcal F_{M}
\end{array}}\hspace{-.5cm}B\!F(I,J)
|\langle f,\tilde{\psi}_{I}^{b_{1}}\rangle |^2
\sum_{k=1}^{2^{-\min(-e,0)}}
\sum_{\tiny \begin{array}{c}J\! \in \! I_{-e,m,k}\end{array}}\hspace{-.4cm}k^{-2\delta}
\Big)^{\frac{1}{2}}
\\
\nonumber
&\Big( \sum_{J\in {\mathcal D}_{M}^{c}}
\sum_{k=1}^{2^{-\min(e,0)}}
\sum_{\tiny \begin{array}{c}I\! \in \! J_{e,m,k}\end{array}}\hspace{-.3cm}
\sup_{\tiny \begin{array}{c}J\in  \mathcal D_{M}^{c}\\
(I,J)\in \mathcal F_{M}
\end{array}}\hspace{-.3cm}B\!F(I,J)
|\langle g,\tilde{\psi}_{J}^{b_{2}}\rangle |^2
\Big)^{\frac{1}{2}}
\\
&\lesssim \epsilon \sum_{e\geq 0}2^{-e\frac{n}{2}}\| f\|_{L^{2}(\mathbb R^{n})}
\Big(2^{e(n-1)}\sum_{k=1}^{2^{e}}k^{-2\delta}
\Big)^{\frac{1}{2}}\| g\|_{L^{2}(\mathbb R^{n})},
\end{align}
using \eqref{useofPlanche} and the facts that the cardinality of 
$I_{-e,m,k}$ is comparable to $n(2^{e}-k)^{(n-1)}$
while the cardinality of $J_{e,m,k}$ is comparable to $n$.

Let $0<\theta <1$ to be chosen later. Using $k\geq 1$ and $\delta >0$, 
%
we have
\begin{align*}
\sum_{k=1}^{2^{e}}
k^{-2\delta}=
\sum_{k=1}^{2^{\theta e}}
k^{-2\delta}
+\sum_{k=2^{\theta e}+1}^{2^{e}}
k^{-2\delta}
&\lesssim  2^{\theta e}
+2^{-2\delta \theta e }2^{e}.
\end{align*}

Then,  expression \eqref{firstpart} is bounded by a constant multiplied by 
\begin{align*}
&\epsilon \sum_{e\geq 0}
2^{-e\frac{n}{2}}
\big(2^{e\frac{n-1+\theta }{2}}+2^{e\frac{n-\theta 2\delta }{2}}\big)
\lesssim \epsilon \Big(\sum_{e\geq 0}2^{-e\frac{1-\theta}{2}}
+\sum_{e\geq 0}2^{-e\theta \delta }\Big)
\lesssim \epsilon ,
\end{align*}
since $0<\theta <1$. This finishes the first case.

2) We now study the case when $I\in \mathcal D_{2M}^{c}$, $J\in \mathcal D_{M}^{c}$ are such that $F(I,J)\geq \epsilon$. By Lemma \ref{smallF}, we have that 
 $|\log(\ec(I,J))|\gtrsim \log M$, or $\rdist(I, J)\gtrsim M^{\frac{1}{8}}$. 
Therefore, instead the smallness of $B\! F$, in this case we use that 
the size and location of the cubes $I$ and $J$ are such that either their eccentricity or their relative distance are extreme. 
We fix $e_{M}\in \{0, \log M\}$, $m_{M}\in \{M^{\frac{1}{8}},1\}$ such that $e_{M}=0$ implies $m_{M}=M^{\frac{1}{8}}$.
When $m> 3$, using \eqref{useofPlanche0} and the calculations developed in the case 1.1), we bound the relevant part of  (\ref{moduloin2}) by a constant times
\begin{align*}
\sum_{|e|\geq e_{M}}&
\sum_{m\geq m_{M}}
\frac{2^{-|e|(\frac{n}{2}+\delta )}}{m^{n+\delta }}
\sum_{J\in {\cal D}_{M}^{c}} \sum_{I\in J_{e,m}\cap {\cal D}_{2M}^{c}}
|\langle f,\tilde{\psi}_{I}^{b_{1}}\rangle |
|\langle g,\tilde{\psi}_{J}^{b_{2}}\rangle | B\!F(I,J)
\\
&
\lesssim \Big(\sum_{|e|\geq e_{M}}2^{-|e|\delta }\sum_{m\geq m_{M}}
m^{-(1+\delta )}\Big)\| f\|_{L^{2}(\mathbb R^{n})}\| g\|_{L^{2}(\mathbb R^{n})}
\\
&\lesssim 2^{-e_{M}\delta }m_{M}^{-\delta }
\lesssim (M^{-\delta }+M^{-\frac{\delta }{8}})
<\epsilon , 
\end{align*}
by the choice of $M$.

When $m\leq 3$, we have that $m_{M}\leq m\leq 3<M^{\frac{1}{8}}$ implies $m_{M}=1$ and so, $e_{M}=\log M$.
Then the calculations of case 1.2) show that, using \eqref{useofPlanche0},  the relevant part of  (\ref{moduloin2}) can be bounded by a constant times
\begin{align*}
\sum_{e\geq e_{M}}&\sum_{m=1}^{3}
\sum_{k=1}^{2^{e}}
2^{-e\frac{n}{2}}\hspace{-.2cm}
\sum_{J\in {\cal D}_{M}^{c}} \sum_{I\in J_{e,m,k}\cap {\cal D}_{M}^{c}}
|\langle f,\tilde{\psi}_{I}^{b_{1}}\rangle |
|\langle g,\tilde{\psi}_{J}^{b_{2}}\rangle | B\!F(I,J)
\\
&\lesssim \sum_{e\geq e_{M}}(2^{-e\frac{1-\theta}{2}}+2^{-e\theta \delta})
\| f\|_{L^{2}(\mathbb R^{n})}\| g\|_{L^{2}(\mathbb R^{n})}
\\
&
\lesssim \sum_{e\geq \log M}2^{-e\frac{\min(\theta 2\delta, 1-\theta)}{2}}
\lesssim M^{-\frac{\delta }{1+2\delta}}
\leq \epsilon, 
\end{align*}
by the choice of $M$ and  $\theta =\frac{1}{1+2\delta}\in (0,1)$.
This completely finishes the proof of compactness on $L^{2}(\mathbb R^{n})$. 

\begin{corollary}
With the hypotheses of Theorem \ref{L2bounds} 
plus 
the extra condition $b_{1},b_{2}\in L^{\infty }(\mathbb R^{n})$, we obtain compactness of $T_{b}$ on $L^{2}(\mathbb R^{n})$.
\end{corollary}
\begin{proof} To prove compactness of $T_{b}$, we use the dual representation of Corollary \ref{repofproj} (or, equivalently, 
equality \eqref{generaldec2}), 
\begin{align*}\label{orthoproforTb}
\nonumber
\langle (P_{M}^{*})^{\perp}&T_{b}P_{2M}^{\perp }f),g\rangle 
=\sum_{I\in {\cal D}_{2M}^{c}}
\sum_{J\in {\cal D}_{M}^{c}}
\langle f,\psi_{I}^{b_{1}}\rangle 
\langle g, \psi_{J}^{b_{2}}\rangle
\langle T_{b} \tilde{\psi}_I^{b_{1}},\tilde{\psi}_J^{b_{2}}\rangle 
\\
&=\sum_{I\in {\cal D}_{2M}^{c}}
\sum_{J\in {\cal D}_{M}^{c}}
\langle f,\psi_{I}^{b_{1}}\rangle 
\langle g, \psi_{J}^{b_{2}}\rangle \langle b_{1}\rangle_{I}\langle b_{2}\rangle_{J}
\langle T_{b} h_I^{b_{1}},h_J^{b_{2}}\rangle .
\end{align*}
By Lemma \ref{Planchedual}, we have
\begin{align*}
\sum_{I\in \mathcal D}
\Big(\frac{[b_{1}]_{I_{p},q}}{|\langle b_{1}\rangle_{I}||\langle b_{1}\rangle_{I_{p}}|}\Big)^{-2}
|\langle f,\psi_{I}^{b_{1}}\rangle |^{2}
\lesssim \| fb\|_{L^{2}(\mathbb R^{n})}^{2}
\leq \| b\|_{L^{\infty }(\mathbb R^{n})}^{2}\| f\|_{L^{2}(\mathbb R^{n})}^{2}
\end{align*}
and the same for $g$ and $b_{2}$.
This implies the following two inequalities, which are similar to \eqref{useofPlanche0} and  \eqref{useofPlanche}:
\begin{equation*}
\Big(\sum_{I\in {\mathcal D}}
B\!F(I,J)
|\langle f,\psi_{I}^{b_{1}}\rangle |^{2}[b_{1}]_{I}^{2}
\Big)^{1/2}
\lesssim \| f\|_{L^{2}(\mathbb R^{n})},
\end{equation*}
for $J\in {\mathcal D}$; and given $\epsilon >0$, there exists $M_{0}\in \mathbb N$ with
\begin{equation*}
\Big(\sum_{I\in {\mathcal D}_{M}^{c}}
\hspace{-.2cm}\sup_{\tiny \begin{array}{c}J\in  \mathcal D_{M}^{c}\\
(I,J)\in \mathcal F_{M} \end{array}}\hspace{-.5cm}B\!F(I,J)
|\langle f,\psi_{I}^{b_{1}}\rangle |^{2}[b_{1}]_{I}^{2}
\Big)^{1/2}
\lesssim \epsilon \| f\|_{L^{2}(\mathbb R^{n})}
\end{equation*}
for all $M>M_{0}$ and $f\in {\mathcal C}_{0}(\mathbb R^{n})$. We have analog inequalities
for $b_{2}$, $g$.
From here we can proceed as in the proof of Theorem \ref{L2bounds}.
\end{proof}

As in \cite{V}, we deduce compactness on $L^{p}(\mathbb R^{n})$ for all $1<p<\infty $ by interpolation between compactness on $L^{2}(\mathbb R^{n})$ and boundedness $L^{p}(\mathbb R^{n})$. 
We refer to the classical Krasnoselskii's Theorem, whose proof 
in a more general setting can be found in \cite{Kra}.

\begin{theorem}
Let $1\leq p_{1}, r_{1}, p_{2}, r_{2} \leq \infty$ be a set of indices with $r_{1}<\infty$. Let $T$ be a given linear operator which is continuous simultaneously as a mapping from $L^{p_{1}}(\mathbb R^{n})$ to $L^{r_{1}}(\mathbb R^{n})$ and from $L^{p_{2}}(\mathbb R^{n})$ to $L^{r_{2}}(\mathbb R^{n})$. Assume in addition that $T$ is compact as a mapping from $L^{p_{1}}(\mathbb R^{n})$ to $L^{r_{1}}(\mathbb R^{n})$. 
Then $T$ is compact as a mapping from $L^{p}(\mathbb R^{n})$ to $L^{r}(\mathbb R^{n})$, where $1/p=t/p_{1}+(1-t)/p_{2}$, 
$1/r=t/r_{1}+(1-t)/r_{2}$, $0<t<1$. 
\end{theorem}

\section{Compact paraproducts}\label{paraproducts}

For general $Tb_{1}, T^{*}b_{2}\in \CMO_{b}(\mathbb R^{n})$, we construct paraproducts $\Pi_{Tb_{1}}$, $\Pi_{T^{*}b_{2}}^{*}$ with
compact C-Z kernels such that
$\Pi_{Tb_{1}}(b_{1})=Tb_{1}$, $\Pi_{T^{*}b_{2}}^{*}(b_{1})=0$ while
$\Pi_{Tb_{1}}^{*}(b_{2})=0$,
$\Pi_{T^{*}b_2}(b_{2})=T^{*}b_2$.
This way, the operator
$$
\tilde{T}=T-\Pi_{Tb_{1}}-\Pi_{T^{*}b_{2}}^{*}
$$
satisfies the hypotheses of Theorem \ref{L2bounds} and so, $\tilde{T}$ is compact on $L^{p}(\mathbb R^{n})$. Since the paraproducts $\Pi_{Tb_{1}}$ and $\Pi_{T^{*}b_{2}}^{*}$ are compact by construction, we deduce that
the operator $T$ is also compact on $L^{p}(\mathbb R^{n})$.

We start with two technical lemmata. The first one describes the $\BMO_{b}(\mathbb R^{n})- 
H^{1}_{b}(\mathbb R^{n})$ duality. Since this result is well known for bounded accretive functions $b$,
we just sketch its proof to show the validity of the result. Some considerations regarding the use of finite decompositions
should be added to obtain a rigorous demonstration (see \cite{AusBan}). However, since we only use the estimates starting at the right hand side of \eqref{simplification}, the calculations in the paper are not affected by these issues.   
\begin{lemma}\label{dualityBMOH1} Let $b$ be a locally integrable function with non-zero dyadic averages. 
Then for all $f\in \BMO_{b}(\mathbb R^{n})$, $g\in H^{1}_{b}(\mathbb R^{n})$ 
$$
\Big|\int_{\mathbb R^{n}} f(x) g(x)b(x)dx\Big|\leq \|f\|_{BMO_{b}(\mathbb R^{n})}\| g\|_{H^{1}_{b}(\mathbb R^{n})}
$$ 
\end{lemma}
\proof 
We assume that $g\in \mathcal C_{0}(\mathbb R^{n})$ with support in $Q\in \mathcal D$. 
By Definition \ref{defH1dual} of $H^{1}_{b}(\mathbb R^{n})$, there exists a decomposition $g=\sum_{I\in \mathcal D}
\lambda_{I}a_{I}$ with $a_{I}$ $L^{2}$-atoms supported on $I_{p}\in \mathcal D(Q)$ and 
$
\| a_{I}b\|_{L^{2}(I)}\lesssim B_{I,2}^{b}
|I|^{-\frac{1}{2}}
$, such that $\sum_{I\in \mathcal D}
B_{I,2}^{b}|\lambda_{I}|\leq 2\| g\|_{H^{1}_{b}(\mathbb R^{n})}$. 
By Lemma \ref{densityinL2}, we have that 
$$
a_{I}b=\sum_{J\in {\mathcal D(I)}} 
\langle a_{I}b,\tilde \psi_{J}^{b}\rangle \psi_{J}^{b}
$$ 
with a.e. convergence and 
$J_{p}\subseteq I_{p}$ since otherwise $\langle a_{I},\tilde\psi_{J}^{b}\rangle=0$. This is trivial when $I_{p}\cap J_{p}=\emptyset $. When $I_{p}\subsetneq J_{p}$, is due to $a_{I}b$ having mean zero and $\tilde\psi_{J}^{b}$ being constant on the support of $a_{I}b$.
Moreover, 
\begin{align*}
\Big\| \sum_{I\in \mathcal D}
\lambda_{I}a_{I}b\Big\|_{L^{1}(Q)}
&
\leq \sum_{I\in \mathcal D}
|\lambda_{I}|\|a_{I}b\|_{L^{2}(I)}|I|^{\frac{1}{2}}
\leq \sum_{I\in \mathcal D}
|\lambda_{I}|B_{I,2}^{b}
\lesssim \| g\|_{H^{1}_{b}(\mathbb R^{n})}
\end{align*}
and so, by Vitali's Dominated Convergence Theorem, 
\begin{equation}\label{simplification}
\int f(x) g(x)b(x)dx
=\sum_{I\in \mathcal D}
\lambda_{I}\sum_{J\in {\mathcal D(I)}} 
\langle f,\psi_{J}^{b}\rangle 
\langle a_{I}b,\tilde\psi_{J}^{b}\rangle .
\end{equation}
Then, by Cauchy-Schwarz inequality,
\begin{align*}
\Big|\int f(x) g(x)b(x)dx\Big|
&\leq \sum_{I\in \mathcal D}
|\lambda_{I}|\Big(\sum_{J\in {\mathcal D(I)}}
\Big(\frac{[b]_{J_{p},2}}{|\langle b\rangle_{J_{p}}|}\Big)^{2}|\langle f,\psi_{J}^{b}\rangle |^{2}\Big)^{\frac{1}{2}}
\\
&\hskip70pt\Big(\sum_{J\in {\mathcal D(I)}}
\Big(\frac{[b]_{J_{p},2}}{|\langle b\rangle_{J_{p}}|}\Big)^{-2}|\langle a_{I}b,\tilde\psi_{J}^{b}\rangle |^{2}\Big)^{\frac{1}{2}}
\end{align*}
By Remark \ref{realPlanche} the last factor is bounded by 
$
\| a_{I}b\|_{L^{2}(I)}\lesssim B_{I,2}^{b}
|I|^{-\frac{1}{2}}
$. 
Then, by Definitions \ref{CMObdef} and \ref{defH1dual},
\begin{align*}
\Big|\int f(x) g(x)dx\Big|
&\leq \sum_{I\in \mathcal D}
|\lambda_{I}|B_{I,2}^{b}
\Big(\frac{
1}{|I|}\sum_{J\in {\mathcal D(I)}} 
\Big(\frac{[b]_{J_{p},2}}{|\langle b\rangle_{J_{p}}|}\Big)^{2}
|\langle f,\psi_{J}^{b}\rangle |^{2}\Big)^{\frac{1}{2}}
\\
&\leq \sum_{I\in \mathcal D}
|\lambda_{I}|B_{I,2}^{b}
\|f\|_{\BMO_{b}(\mathbb R^{n})}
\leq \| f\|_{\BMO_{b}(\mathbb R^{n})}\| g\|_{H^{1}_{b}(\mathbb R^{n})} .
\end{align*}

%

Although the wavelet system $(\psi_{I})_{I\in \mathcal D}$ is not orthogonal, we have
the following lemma, which is a direct consequence of \eqref{psiortho1}, \eqref{psiortho2} and \eqref{2^d-1}.
\begin{lemma}\label{enoughortho} 
Let $(\alpha_{J})_{J\in \mathcal D}$ a sequence of complex numbers and let $f=\sum_{J\in \mathcal D}
\alpha_{J}
\psi_{J}^{b}$. Then 
$
\langle f,\tilde\psi_{I}^{b}\rangle =\alpha_{I}.
$
\end{lemma}

Now we state and prove the main result of this section. 
\begin{proposition}\label{paraproducts1} Let $b_{1}, b_{2}$ be locally integrable functions with non-zero dyadic averages 
and let $(\psi_I^{b_{2}})_{I\in \mathcal D}$ 
be 
the Haar wavelet system of Definition \ref{defpsi}. We assume 
$Tb_{1}\in \CMO_{b}(\mathbb R^{n})$. 
Then the  operator 
\begin{align}\label{defparaproduct}
\langle \Pi_{Tb_{1}}f,g\rangle 
&=\sum_{I\in\mathcal D}
\langle Tb_1, \psi_{I}^{b_{2}}\rangle \frac{\langle f\rangle_{I_{p}}}{\langle b_{1}\rangle_{I_{p}}}
 \langle g,\tilde\psi_{I}^{b_{2}}\rangle
\end{align}
has a compact Calder\'on-Zygmund kernel, it  is compact on $L^{p}(\mathbb R^{n})$ for all $1<p<\infty $, and satisfies
$
\langle \Pi_{Tb_{1}}b_{1},g\rangle =\langle Tb_1,g\rangle
$
and
$
\langle \Pi_{Tb_{1}}^{*}b_{2},f\rangle =0
$.
\end{proposition}

\begin{remark} The proof shows that $\Pi_{Tb_{1}}$ is a perfect Calder\'on-Zygmund operator (see \cite{AHMTT} for the definition). Moreover, writing $\bar{E}_{I}^{b_{1}}(f)=\frac{\langle f\rangle_{I}}{\langle b_{1}\rangle_{I}}\chi_{I}$, we have
$
\Pi_{Tb_{1}}f=\sum_{k\in \mathbb Z}\Delta_{k}^{b_{2}}(Tb_{1})\bar{E}^{b_{1}}_{k-1}(f)
$.

On the other hand, since $\langle Tb_{1}, \psi_{J}^{b_{2}}\rangle =\langle T_{b}1, h_{J}^{b_{2}}\rangle$ and 
$\tilde\psi_{I}^{b_{2}}=\langle b_{2}\rangle_{I}h_{I}^{b_{2}}$, we have 
\begin{align*}
\langle \Pi_{Tb_{1}}f,b_{2}g\rangle
&=\sum_{I\in\mathcal D}
\langle T_{b}1, \tilde \psi_{I}^{b_{2}}\rangle \frac{\langle f\rangle_{I_{p}}}{\langle b_{1}\rangle_{I_{p}}}
 \langle g,\psi_{I}^{b_{2}}\rangle
 =: \langle {\bf \Pi}_{T_{b}1}f,g\rangle .
\end{align*}
This is the paraproduct needed to prove compactness of $T_{b}$. Moreover,  
$
{\bf \Pi}_{T_{b}1}f =\sum_{k\in \mathbb Z}(\Delta_{k}^{b_{2}})^{*}(T_{b}1)\bar{E}^{b_{1}}_{k-1}(f)
$.

\end{remark}


\proof



Formally, 
$
\langle \Pi_{Tb_{1}}b_{1},g\rangle 
=\langle Tb_1, \sum_{I\in\mathcal D}
\langle g, \tilde{\psi}_{I}^{b_{2}}\rangle \psi_{I}^{b_{2}}
\rangle
=\langle  Tb_1, g\rangle 
$.
Moreover, by Lemma \ref{psiortho} we get $\langle b_{2},\tilde{\psi}_{I}^{b_{2}}\rangle
=0$ and so, 
$\langle \Pi_{Tb_{1}}f,b_{2}\rangle =0
$.

To prove that $\Pi_{Tb_{1}}$ is compact on $L^2(\mathbb R^{n})$ 
we verify that 
$\langle {P_{M}^{*}}^{\perp}\Pi_{Tb_{1}}f,g\rangle $
tends to zero uniformly for all $f,g$ in the unit ball of $L^{2}(\mathbb R^{n})$, with $P_{M}$ the projection operator associated with $(\psi_{J}^{b_{2}})_{J}$.  
We start by proving the equality
$
\langle {P_{M}^{*}}^{\hspace{-.1cm}\perp}\Pi_{Tb_{1}}f,g\rangle 
=\langle \Pi_{{P_{M}^{*}}^{\hspace{-.1cm}\perp}Tb_1}f,g\rangle .
$

Since $g\in L^{2}(\mathbb R^{n})$, by Lemma \ref{densityinL2} we have  
$
P_{M}^{\perp }g=\sum_{J\in {\mathcal D}_{M}^{c}}
\langle g, \tilde{\psi}_{J}^{b_{2}}\rangle \psi_{J}^{b_{2}}
$
with a.e. pointwise convergence. Moreover, by the orthogonality properties of Lemma \ref{psiortho},
$\langle P_{M}^{\perp}g,\tilde{\psi}_{I}^{b_{2}}\rangle = \sum_{J\in \child(I_{p})}\langle g,\tilde{\psi}_{J}^{b_{2}}\rangle 
\langle \psi_{J}^{b_{2}},\tilde{\psi}_{I}^{b_{2}}\rangle $.
Then 
\begin{align*}
\langle {P_{M}^{*}}^{\perp}\Pi_{Tb_{1}}f,g\rangle &
=\sum_{I\in {\mathcal D}}
\langle Tb_1,\psi_{I}^{b_{2}}\rangle
\frac{\langle f\rangle_{I_{p}}}{\langle b_{1}\rangle_{I_{p}}}
\langle P_{M}^{\perp}g,\tilde{\psi}_{I}^{b_{2}}\rangle 
\\
&=\sum_{I\in {\mathcal D}_{M}^{c}}\sum_{J\in \child(I_{p})}\langle Tb_1,\psi_{I}^{b_{2}}\rangle
\frac{\langle f\rangle_{I_{p}}}{\langle b_{1}\rangle_{I_{p}}}
\langle g,\tilde{\psi}_{J}^{b_{2}}\rangle 
\langle \psi_{J}^{b_{2}},\tilde{\psi}_{I}^{b_{2}}\rangle .
\end{align*}

On the other hand, since $Tb_{1}\in \CMO_{b}(\mathbb R^{n})$, we have
$
{P_{M}^{*}}^{\hspace{-.1cm}\perp }Tb_{1}=\sum_{J\in {\mathcal D}_{M}^{c}}
\langle Tb_{1}, \psi_{J}^{b_{2}}\rangle \tilde{\psi}_{J}^{b_{2}}
$
with a.e. pointwise convergence. By  Lemma \ref{psiortho}, 
$\langle {P_{M}^{*}}^{\hspace{-.1cm}\perp}Tb_{1},\psi_{I}^{b_{2}}\rangle = \sum_{J\in \child(I_{p})} \langle Tb_1,\psi_{J}^{b_{2}}\rangle 
\langle  \psi_{I}^{b_{2}},\tilde{\psi}_{J}^{b_{2}}\rangle $. Then, since $J_{p}=I_{p}$,
\begin{align*}
\langle \Pi_{{P_{M}^{*}}^{\hspace{-.1cm}\perp}Tb_{1}}f,g\rangle 
&=\sum_{I\in {\mathcal D}} \langle {P_{M}^{*}}^{\hspace{-.1cm}\perp}Tb_{1},\psi_{I}^{b_{2}}\rangle 
\frac{\langle f\rangle_{I_{p}}}{\langle b_{1}\rangle_{I_{p}}}
\langle g,\tilde{\psi}_{I}^{b_{2}}\rangle
\\
&=\sum_{J\in {\mathcal D}_{M}^{c}}\sum_{I\in \child(J_{p})} \langle Tb_1,\psi_{J}^{b_{2}}\rangle 
\langle  \psi_{I}^{b_{2}},\tilde{\psi}_{J}^{b_{2}}\rangle
\frac{\langle f\rangle_{J_{p}}}{\langle b_{1}\rangle_{J_{p}}}
\langle g,\tilde{\psi}_{I}^{b_{2}}\rangle .
\end{align*}
%
Symmetry of $I$, $J$ in previous expressions proves the claimed equality.

Now, by Carleson's Theorem 
(in particular \eqref{Carleson2}) and remark \ref{realPlanche},   
\begin{align*}
|\langle &{P_{M}^{*}}^{\hspace{-.1cm}\perp}\Pi_{Tb_{1}}f,g\rangle|
\!=\! |\langle \Pi_{{P_{M}^{*}}^{\hspace{-.1cm}\perp}Tb_1}f,g\rangle |
\!=\! \Big|\sum_{I\in {\mathcal D}} 
\langle {P_{M}^{*}}^{\hspace{-.1cm}\perp}Tb_{1}, \psi_{I}^{b_{2}}\rangle 
\frac{\langle f\rangle_{I_{p}}}{\langle b_{1}\rangle_{I_{p}}}
\langle g,\tilde\psi_{I}^{b_{2}}\rangle \Big|
\\
&\lesssim \Big(\sum_{I\in {\mathcal D}}
\Big(\frac{[b_{2}]_{I_{p},2}}{|\langle b_{1}\rangle_{I_{p}}||\langle b_{2}\rangle_{I_{p}}|}\Big)^{2}
|\langle {P_{M}^{*}}^{\hspace{-.1cm}\perp}Tb_{1}, \psi_{I}^{b_{2}}\rangle|^{2}
|\langle f\rangle_{I_{p}}|^{2}\Big)^{\frac{1}{2}}
\\
&\hskip40pt
\Big(\sum_{I\in {\mathcal D}}
\Big(\frac{[b_{2}]_{I_{p},2}}{|\langle b_{2}\rangle_{I_{p}}|}\Big)^{-2}|\langle g,\tilde{\psi}_{I}^{b_{2}}\rangle|^{2} \Big)^{\frac{1}{2}}
\\
&\lesssim \sup_{I\in \mathcal I}\Big(\frac{1}{|I|}
\hspace{-.3cm}\sum_{\tiny \begin{array}{c}J\in {\mathcal D}_{M}^{c}\\ J\subset I \end{array}}
\hspace{-.3cm}
\Big(\frac{[b_{2}]_{J_{p},2}}{|\langle b_{1}\rangle_{J_{p}}||\langle b_{2}\rangle_{J_{p}}|}\Big)^{2}
|\langle {P_{M}^{*}}^{\hspace{-.1cm}\perp}Tb_{1}, \psi_{J}^{b_{2}}\rangle|^{2}
\Big)^{\frac{1}{2}}
\| f\|_{2}\| g\|_{2}
\\
&
\leq \| {P_{M}^{*}}^{\hspace{-.1cm}\perp}Tb_{1}\|_{\BMO_{b}(\mathbb R^{n})},
\end{align*}
which tends to zero when $M$ tends to infinity since $Tb_{1}\in \CMO_{b}(\mathbb R^{n})$. 
To end the proof, we show that $\Pi_{Tb_{1}}$ has a a compact Calder\'on-Zygmund kernel, namely, 
that the integral representation of Definition \ref{intrep} holds.
For $f,g$ with disjoint support, 
$$
\langle \Pi_{Tb_{1}}f,g\rangle 
=\int_{\mathbb R^{n}}\int_{\mathbb R^{n}} f(t)g(x)
\sum_{I\in {\mathcal D}}
\langle Tb_{1}, \psi_{I}^{b_{2}}\rangle
\frac{1}{\langle b_{1}\rangle_{I_{p}}}\frac{\chi_{I_{p}}(t)}{|I_{p}|}\tilde{\psi}_{I}^{b_{2}}(x)dtdx .
$$
As we will see later, the disjointness of the supports of $f$ and $g$ guarantees the convergence of the infinite sum.
The kernel of $\Pi_{Tb_{1}}$ is hence 
$$
K(t,x)=\Big\langle Tb_{1},
\sum_{I\in {\mathcal D}}
\frac{1}{\langle b_{1}\rangle_{I_{p}}}\frac{\chi_{I_{p}}(t)}{|I_{p}|}\tilde{\psi}_{I}^{b_{2}}(x)
\psi_{I}^{b_{2}}\Big\rangle .
$$

Due to the singularities of $\chi_{I_{p}}$ and $\tilde{\psi}_{I}^{b_{2},i}$, 
this kernel does not satisfy 
Definition \ref{prodCZoriginal} of a compact Calder\'on-Zygmund kernel. However, a 
careful read of the proofs presented shows that all results hold if the kernel satisfies the following 
alternative inequality: given $I,J\in \mathcal D$, 
$$
|K(t,x)-K(t,x')|\lesssim \frac{\ell(J)^{\delta }}{|t-x|_{\infty }^{n+\delta}}L(|t-x|_{\infty })S(|t-x|_{\infty })D(|t+x|_{\infty })
$$
for all $t\in I$ and $x,x'\in J$ with $2|x-x'|_{\infty }<|t-x|_{\infty }$. 
We will  prove that the kernel of $\Pi_{Tb_{1}}$ satisfies this inequality with $\delta=1$. 
This is equivalent to saying that $\Pi_{Tb_{1}}$ has a perfect Calder\'on-Zygmund kernel. 

In fact, we will prove that $|K(t,x)-K(t,x')|$ can be estimated by $\ell(J)/|t-x|_{\infty }^{n+1}$ times a bounded function which tends to zero when $|t-x|_{\infty }\rightarrow 0$ or $|x-t|_{\infty }\rightarrow 0$ or $|t+x|_{\infty }\rightarrow \infty$.
%
First of all, we have
$$
K(t,x)-K(t,x')=
\sum_{I\in {\mathcal D}}
\langle Tb_{1}, \psi_{I}^{b_{2}}\rangle \frac{1}{\langle b_{1}\rangle_{I_{p}}}\frac{\chi_{I_{p}}(t)}{|I_{p}|}
(\tilde{\psi}_{I}^{b_{2}}(x)-\tilde{\psi}_{I}^{b_{2}}(x')).
$$
To simplify notation, we write $\tilde{\Psi}_{I}^{b_{2}}(x,x')=\tilde{\psi}_{I}^{b_{2}}(x)-\tilde{\psi}_{I}^{b_{2}}(x')$.

We note that 
$\chi_{I}(t)\tilde{\Psi}_{I}^{b_{2}}(x,x')\neq 0$  
implies that for all cubes $I$ in the sum we have
$t\in I_{p}$ and either $x\in I$ or $x'\in I$. Let $I_{t,x,x'}$, $I_{t,x}$, $I_{t,x'}$ and $I_{x,x'}$
be the smallest dyadic cubes containing 
the points in the subindexes.  
By hypothesis, $|t-x|_{\infty }\approx |t-x'|_{\infty }$. And 
by symmetry, we assume $|t-x|_{\infty }\leq |t-x'|_{\infty }$. 
Then all cubes $I$ in the sum satisfy $I_{t,x}\subset I_{p}$ and 
the previous expression can be written as
\begin{equation}\label{K-K}
\sum_{\tiny \begin{array}{c}I\in {\mathcal D}\\ I_{t,x}\subseteq I_{p}\end{array}}
\langle Tb_{1}, \psi_{I}^{b_{2}}\rangle \frac{\chi_{I_{p}}(t)}{\langle b_{1}\rangle_{I_{p}}|I_{p}|}
\tilde{\Psi}_{I}^{b_{2}}(x,x').
\end{equation}
Notice that $|t-x|_{\infty }\leq  \ell(I_{t,x})\leq \ell(I_{t,x,x'})$. We will see later that
if $I_{x,x'}\subsetneq I_{t,x}=I_{t,x,x'}$ then $K(t,x)-K(t',x)=0$. That is, $K$ is a perfect Calder\'on-Zygmund kernel.

Since $Tb_{1}\in \CMO_{b}(\mathbb R^{n})$, for every $\epsilon >0$ there is $M_{0}\in \mathbb N$ such that
$\| {P_{M}^{*}}^{\hspace{-.1cm}\perp}Tb_{1}\|_{\BMO_{b}(\mathbb R^{n})}<\epsilon $ and $2^{-Mn/2}(1+\| Tb_{1}\|_{\BMO_{b}(\mathbb R^{n})})<\epsilon $ for all $M>M_{0}$. 
We are going to prove that 
$$
|K(t,x)-K(t,x')|\lesssim \epsilon \frac{\ell(I_{x,x'})}{|t-x|_{\infty }^{n+1}} ,
$$
when either $|t-x|_{\infty }>2^{M+1}$, or $|t+x|_{\infty }>M2^{M+2}$, or $|t-x|_{\infty }<2^{-2M}$.

\vskip10pt
{\bf 1)} When $|t-x|_{\infty }>2^{M+1}$, all cubes $I\in \mathcal D$ in the sum satisfy $2\ell(I)=\ell(I_{p})\geq \ell(I_{t,x})\geq  |x-t|_{\infty }>2^{M+1}$ and so, 
$I\in {\mathcal D}_{M}^{c}$. 
We can then rewrite (\ref{K-K}) as 
\begin{align}\label{paracase1}
&\sum_{I\in {\mathcal D}} 
\langle {P_{M}^{*}}^{\hspace{-.1cm}\perp}Tb_{1}, \psi_{I}^{b_{2}}\rangle \frac{1}{\langle b_{1}\rangle_{I_{p}}}\frac{\chi_{I_{p}}(t)}{|I_{p}|}\tilde{\Psi}_{I}^{b_{2}}(x,x').
\end{align}
%

To be used in case 2), we note that this is the only instance when we use the actual inequality $|t-x|_{\infty }>2^{M+1}$. From now, we will only use that $I\in \mathcal D_{M}^{c}$.
By Lemma \ref{enoughortho}, we have
$$
\Big\langle \sum_{J\in {\mathcal D}}
\frac{\chi_{J_{p}}(t)}{|J_{p}|}
\tilde{\Psi}_{J}^{b_{2}}(x,x')\psi_{J}^{b_{2}},\tilde{\psi}_{I}^{b_{2}}\Big\rangle
=\frac{\chi_{I_{p}}(t)}{|I_{p}|}
\tilde{\Psi}_{I}^{b_{2}}(x,x')
$$
Then, \eqref{paracase1}, and thus $K(t,x)-K(t,x')$, can be rewritten as 
\begin{align}\label{lastone}
\nonumber
\sum_{I\in {\mathcal D}}
& \sum_{J\in {\mathcal D}}
\frac{1}{\langle b_{1}\rangle_{I_{p}}}
\langle {P_{M}^{*}}^{\hspace{-.1cm}\perp}Tb_{1},\psi_{I}^{b_{2}}\rangle 
\Big\langle \frac{\chi_{J_{p}}(t)}{|J_{p}|^{\frac{1}{2}}}\tilde{\Psi}_{J}^{b_{2}}(x,x')
|J_{p}|^{-\frac{1}{2}}\psi_{J}^{b_{2}},\tilde\psi_{I}^{b_{2}}\Big\rangle 
\\
&=\sum_{I\in {\mathcal D}}
\frac{\chi_{I_{p}}(t)}{|I_{p}|^{\frac{1}{2}}}\tilde{\Psi}_{I}^{b_{2}}(x,x')
\hspace{-.3cm}\sum_{J\in {\mathcal D}(I_{p})}
\hspace{-.1cm}\frac{1}{\langle b_{1}\rangle_{J_{p}}}
\langle {P_{M}^{*}}^{\hspace{-.1cm}\perp}Tb_{1},\psi_{J}^{b_{2}}\rangle 
\langle 
|I_{p}|^{-\frac{1}{2}}\psi_{I}^{b_{2}},\tilde\psi_{J}^{b_{2}}\rangle .
\end{align}
where we interchanged $I,J$ and condition $J\in {\mathcal D}(I_{p})$ is due to \eqref{psiortho2}. 
By Lemma \ref{psiortho},
$\| \psi_{I}^{b_{2}}\|_{L^{2}(\mathbb R^{n})}\lesssim B_{I,q_{2}}^{b_{2}}$. 
Since \eqref{lastone} coincides with the right hand side of 
\eqref{simplification},
by the proof of Lemma \ref{dualityBMOH1}, \eqref{lastone} is bounded by 
\begin{align*}
\sum_{I\in {\mathcal D}}&
 \frac{\chi_{I_{p}}(t)}{|I_{p}|^{\frac{1}{2}}}|\tilde{\Psi}_{J}^{b_{2}}(x,x')|
\sup_{I\in \mathcal I}\Big(\frac{
(B^{b_{2}}_{I,q_{2}})^{2}}{|I_{p}|}
\hspace{-.6cm}\sum_{\tiny \begin{array}{c}J\in {\mathcal D}(I_{p})
 \end{array}}
\hspace{-.3cm}
\frac{1}{|\langle b_{1}\rangle_{J_{p}}|^{2}}\frac{[b_{2}]_{J_{p},2}^{2}}{|\langle b_{2}\rangle_{J_{p}}|^{2}}
|\langle {P_{M}^{*}}^{\hspace{-.1cm}\perp}Tb_{1}, \psi_{J}^{b_{2}}\rangle|^{2}
\Big)^{\frac{1}{2}}
\\
&\leq \| {P_{M}^{*}}^{\hspace{-.1cm}\perp}Tb_{1}\|_{\BMO_{b}(\mathbb R^{n})}
\sum_{I\in {\mathcal D}}
\frac{\chi_{I_{p}}(t)}{|I|^{\frac{1}{2}}}|\tilde{\Psi}_{I}^{b_{2}}(x,x')|,
\end{align*}
where we used $|I|\leq |I_{p}|$. We now work to bound
\begin{align}\label{H1norm}
\sum_{\tiny \begin{array}{c}I\in {\mathcal D}\\I_{t,x}\subseteq I_{p}\end{array}}
\frac{\chi_{I_{p}}(t)}{|I|^{\frac{1}{2}}}|\tilde{\psi}_{I}^{b_{2}}(x)-\tilde{\psi}_{I}^{b_{2}}(x')|.
\end{align}
If the sum is non-zero then $I_{t,x}\subseteq I_{x,x'}$: if 
$I_{x,x'}\subsetneq I_{t,x}$, then all cubes $I$ such that $I_{t,x}\subseteq I$ satisfy $x,x'\in I'$ with 
$I'\in \child (I_{p})$; this implies 
$\tilde{\psi}_{I}^{b_{2}}(x)=\tilde{\psi}_{I}^{b_{2}}(x')$ and so, the sum in \eqref{H1norm} is zero. 
Moreover, if $\tilde{\Psi}_{I}^{b_{2}}(x,x')$ is non-zero only then $x,x'$ do not belong to the same child of $I_{p}$. Then
$$
|\tilde{\psi}_{I}^{b_{2}}(x)-\tilde{\psi}_{I}^{b_{2}}(x')|\leq  |\langle b_{2}\rangle_{I}| |I|^{\frac{1}{2}}\frac{1}{|I|\langle b_{2}\rangle_{I}|}
=|I|^{-\frac{1}{2}}.
$$

Now, we parametrize the cubes in the sum by their side length: 
let $(I^{k})_{k\in \mathbb N}$ be the family of dyadic cubes such that
$I_{t,x}\subset I^{k}$ with 
$\ell(I^{k})=2^{k}\ell(I_{t,x})$. 
Then, the sum in \eqref{H1norm} can be bounded by
$$
\sum_{k\geq 0}\frac{1}{|I^{k}|}
=\sum_{k\geq 0}\frac{1}{2^{kn}\ell(I_{t,x})^{n}}
\lesssim 
\frac{\ell(I_{t,x})}{\ell(I_{t,x})^{n+1}}\leq \frac{\ell(I_{x,x'})}{|t-x|_{\infty }^{n+1}}.
$$
With all this, we obtain by the choice of $M$,
$$
|K(t,x)-K(t',x)|
\lesssim \| {P_{M}^{*}}^{\hspace{-.1cm}\perp}Tb_{1}\|_{\BMO_{b}(\mathbb R^{n})}
\frac{\ell(I_{x,x'})}{|t-x|_{\infty }^{n+1}}
\leq \epsilon \frac{\ell(I_{x,x'})}{|t-x|_{\infty }^{n+1}}.
$$

{\bf 2)} We work the case $|t+x|_{\infty }>M2^{M+2}$. Since
every cube $I$ in the sum \eqref{K-K} satisfies $I_{t,x}\subseteq I_{p}$, we have $(t+x)/2\in I_{p}$. Then 
$|c(I_{p})-(x+t)/2|_{\infty }<\ell(I_{p})/2$ and so,
$
|c(I_{p})|_{\infty }\geq |t+x|_{\infty }/2-\ell(I_{p})/2 
$.
Now, if $\ell(I_{p})>2^{M+1}$, we get as before that $I\in {\mathcal D}_{M}^{c}$.
If instead $\ell(I)\leq 2^{M+1}$, 
\begin{align*}
\rdist(I,{\mathbb B}_{2^{M}})&\geq \rdist(I_{p},{\mathbb B}_{2^{M}})=\frac{\diam(I_{p}\cup {\mathbb B}_{2^{M}})}{2^{M}}
\\
&
\gtrsim \frac{|c(I_{p})|_{\infty }+2^{M-1}+\ell(I_{p})/2}{2^{M}}
\geq 
\frac{|t+x|_{\infty }}{2^{M+1}}+\frac{1}{2}>M,
\end{align*}
by the property of $|t+x|_{\infty }$.  Therefore, we also obtain $I\in {\mathcal D}_{M}^{c}$ and with this, we conclude as in the previous case that
$$
|K(t,x)-K(t',x)|
\lesssim \| {P_{M}^{*}}^{\hspace{-.1cm}\perp}Tb_{1}\|_{\BMO_{b}(\mathbb R^{n})}
\frac{\ell(I_{x,x'})}{|t-x|_{\infty }^{n+1}}
\leq \epsilon \frac{\ell(I_{x,x'})}{|t-x|_{\infty }^{n+1}}.
$$

{\bf 3)} The last case, $|t-x|_{\infty }<2^{-2M}$, is more involved. The cubes in the sum such that 
$\ell(I)<2^{-M}$ or $\ell(I)>2^{M}$ satisfy $I\in {\mathcal D}_{M}^{c}$ and so, they may be taken care of as in the two previous cases.

However, those cubes such that $2^{-M}\leq \ell(I)\leq 2^{M}$ may belong to ${\mathcal D}_{M}$ and the previous argument can not be used. Instead, we reason as follows. The terms under consideration in \eqref{K-K}
are given by those cubes $I\in {\mathcal D}$ such that $I_{t,x}\subseteq  I_{p}$ and  $2^{-M}\leq \ell(I)\leq 2^{M}$. 
From the work in case 1), we know that $I_{t,x}\subseteq I_{x,x'}$. 
Therefore, these cubes can be parametrized by their side length as $\ell(I^{k})=2^{k}\ell(I_{t,x})$ with $k_{0}\leq k \leq k_{1}$
where $k_{0}=\max(-M-\log\ell(I_{t,x}),0)$ and $k_{1}=M-\log\ell(I_{t,x})$.
Then,
$$
K(t,x)-K(t,x')
=\hspace{-.3cm}\sum_{k_{0}\leq k \leq k_{1}}\hspace{-.1cm}
\langle Tb_{1}, \psi_{I^{k}}^{b_{2}}\rangle \frac{\chi_{I_{p}^{k}}(t)}{\langle b_{1}\rangle_{I_{p}^{k}}|I_{p}^{k}|}(\tilde{\psi}_{I^{k}}^{b_{2}}(x)-\tilde{\psi}_{I^{k}}^{b_{2}}(x')).
$$



As in the first case, we bound the modulus of previous expression by
\begin{align*}
\| Tb_{1}\|_{\BMO_{b}(\mathbb R^{n})}
\hspace{-.1cm}\sum_{k_{0}\leq k \leq k_{1}}
\frac{\chi_{I_{p}^{k}}(t)}{|I^{k}|^{\frac{1}{2}}}|\tilde{\psi}_{I^{k}}^{b_{2}}(x)-\tilde{\psi}_{I^{k}}^{b_{2}}(x')|
\end{align*}

By the same reasoning, we bound the last factor by a constant times
\begin{equation}\label{sumK-K}
\sum_{k_{0}\leq k\leq k_{1}}\frac{1}{2^{kn}}\frac{1}{\ell(I_{t,x})^{n}}.
\end{equation}
We distinguish two cases depending whether $|t-x|_{\infty }\leq 2^{-M/2}\ell(I_{t,x})$ or  $|t-x|_{\infty }> 2^{-M/2}\ell(I_{t,x})$. In the first case, 
\eqref{sumK-K} is bounded by
\begin{align*}
\sum_{0\leq k}\frac{1}{2^{kn}}\frac{1}{\ell(I_{t,x})^{n}}
&\lesssim \frac{\ell(I_{t,x})}{\ell(I_{t,x})^{n+1}}
\leq 2^{-\frac{M}{2}(n+1)}\frac{\ell(I_{x,x'})}{|t-x|_{\infty }^{n+1}}.
\end{align*}

In the second case, we have that $\ell(I_{t,x})<2^{M/2}|t-x|_{\infty }<2^{-3M/2}$  and so, 
$k_{0}\geq -M-\log \ell(I_{t,x})\geq M/2$. Then, \eqref{sumK-K} is bounded by 
$$
\sum_{k\geq \frac{M}{2}}\frac{1}{2^{kn}}\frac{1}{\ell(I_{t,x})^{n}}
\lesssim \frac{1}{2^{\frac{M}{2}n}}\frac{\ell(I_{t,x})}{\ell(I_{t,x})^{n+1}}
\leq \frac{1}{2^{\frac{M}{2}n}}\frac{\ell(I_{x,x'})}{|t-x|_{\infty }^{n+1}}.
$$

Therefore, in both cases we have by the choice of $M$ again, 
$$
|K(t,x)-K(t,x')|
\lesssim  
\| Tb_{1}\|_{\BMO_{b}(\mathbb R^{n})}
\frac{1}{2^{\frac{M}{2}n}}\frac{\ell(I_{x,x'})}{|t-x|_{\infty }^{n+1}}
\leq \epsilon \frac{\ell(I_{x,x'})}{|t-x|_{\infty }^{n+1}}.
$$

Similar reasoning applies to $K(t,x)-K(t',x)$ finishing the proof.

\section{Necessity of the hypotheses}\label{necessity section}
In this last section, we prove necessity of the three hypotheses of Theorem \ref{Mainresult2}.
Since in \cite{V}, we proved that Calder\'on-Zygmund operators compact on $L^{p}(\mathbb R^{n})$ have compact Calder\'on-Zygmund kernels, we focus on the other two hypotheses:
the weak compactness condition and the membership of $Tb_{1}$ and $T^{*}b_{2}$ to $\CMO_{b}(\mathbb R^{n})$. 

\subsection{The weak compactness condition}

\begin{proposition}\label{neceweakcompactnessfinal}
Let $1<p<\infty $ and $T$ be bounded on $L^{p}(\mathbb R^{n})$.
Let $1\leq q_{i}\leq \infty $ and 
$b_{i}$ be two locally integrable functions.
If either $p\leq q_{1}$ and $p'\leq q_{2}$, or $b_{1}, b_{2}$ are accretive
then for every $M\in \mathbb N$,
$Q\in \mathcal D$, 
\begin{align*}
|\langle &T_{b}\chi_Q,\chi_{Q}\rangle |
\lesssim 
|Q|[b_{1}]_{Q,q_{1}}[b_{2}]_{Q,q_{2}}
\Big[ \| (P_{M}^{*})^{\perp} T\|_{p, p}
\\
&
+\| P_{M}^{*}T\|_{p,p}\,\, \chi_{[0,1]}\Big(\frac{\ell(Q)}{2^{M}}\Big) \,
\Big(1+\frac{2^{-M}}{\ell(Q)}\Big)^{-\frac{n}{p}}
\chi_{[0,1]}\Big(\frac{\rdist (Q,\mathbb B_{2^{M}})}{M}\Big)\Big].
\end{align*}
\end{proposition}
\begin{corollary}
Let $1<p<\infty $ and $T$ compact on $L^{p}(\mathbb R^{n})$.
Let $q_{i}$ and 
$b_{i}$ as before.
Then $T$ satisfies the weak compactness condition
\end{corollary}

%
\begin{proof}
We start with the decomposition
$$
|\langle T_{b}\chi_Q,\chi_{Q}\rangle |
\leq |\langle P_{M}^{\perp}T_{b}\chi_Q,\chi_{Q}\rangle |
+|\langle P_{M}T_{b}\chi_Q,\chi_{Q}\rangle |.
$$
Since 
$
\langle b_{2}f,\tilde{\psi}_{I}^{b_{2}}\rangle \psi_{I}^{b_{2}}
=b_{2}
\langle f,\psi_{I}^{b_{2}}\rangle \tilde{\psi}_{I}^{b_{2}}
$ 
for all $I\in {\mathcal D}$, we have
$
P_{M}(b_{2}f)=b_{2}P_{M}^{*}f
$.
With this and 
the hypothesis on $q_{i}$ or $b_{i}$, we get
\begin{align*}
|\langle &P_{M}^{\perp }T_{b}\chi_Q,\chi_{Q}\rangle|
=|\langle (P_{M}^{*})^{\perp }T(b_{1}\chi_Q),b_{2}\chi_{Q}\rangle|
\\
&
\leq \| (P_{M}^{*})^{\perp }T\|_{p,p}
\|b_{1}\|_{L^{p}(Q)}\|b_{2}\|_{L^{p'}(Q)}
= \| (P_{M}^{*})^{\perp }T\|_{p,p}
|Q|[b_{1}]_{Q,p}[b_{2}]_{Q,p'}
\\
&
\lesssim \| (P_{M}^{*})^{\perp }T\|_{p,p}
|Q|[b_{1}]_{Q,q_{1}}[b_{2}]_{Q,q_{2}}.
\end{align*}

We deal now with the second term. If $Q\in {\cal D}_{M}$, we have as before
$
|\langle P_{M}T_{b}\chi_Q,\chi_{Q}\rangle|
\lesssim  
\| P_{M}^{*}T\|_{p,p}
|Q|[b_{1}]_{Q,q_{1}}[b_{2}]_{Q,q_{2}}
$,
which is compatible with the statement 
since $2^{-M}\leq \ell(Q)\leq 2^{M}$ and 
$\rdist(Q,\mathbb B_{2^{M}})<M$.

If $Q\notin {\cal D}_{M}$, we proceed in a different way.
By Lemma \ref{densityinL2}, 
$$
b_{2}\chi_{Q}=\sum_{\tiny \begin{array}{c}J\in {\cal D}\\ Q\subsetneq J\end{array}}
\langle b_{2}\chi_{Q},\tilde{\psi}_{J}^{b_{2}}\rangle \psi_{J}^{b_{2}}
$$
with a.e. pointwise convergence. The constraint $Q\subsetneq J$ is due to 
$\langle b_{1}\chi_{Q},\tilde{\psi}_{J}^{2}\rangle =0$ for $J\cap Q=\emptyset $ 
while, by Lemma \ref{psiortho}, 
$\langle b_{2}\chi_{Q},\tilde{\psi}_{J}^{b_{1}}\rangle =\langle b_{2},\tilde{\psi}_{J}^{b_{2}}\rangle =0$ for 
$J\subseteq Q$. Therefore, 
\begin{align}\label{telescopePM}
P_{M}(b_{2}\chi_{Q})=\sum_{\tiny \begin{array}{c}J\in {\cal D_{M}}\\ Q\subsetneq J\end{array}}
\langle b_{2}\chi_{Q},\tilde{\psi}_{J}^{b_{2}}\rangle \psi_{J}^{b_{2}},
\end{align}
where the sum is finite. 
With 
$
P_{M}(b_{2}f)=b_{2}P_{M}^{*}f
$
and $P_{M}^{2}=P_{M}$, we have
\begin{align}\label{necweak2}
\nonumber
\langle P_{M}T_{b}&\chi_Q,\chi_{Q}\rangle
=\langle P_{M}^{*}T(b_{1}\chi_Q),b_{2}P_{M}^{*}\chi_{Q}\rangle
\\
&=\langle P_{M}^{*}T(b_{1}\chi_Q),P_{M}(b_{2}\chi_{Q})\rangle
=\hspace{-.5cm}\sum_{\tiny \begin{array}{c}J\in {\cal D_{M}}\\ Q\subsetneq J\end{array}}
\hspace{-.3cm}\langle P_{M}^{*}T(b_{1}\chi_{Q}),\psi_J^{b_{2}}\rangle
\langle b_{2}\chi_{Q}, \tilde{\psi}_{J}^{b_{2}}\rangle .
\end{align}

Now, we separate into three cases:
$\ell(Q)>2^{M}$; $\ell(Q)<2^{-M}$; and $2^{-M}<\ell(Q)<2^{M}$ with 
$\rdist (Q, \mathbb B_{2^{M}})>M$. 

{\bf 1)} When $\ell(Q)>2^{M}$, all cubes $J$ in the sum satisfy $\ell(J)\geq \ell(Q)>2^{M}$, which is contradictory with $J\in \mathcal D_{M}$. Then the sum in \eqref{necweak2} is empty and  $\langle P_{M}T_{b}\chi_Q,\chi_{Q}\rangle
=0$.

{\bf 2)} When $\ell(Q)<2^{-M}$, since $J\in \mathcal D_{M}$ we have that $\ell(Q)<2^{-M}\leq \ell(J)$. 
By telescoping, the sum in \eqref{telescopePM} can be rewritten as 
$
P_{M}(b_{2}\chi_{Q})=\sum_{-M< k\leq M}
\Delta_{k}^{b_{2}}(b_{2}\chi_{Q})
=E_{M}^{b_{2}}(b_{2}\chi_{Q})-E_{-M}^{b_{2}}(b_{2}\chi_{Q}).
$
Let $J^{0},J^{1}\in \mathcal D_{M}$ such that with $Q\subsetneq J^{i}$ and $\ell(J^{i})=2^{(-1)^{i}M}$. Then
\begin{align*}
\| E_{-M}^{b_{2}}(b_{2}\chi_{Q})\|_{L^{p'}(\mathbb R^{n})}
&=\frac{|\langle b_{2}\chi_{Q}\rangle_{J^{1}}|}{|\langle b_{2}\rangle_{J^{1}}|}\| b_{2}\chi_{J^{1}}\|_{L^{p'}(\mathbb R^{n})}
\\
&=\frac{|Q|}{|J^{1}|}|\langle b_{2}\rangle_{Q}|\frac{|J^{1}|^{\frac{1}{p'}}[b_{2}]_{J^{1},p'}}{|\langle b_{2}\rangle_{J^{1}}|}
\lesssim 2^{\frac{M}{p}}|Q||\langle b_{2}\rangle_{Q}|,
\end{align*}
using 
the hypothesis on $q_{i}$ or $b_{i}$. Similarly, we get a smaller estimate for $\| E_{M}^{b_{2}}(b_{2}\chi_{Q})\|_{L^{p'}(\mathbb R^{n})}$. Whence, from the first equality in \eqref{necweak2}, we have
\begin{align*}
|\langle P_{M}T_{b}\chi_Q,\chi_{Q}\rangle |
&
\leq \| P_{M}^{*}T\|_{p,p}\| b_{1}\|_{L^{p}(Q)} \|P_{M}(b_{2}\chi_{Q})\|_{L^{p'}(\mathbb R^{n})}
\\
&\lesssim \| P_{M}^{*}T\|_{p,p}|Q|^{\frac{1}{p}}[b_{1}]_{Q,p}|\langle b_{2}\rangle_{Q}| \, 2^{\frac{M}{p}}|Q|
\\
&= \| P_{M}^{*}T\|_{p,p}|Q|[b_{1}]_{Q,q_{1}}[b_{2}]_{Q,q_{2}} \big(2^{M}\ell(Q)\big)^{\frac{n}{p}},
\end{align*}
with the hypothesis on $q_{i}$ or $b_{i}$.
This ends the second case.

{\bf 3)} We consider the case $2^{-M}<\ell(Q)<2^{M}$ and $\rdist (Q,\mathbb B_{2^{M}})>M$.  Since 
$2^{-M}\leq \ell(J)\leq 2^{M}$ and $Q\subset J$, we have that 
\begin{align*}
\rdist (Q,\mathbb B_{2^{M}})&=\frac{\ell(\langle Q, \mathbb B_{2^{M}}\rangle )}{2^{M}}
\leq \frac{\ell(\langle J, \mathbb B_{2^{M}}\rangle )}{2^{M}}=\rdist (J,\mathbb B_{2^{M}})
\end{align*}
Then, $\rdist (J,\mathbb B_{2^{M}})>M$, which is contradictory with  $J\in \mathcal D_{M}$. Therefore,  
 the sum in \eqref{necweak2} is again empty and $\langle P_{M}T_{b}\chi_Q,\chi_{Q}\rangle
=0$. 
\end{proof}

\subsection{Membership in $\CMO_{b}(\mathbb R^{n})$}
\begin{proposition}\label{necessity2}
Let $T$ be a linear operator with a standard Calder\'on-Zygmund kernel
that extends compactly on $L^{p}(\mathbb R^{n})$ 
for some $1<p<\infty $. Then for  
$b_{1}, b_{2}$ locally integrable functions compatible with $T$
we have 
$Tb_{1},T^{*}b_{2}\in \CMO_{b}(\mathbb R^{n})$.
\end{proposition}
\proof
Since $T$ is bounded on $L^{p}(\mathbb R^{n})$, by the classical theory, $T$ is bounded on 
$L^{p}(\mathbb R^{n})$ for all $1<p<\infty $. Thus, by interpolation, $T$ turns out to be compact on all $L^{p}(\mathbb R^{n})$ spaces.

To prove membership in $\BMO_{b}(\mathbb R^{n})$, we show first
that $\mathcal L_{b}$ defined in  Lemma \ref{definecmo} is a bounded linear functional on $H^{1}_{b}(\mathbb R^{n})$. Since linearity is trivial, we prove 
its continuity on $H^{1}_{b}(\mathbb R^{n})$. 

By standard arguments, it is enough to prove the result for $p'$-atoms.
Let $I\in \mathcal D$ 
be fixed and $f$ be an 
atom in 
$H_{b}^{1}(\mathbb R^{n})$ supported on
$I_{p}$ with mean zero with respect $b_{2}$ and
 $\| fb_{2}\|_{L^{p'}(\mathbb R^{n})}\lesssim B_{I,p'}^{b_{2}}
 |I|^{-1/p}$. 
 
Let $\Psi_{k}=\chi_{2^{k+1}I_{p}}-\chi_{2^{k}I_{p}}$. 
For $k\in \mathbb N$, $k>1$,
we have
\begin{align*}
|{\mathcal L}_{b}(f)|&\leq |\langle T_{b}\chi_{I_{p}},f\rangle |
+\sum_{k'=0}^{k-1}|\langle T_{b}\Psi_{k'} ,f\rangle |
+|{\mathcal L}_{b}(f)-\langle T_{b}\chi_{2^{k}I_{p}},f\rangle |.
\end{align*}


Using boundedness of $T$ on $L^{p}(\mathbb R^{n})$
we estimate the first term by
\begin{align*}
\|T\|_{p,p} \| b_{1}\chi_{I_{p}} \|_{L^{p}(\mathbb R^{n})}
\| b_{2}f\|_{L^{p'}(\mathbb R^{n})}
&\lesssim [b_{1}]_{I,p}
|I_{p}|^{\frac{1}{p}}
B_{I,p'}^{b_{2}}
|I|^{-\frac{1}{p}}
=[b_{1}]_{I,p}
B_{I,p'}^{b_{2}}.
\end{align*}
Since $T$ is compact, its kernel $K$ is a compact Calder\'on-Zygmund kernel
with parameter $\delta$. Then 
from the proof of Lemma \ref{definecmo} 
the second term is bounded by a constant times
\begin{align*}
\sum_{k'=0}^{k-1}2^{-k'\delta } &
|I_{p}|^{\frac{1}{q_{2}'}}\| b_{2}f\|_{L^{q_{2}}(\mathbb R^{n})}
\inf_{x\in 2^{k'}I} M_{q_{1}}b_{1}(x)
\tilde F_{K}(2^{k'}I,I,2^{k'}I)
\\
&\lesssim |I_{p}|^{\frac{1}{q_{2}'}}B_{I,q_{2}}^{b_{2}}
|I|^{-\frac{1}{q_{2}'}}
\inf_{x\in I} M_{q_{1}}b_{1}(x)
\lesssim \langle M_{q_{1}}b_{1}\rangle_{I}B_{I,q_{2}}^{b_{2}},
\end{align*}
where we used boundedness of $\tilde F_{K}$.
Finally, 
we apply the result of Lemma \ref{definecmo} to bound the last term by
\begin{align*}
2^{-k\delta }&
|I_{p}|^{\frac{1}{q_{2}'}}\| b_{2}f\|_{L^{q_{2}}(\mathbb R^{n})}
\inf_{x\in 2^{k}I} M_{q_{1}}b_{1}(x)
\tilde F_{K}(2^{k}I,I,2^{k}I)
\lesssim 
B_{I, q_{2}}^{b_{2}}\langle M_{q_{1}}b_{1}\rangle_{I}.
\end{align*}
 
These estimates
show
$
|{\mathcal L}_{b}(f)|\lesssim 1
$ 
for every atom $f$, proving that 
${\mathcal L}_{b}$ 
defines
a bounded linear functional on 
$H_{b}^1(\mathbb R^{n})$. Hence, 
by the $H_{b}^1(\mathbb R^{n})$-$\BMO_{b}(\mathbb R^{n})$ duality of Lemma \ref{dualityBMOH1}, 
the functional ${\mathcal L}_{b}$ is represented
by a function in $\BMO_{b}(\mathbb R^{n})$ denoted by  $Tb_{1}$, that is, 
${\mathcal L}_{b}(f)=\langle Tb_{1},b_{2}f\rangle $.

\vskip10pt
In order to prove membership in $\CMO_{b}(\mathbb R^{n})$, we need to show that 
$
\lim_{M\rightarrow \infty}\langle {P_{M}^{*}}^{\hspace{-.1cm}\perp}Tb_{1},b_{2}f\rangle =0
$
uniformly 
for all $f$ 
in the unit ball of $H_{b}^{1}(\mathbb R^{n})$. 
Let $I\in \mathcal D$ and $f$ be an 
atom in 
$H_{b}^{1}(\mathbb R^{n})$ supported on
$I$ with zero mean with respect $b_{2}$ and 
$\| fb_{2}\|_{L^{p'}(\mathbb R^{n})}\lesssim B_{I,p'}^{b_{2}}
|I|^{-1/p}$.

For $\epsilon >0$, we fix $k\in \mathbb N$, $k>1$, so that 
$2^{-k\delta }\langle M_{q_{1}}b_{1}\rangle_{I}
B_{I,q_{2}}^{b_{2}}<\epsilon $. 
Moreover, due to compactness of $T$, 
we can choose  
$M>0$
such that $I\in \mathcal D_{M}$ and
$\| {P_{M}^{*}}^{\hspace{-.1cm}\perp}T\|_{p,p} 2^{k/p}[b_{1}]_{2^{k}I_{p},p}
B_{I,p'}^{b_{2}}<\epsilon $.

We decompose as follows:
\begin{align}\label{decortho}
\langle {P_{M}^{*}}^{\hspace{-.1cm}\perp}Tb_{1},b_{2}f\rangle
\hspace{-.05cm}=\hspace{-.05cm}\langle {P_{M}^{*}}^{\perp}T(b_{1}\chi_{2^{k}I_{p}}),b_{2}f\rangle
\hspace{-.05cm}+\hspace{-.06cm}\sum_{k'=k}^{\infty }\langle T(b_{1}\Psi_{k'}),P_{M}^{\perp}(b_{2}f)\rangle .
\end{align}
Notice that 
$
P_{M}(b_{2}f)
=\sum_{J\in \mathcal D_{M}}
|J|^{1/2}\langle b_{2}f,\tilde\psi_{J}^{b_{2}}\rangle |J|^{-1/2}\psi_{J}^{b_{2}}
$
is a finite linear combination of functions $|J|^{-1/2}\psi_{J}^{b_{2}}=|J|^{-1/2}h_{I}^{b_{2}}b_{2}$ with 
$|J|^{-1/2}h_{I}^{b_{2}}$ being $p'$-atoms  and so, it belongs to $H_{b}^{1}(\mathbb R^{n})$.
Then we also get $P_{M}^{\perp}(b_{2}f)\in H_{b}^{1}(\mathbb R^{n})$, which justifies \eqref{decortho}.
We bound the first term in \eqref{decortho} as follows: 
\begin{align*}
\| {P_{M}^{*}}^{\hspace{-.1cm}\perp}T\|_{p,p} &\| b_{1}\chi_{2^{k}I_{p}} \|_{L^p(\mathbb R^{n})}
\| b_{2}f\|_{L^{p'}(\mathbb R^{n})}
\\
&\lesssim \| {P_{M}^{*}}^{\hspace{-.1cm}\perp}T\|_{p,p} [b_{1}]_{2^{k}I_{p},p}
|2^{k}I|^{\frac{1}{p}}
B_{I,p'}^{b_{2}}
|I|^{-\frac{1}{p}}
\\
&= \| {P_{M}^{*}}^{\hspace{-.1cm}\perp}T\|_{p,p} [b_{1}]_{2^{k}I_{p},p} 2^{\frac{k}{p}}
B_{I,p'}^{b_{2}}
<\epsilon .
\end{align*}

Applying the definition of ${P_{M}^{*}}^{\perp}$, we rewrite the second term as
$$
\sum_{k'=k}^{\infty }\langle T(b_{1}\Psi_{k'}),b_{2}f\rangle 
-\sum_{k'=k}^{\infty }\langle P_{M}^{*}T(b_{1}\Psi_{k'}),b_{2}f\rangle =A+B
$$

For the new first term, we have from the proof of Lemma \ref{definecmo} that
\begin{align*}
|A|&\lesssim\sum_{k'=k}^{\infty }2^{-k'\delta } 
|I_{p}|^{\frac{1}{q_{2}'}}\| b_{2}f\|_{L^{q_{2}}(\mathbb R^{n})}
\inf_{x\in 2^{k'}I} M_{q_{1}}b_{1}(x)
F^{D}_{K}(2^{k'}I,I,2^{k'}I)
\\
&\lesssim \sum_{k'=k}^{\infty }2^{-k'\delta } 
B_{I,q_{2}}^{b_{2}}
\inf_{x\in I} M_{q_{1}}b_{1}(x)
\lesssim 2^{-k\delta }B_{I,q_{2}}^{b_{2}}\langle M_{q_{1}}b_{1}\rangle_{I}
<\epsilon .
\end{align*}
By definition, we rewrite each term of $B$ as
\begin{align}\label{finallast}
\langle P_{M}^{*}T(b_{1}\Psi_{k'}),b_{2}f\rangle
&=\sum_{J\in  {\mathcal D}_{M}(I)}
\langle T(b_{1}\Psi_{k'}),\psi_{J}^{b_{2}}\rangle
\langle \tilde\psi_{J}^{b_{2}},b_{2}f\rangle
\end{align}
Since $b_{2}f$ and $\tilde\psi_{J}^{b_{2}}$ have compact support on $I_{p}$ and $J_{p}$ respectively, the non-null terms in the sum arise when $J_{p}\subset I_{p}$.  This is obvious when  
$I_{p}\cap J_{p}=\emptyset$.
For those cubes $J$ so that $I_{p}\subsetneq J_{p}$, we have 
 $\langle \tilde\psi_{J}^{b_{2}}, b_{2}f\rangle =\langle b_{2}\rangle_{J}\langle h_{J}^{b_{2}}, b_{2}f\rangle =0$ since
 $h_{J}^{b_{2}}$ is constant on $I_{p}$ and $b_{2}f$ has mean zero.
 
Now, $b_{1}\Psi_{k}$ is supported on $(2^{k'}I_{p})^{c}$ and $\psi_{J}^{b_{2}}$ is supported on $I_{p}$ and so, we can use the integral representation to rewrite \eqref{finallast} as
\begin{align}\label{ortoT1other}
\nonumber
\sum_{J\in  {\mathcal D}_{M}(I)}
&\int \int b_{1}(t)\Psi_{k'}(t) \psi_{J}^{b_{2}}(z)K(t,z)dzdt
\int b_{2}(x)f(x)\tilde\psi_{J}^{b_{2}}(x)dx
\\
&=\int \int b_{1}(t)\Psi_{k'}(t)b_{2}(x)f(x)
{\mathcal K}(t,x)dtdx
\end{align}
with $t\! \in \! (2^{k'+1}I_{p})\backslash (2^{k'}I_{p})$,
$
{\displaystyle
{\mathcal K}(t,x)\! =\hspace{-.4cm}\sum_{J\in  {\mathcal D}_{M}(I)}
\int \psi_{J}^{b_{2}}(z)K(t,z)dz \, \tilde\psi_{J}^{b_{2}}(x)}.
$


We now aim to estimate ${\mathcal K}(t,x)$.
By the mean zero of $\psi_{J}^{b_{2}}(z)$, 
$$
\int \psi_{J}^{b_{2}}(z)K(t,z)dz
=\int \psi_{J}^{b_{2}}(z)(K(t,z)-K(t,c(J)))dz.
$$
Since $t\in (2^{k'}I_{p})^{c}$, $z\in J_{p}\subset I_{p}$ and $k'>1$, we get 
\begin{align*}
|t-z|_{\infty }&\geq |t-c(I_{p})|_{\infty }-|c(I_{p})-z|_{\infty }\geq 2^{k'-1}\ell(I_{p})-\ell(I_{p})/2
\\
&>\ell(I_{p})\geq \ell(J_{p})\geq 2|z-c(J_{p})|_{\infty }.
\end{align*}
Whence, 
\begin{align*}
\Big| \int \psi_{J}^{b_{2}}(z)K(t,z)dz\Big|
&
\lesssim 
\int_{J}|\psi_{J}^{b_{2}}(z)|\frac{|z-c(J_{p})|_{\infty }^{\delta }}{|t-z|_{\infty }^{n+\delta }}F_{K}(t,z,c(J_{p}))dz,
\end{align*}
with 
$
F_{K}(t,z,c(J_{p}))=L(|t-c(J_{p})|_{\infty })S(|z-c(J_{p})|_{\infty })D\Big(1+\frac{|t+c(J_{p})|_{\infty }}{1+|t-c(J_{p})|_{\infty }}\Big)
$. 
By Lemma \ref{boundtoFK}, 
$
F_{K}(t,z,c(J_{p}))\lesssim F_{K}(2^{k'}I, J,2^{k'}I) 
$
and so,
\begin{align*}
\Big| \int \psi_{J}^{b_{2}}(z)K(t,z)dz\Big|
&\leq \|\psi_{J}^{b_{2}}\|_{L^{1}(\mathbb R^{n})} \frac{\ell(J)^{\delta }}{2^{k'(n+\delta)}\ell(I)^{n+\delta }}F_{K}(2^{k'}I, J,2^{k'}I)
\\
&
\leq  B_{J,1}^{b_{2}}
|J|^{\frac{1}{2}}\frac{\ell(J)^{\delta }}{2^{k'(n+\delta)}\ell(I)^{n+\delta }}F_{K}(2^{k'}I, J,2^{k'}I).
\end{align*}
using Lemma \ref{psiortho}. With this 
and $\| \tilde\psi_{J}^{b_{2}}\|_{L^{\infty }(\mathbb R^{n})}\lesssim C_{J}^{b_{2}}|\langle b_{2}\rangle_{J}||J|^{-\frac{1}{2}}$,
\begin{align*}
|{\mathcal K}(t,x)|
&\lesssim \hspace{-.2cm}
\sum_{\tiny \begin{array}{l}J\in  {\mathcal D}_{M}\\ x\in J_{p}\subseteq I_{p}\end{array}}
(C_{J}^{b_{2}})^{2}[b_{2}]_{J,1}|\langle b_{2}\rangle_{J}|F_{K}(2^{k'}I, J,2^{k'}I)
\frac{\ell(J)^{\delta }}{2^{k'(n+\delta)}\ell(I)^{n+\delta }}
\\
&\lesssim \frac{1}{2^{k'(n+\delta)}|I|}\sum_{\tiny \begin{array}{l}J\in  {\mathcal D}_{M}\\ x\in J_{p}\subseteq I_{p}\end{array}}
B\! F(I,J)
\Big(\frac{\ell(J)}{\ell(I)}\Big)^{\delta}.
\end{align*}

Since the cubes in the sum satisfy $ x\in J_{p}\in  {\mathcal D}_{M}(I)$, they can be parametrized by their side length as $\ell(J^{r})=2^{-r}\ell(I_{p})$ with 
$0\leq r \leq M+\log \ell(I)$. With this and the bound $B\! F(I,J)\lesssim 1$, 
\begin{align*}
|{\mathcal K}(t,x)|
&\lesssim \frac{1}{2^{k'(n+\delta)}|I_{p}|}\sum_{r\geq 0}
2^{-r\delta}
\lesssim \frac{1}{2^{k'(n+\delta)}|I_{p}|}
\end{align*}
With this estimate we have from \eqref{finallast} and \eqref{ortoT1other}
\begin{align*}
|\langle P_{M}^{*}T(b_{1}\Psi_{k'}),b_{2}f\rangle |
&\lesssim \int \int |b_{1}(t)\Psi_{k'}(t)| |b_{2}(x)f(x)|
\frac{1}{2^{k'(n+\delta)}|I_{p}|}dtdx
\\
&\leq \frac{1}{2^{k'(n+\delta)}|I_{p}|}\| b_{1}
\|_{L^{1}(2^{k'}I)} \|b_{2}f\|_{L^{1}(I)}
\\
&\lesssim  \frac{1}{2^{k'(n+\delta)}|I_{p}|}[b_{1}]_{2^{k'}I,1}|2^{k'}I| B_{I,1}^{b_{2}}
\lesssim \frac{1}{2^{k'\delta}}\langle M_{1}b_{1}\rangle_{I}B_{I,1}^{b_{2}}
\end{align*}
Then, by the choice of $k$, we finally get
\begin{align*}
|B|&\leq \sum_{k'=k}^{\infty }|\langle P_{M}^{*}T(b_{1}\Psi_{k'}),b_{2}f\rangle |
\lesssim \langle M_{1}b_{1}\rangle_{I} B_{I,1}^{b_{2}} 
\sum_{k'=k}^{\infty }\frac{1}{2^{k'\delta}}
<\epsilon
\end{align*}

\bibliographystyle{amsplain}

\begin{thebibliography}{10}

\bibitem{AusBan}
P.~Auscher, \emph{{Real harmonic analysis / lectures by Pascal Auscher with the
  assistance of Lashi Bandara}}, ANU eView. The Australian National University,
  2012.

\bibitem{AHMTT}
P.~Auscher, S.~Hofmann, C.~Muscalu, T.~Tao, and C.~Thiele, \emph{{Carleson
  measures, trees, extrapolation, and T(b) Theorems}}, Publ. Math. \textbf{46}
  (2002), no.~2, 257--325.

\bibitem{BEO}
B.~Bongioanni, E.~Harboure, and O.~Salinas, \emph{Classes of weights related to
  schrodinger operators}, J. Math. Anal. Appl. \textbf{373} (2011), 563--579.

\bibitem{David}
G.~David, \emph{Wavelets and singular integrals on curves and surfaces},
  Lecture notes in Mathematics, no. 1465, Springer-Verlag, Berlin, 1991.

\bibitem{DJ}
G.~David and J.~L. Journ\'e, \emph{{A boundedness criterion for generalized
  Calder\'on-Zygmund operators}}, Ann. of Math. \textbf{120} (1984), 371--397.

\bibitem{DJS}
G.~David, J.~L. Journ\'e, and S.~Semmes, \emph{{Operateurs de
  Calder{\'o}n-Zygmund, fonctions para-accretives et interpolation}}, Rev. Mat.
  Ib. \textbf{1} (1985), 1--56.

\bibitem{Hy}
T.~Hyt{\"o}nen, \emph{{An operator-valued T(b) theorem}}, J. Funct. Anal.
  \textbf{234} (2006), 420--463.

\bibitem{Hy3}
\bysame, \emph{The sharp weighted bound for general calder\'on-zygmund
  operators}, Ann. of Math. (2010).

\bibitem{HyNa}
T.~Hyt{\"o}nen and F.~Nazarov, \emph{{The local T(b) Theorem with rough test
  functions}}, Preprint, arXiv:1206.0907.

\bibitem{ke}
C.~Kenig, \emph{{Harmonic analysis techniques for second order elliptic
  boundary problems}}, vol.~83, Regional Conference Series in Mathematics,
  1994.

\bibitem{Kra}
M.~A. Krasnoselski, \emph{On a theorem of {M}. {R}iesz}, Dokl. Akad. Nauk
  \textbf{131} (1959), 246--248.

\bibitem{MM}
A.~McIntosh and Y.~Meyer, \emph{Alg{\`e}bres d'op{\'e}rateurs d{\'e}finis par
  des int{\'e}grals singuli{\`e}res}, C. R. Acad. Sci. Paris S{\'e}r I Math.
  \textbf{301} (1985), no.~8, 395--397.

\bibitem{NTVTb}
F.~Nazarov, S.~Treil, and A.~Volberg, \emph{{The T(b)-theorem on
  non-homogeneous spaces}}, Acta Math. \textbf{190} (2003), no.~2, 151--239.

\bibitem{V}
P.~Villarroya, \emph{A characterization of compactness for singular integrals},
  Journal de Math\'ematiques Purees et Appliqu\'ees (2015), no.~3, 485--532.

\end{thebibliography}

\end{document}